\documentclass[12pt]{amsart}  	

\usepackage[margin=1in]{geometry}  

\usepackage[all]{xy}						
\usepackage{amssymb}
\usepackage{amscd,latexsym,amsthm,amsfonts,amssymb,amsmath,amsxtra}
\usepackage[mathscr]{eucal}
\usepackage[colorlinks]{hyperref}
\usepackage{mathtools}
\usepackage{tabularx}
\usepackage{tikz-cd}
\usepackage{tensor}

\usepackage{chngcntr}
\numberwithin{equation}{section}
\newcounter{keepeqno}

\hypersetup{linkcolor=black, citecolor=black}

\newcommand{\BA}{{\mathbb {A}}}
\newcommand{\BB}{{\mathbb {B}}}
\newcommand{\BC}{{\mathbb {C}}}
\newcommand{\BD}{{\mathbb {D}}}

\newcommand{\BF}{{\mathbb {F}}}
\newcommand{\BG}{{\mathbb {G}}}
\newcommand{\BH}{{\mathbb {H}}}

\newcommand{\BJ}{{\mathbb {J}}}

\newcommand{\BL}{{\mathbb {L}}}

\newcommand{\BP}{{\mathbb {P}}}

\newcommand{\BT}{{\mathbb {T}}}

\newcommand{\BV}{{\mathbb {V}}}
\newcommand{\BW}{{\mathbb {W}}}
\newcommand{\BX}{{\mathbb {X}}}
\newcommand{\BY}{{\mathbb {Y}}}
\newcommand{\BZ}{{\mathbb {Z}}}

\newcommand{\Bw}{{\mathbf {w}}}

\newcommand{\CC}{{\mathcal {C}}}

\newcommand{\CF}{{\mathcal {F}}}
\newcommand{\CG}{{\mathcal {G}}}
\newcommand{\CH}{{\mathcal {H}}}
\newcommand{\CI}{{\mathcal {I}}}

\newcommand{\CL}{{\mathcal {L}}}

\newcommand{\CO}{{\mathcal {O}}}
\newcommand{\CP}{{\mathcal {P}}}

\newcommand{\CS}{{\mathcal {S}}}
\newcommand{\CT}{{\mathcal {T}}}

\newcommand{\CX}{{\mathcal {X}}}

\newcommand{\FC}{{\mathfrak {C}}}

\newcommand{\FL}{{\mathfrak {L}}}

\newcommand{\FS}{{\mathfrak {S}}}

\newcommand{\FV}{{\mathfrak {V}}}

\newcommand{\FX}{{\mathfrak {X}}}

\newcommand{\Fc}{{\mathfrak {c}}}

\newcommand{\Fg}{{\mathfrak {g}}}

\newcommand{\Fl}{{\mathfrak {l}}}

\newcommand{\Fo}{{\mathfrak {o}}}
\newcommand{\Fp}{{\mathfrak {p}}}

\newcommand{\Fs}{{\mathfrak {s}}}

\newcommand{\RA}{{\mathrm {A}}}
\newcommand{\RB}{{\mathrm {B}}}
\newcommand{\RC}{{\mathrm {C}}}
\newcommand{\RD}{{\mathrm {D}}}

\newcommand{\RF}{{\mathrm {F}}}
\newcommand{\RG}{{\mathrm {G}}}
\newcommand{\RH}{{\mathrm {H}}}
\newcommand{\RI}{{\mathrm {I}}}

\newcommand{\RK}{{\mathrm {K}}}
\newcommand{\RL}{{\mathrm {L}}}

\newcommand{\RN}{{\mathrm {N}}}
\newcommand{\RO}{{\mathrm {O}}}
\newcommand{\RP}{{\mathrm {P}}}

\newcommand{\RS}{{\mathrm {S}}}
\newcommand{\RT}{{\mathrm {T}}}
\newcommand{\RU}{{\mathrm {U}}}
\newcommand{\RV}{{\mathrm {V}}}
\newcommand{\RW}{{\mathrm {W}}}
\newcommand{\RX}{{\mathrm {X}}}

\newcommand{\RZ}{{\mathrm {Z}}}
\newcommand{\Sat}{{\mathrm {Sat}}}

\newcommand{\Ad}{{\mathrm{Ad}}}

\newcommand{\disc}{{\mathrm{disc}}}

\newcommand{\End}{{\mathrm{End}}}

\newcommand{\GL}{{\mathrm{GL}}}

\newcommand{\Orb}{{\mathrm{Orb}}}

\newcommand{\Hom}{{\mathrm{Hom}}}

\newcommand{\Id}{{\mathrm{Id}}}

\newcommand{\Mat}{{\mathrm{Mat}}}

\newcommand{\Mp}{{\mathrm{Mp}}}

\newcommand{\PGL}{{\mathrm{PGL}}}

\renewcommand{\Re}{{\mathrm{Re}}}

\newcommand{\Res}{{\mathrm{Res}}}
\newcommand{\res}{{\mathrm{res}}}

\newcommand{\SL}{{\mathrm{SL}}}

\newcommand{\Spec}{{\mathrm{Spec}}}
\newcommand{\SO}{{\mathrm{SO}}}

\newcommand{\Sym}{{\mathrm{Sym}}}

\newcommand{\Sp}{{\mathrm{Sp}}}

\newcommand{\st}{{\mathrm{st}}}

\newcommand{\tr}{{\mathrm{tr}}}

\newcommand{\ud}{\,\mathrm{d}}
\newcommand{\vol}{{\mathrm{vol}}}

\def\ac{\mathrm{ac}}

\def\diag{{\rm diag}}

\def\std{\rm Std}

\def\val{\mathrm{val}}
\newcommand{\sslash}{\mathbin{/\mkern-6mu/}}

\newtheorem{thm}{Theorem}[section]
\newtheorem{dfn}[thm]{Definition}
\newtheorem{rmk}[thm]{Remark}
\newtheorem{exm}[thm]{Example}
\newtheorem{prp}[thm]{Proposition}
\newtheorem{lem}[thm]{Lemma}
\newtheorem{cor}[thm]{Corollary}

\newtheorem{cnj}[thm]{Conjecture}

\makeatletter

\newcommand{\Rmnum}[1]{\expandafter\@slowromancap\romannumeral #1@}
\makeatother

\begin{document}

    \title[]{Fundamental Lemma for Rank One Spherical Varieties of Classical Types}

    \date{\today}
    
    \author{Zhaolin Li}
    
    \address{School of Mathematics, University of Minnesota, 206 Church St. S.E., Minneapolis, MN 55455, USA.
}
    \email{li001870@umn.edu}
    \date{\today}

\begin{abstract}
   According to the relative Langlands functoriality conjecture, an admissible morphism between the $L$-groups of spherical varieties should induce a functorial transfer of the corresponding local and global automorphic spectra. Via the relative trace formula approach, two basic problems are the local transfer and the fundamental lemma on the geometric side of the relative trace formulae. In this paper, we consider the rank one spherical variety case, where the admissible morphism between the $L$-groups is the identity morphism, in which case, Y. Sakellaridis has already established the local transfer (\cite{Sak21}). We formulate the statement of the fundamental lemma for the general rank one spherical variety case and prove the fundamental lemma for the rank one spherical varieties of classical types. 
\end{abstract}

    \maketitle
    \tableofcontents

\section{Introduction}

The origins of the relative trace formula can be traced back to the early 1930s. In his seminal work, H. Petersson expressed sums of products of Fourier coefficients of holomorphic forms in terms of Bessel functions and Kloosterman sums in \cite{Pet32}. Decades later, in 1980, N.V. Kuznetsov established a similar formula for Maass forms in \cite{Kuz80}. This Petersson-Kuznetsov trace formula is now recognized as an early example of the relative trace formula.

To study periods of automorphic forms, H. Jacquet, S. Rallis, and their collaborators developed the framework of the relative trace formula. The central idea is to compare geometric sides of two relative trace formulae in order to establish relationships between different automorphic periods, rather than analyzing a single relative trace formula in isolation. A notable success of this method is Jacquet's reproof of Waldspurger's celebrated formula by comparing the relative trace formula associated with a non-split torus and a split torus in $\PGL_2$, see \cite{Jac86, Sak13, Sak19}.

In the theory of automorphic forms, another central problem is the functoriality conjecture, which aims to establish a correspondence between automorphic representations of two reductive groups $\RG_1$ and $\RG_2$, given and $L$-homomorphism between their respective $L$-groups $^L\RG_1\rightarrow{^L\RG_2}$.

Through the influential work of H. Jacquet, S. Rallis, A. Venkatesh, Y. Sakellaridis, and others, it has become increasingly evident that the functoriality conjecture admits a natural generalization to a broader class of $\RG$-spaces known as spherical varieties, which include symmetric spaces as a prominent subclass. Notably, the classical group case can be interpreted as a special instance of a symmetric space, wherein the group $\RG\times\RG$ acts on $\RG$ via left and right multiplication. Spherical varieties are expected to possess associated $L$-groups, and according to \cite{SV17}, any morphism of $L$-groups $^L\RG_{\RX_1}\rightarrow{^L\RG_{\RX_2}}$ between the $L$-groups of two spherical varieties $\RX_1$ and $\RX_2$ should induce relations between their automorphic spectra.

In this context, the relative trace formula is envisioned as a distribution on the quotient stack $\CX:=[\RX\times\RX/\RG]$, where $\RX$ is a $\RG$-space, equipped with both geometric and spectral expansions, see \cite{Sak16}. The overarching goal is to deduce relations between the spectral sides of two such relative trace formulas by comparing their corresponding geometric sides.

In \cite{Sak21}, Y. Sakellaridis proposed a new form of comparison for spherical varieties of rank one, that is, in cases where the associated $L$-groups satisfy $^L\RG_{\RX}=\SL_2$ or $\PGL_2$. The central idea is to compare the relative trace formula associated with an arbitrary rank one spherical variety with the corresponding Kuznetsov trace formula, which is conjectured to represent a fundamental case in this broader framework of comparisons. As far as current understanding goes, there are principally two reasons for favoring the Kuznetsov trace formula over the classical Arthur-Selberg trace formula in this context. First, given a spherical $\RG$-variety $\RX$, the spectral side of the Kuznetsov trace formula is weighted by 

\[
\frac{1}{L(\varphi,\Ad,1)}
\]
while the relative trace formula will have an additional factor
\[
\frac{L_{\RX}(\varphi)}{L(\varphi,\Ad,1)},
\]
where $\varphi$ is a global Arthur parameter. In the Arthur-Selberg trace formula case, we have $L_{\RX}(\varphi)=L(\varphi,\Ad,1)$. The second reason arises from geometric considerations, which lie somewhat beyond the scope of this paper; we refer the interested reader to \cite{Sak23} for further details.

We now outline the strategy for approaching the functoriality conjecture for rank-one spherical varieties via trace formulae. For simplicity, we assume throughout that all groups and varieties are split over the global field under consideration. Let $\RX$ be a $\RG$-spherical variety, $^{L}\RG_{\RX}=\SL_2$ or $\PGL_2$, and write $\RG_{\RX}=\PGL_2$ or $\SL_2$ respectively.

Locally, let $\CS(\RX\times\RX/\RG)$ be the push-forward of Schwartz measures $\CS(\RX\times\RX)$ to $\RX\times\RX\sslash\RG\cong \BA^1$ along the canonical map fixed in \cite{Sak21}, $\BL_{\RX}^{\circ}$ be the push-forward image of the probability measure, which is equal to the characteristic function on $\RX(\Fo_{\RF})\times\RX(\Fo_{\RF})$ times an invariant measure. On the Kuznetsov trace formula side, let $\CS^-_{L_{\RX}}(\RN_{\psi}\backslash\RG_{\RX}/\RN,\psi)$ be an enlarged space of test measures, and $f_{L_{\RX}}$ be the non-standard basic vector determined by $L_{\RX}$, that we will explain in Section \ref{ssec_localktf}. Due to \cite{Sak21}, there is a linear isomorphism 
\begin{align*}
	\CT:\CS^-_{L_{\RX}}(\RN_{\psi}\backslash\RG_{\RX}/\RN,\psi)\rightarrow \CS(\RX\times\RX/\RG),
\end{align*}
with an explicit formula given by a composition of classical Fourier transforms with certain simple correction factors. In this paper, we prove the fundamental lemma when $\RG$ is of classical type, that is, the inverse of $\CT$ takes $\BL_{\RX}^{\circ}$ to an explicit multiple of $f_{L_{\RX}}$, and is compatible with the action of Hecke algebras. 

Our approach is purely local, extending the methodology introduced by S. Rallis \cite{Ral82}, and later in \cite{GW21}. After establishing the comparison between the transfer operators in \cite{Sak21} and those arising from the Weil representation (see Section \ref{sec_trWeil}), the fundamental lemma for the unit element allows readily for types $A_n$, $C_n$, and $D_n$. The case of $B_n$ is more nuanced, as the dual pair involves the metaplectic cover of $\SL_2$, precluding the direct identification possible in the other cases. Nevertheless, a manual matching of the orbital integrals remains feasible, as we demonstrate in Section \ref{ssec_bunit}. Regarding the general Hecke elements, S. Rallis \cite{Ral82} demonstrated that within the framework of Weil representation, the Hecke algebras of the even orthogonal and symplectic groups act compatibly. We verify that this compatibility is consistent with the Hecke algebra homomorphism arising from the map between the universal Cartan of $\RG$ and $\RX$, as we detail in Section \ref{intro_tfandfl}. This correspondence ensures that the matching of the unit elements extends to the full Hecke algebra, thereby establishing the fundamental lemma for $\RX=\SO_{2n-1}\backslash\SO_{2n}$. We further generalize this technique to all rank one spherical varieties of classical types. A significant technical challenge arises when $\RX=\SO_{2n}\backslash\SO_{2n+1}$, again, as the dual pair setup involves the metaplectic cover $\Mp_2$ of $\SL_2$. Here, we utilize a comparison between $\RA\backslash\PGL_2/(\RN,\psi)$ and $(\RN,\psi)\backslash\Mp_2/(\RN,\psi)$ in the foundational work of H. Jacquet \cite{Jac87}. Since the functions on $\Mp_2$ arising from the Weil representation are non-standard, we enlarge the space of test measures in Jacquet's comparison. Specifically, to ensure the $L$-values match, we insert the factor $L\left(\st,n-\frac{1}{2}\right)$, a procedure analogous to the one employed by Y. Sakellaridis in \cite{Sak13}. And the comparison between $\RA\backslash\PGL_2/(\RN,\psi)$ and $(\RN,\psi)\backslash\PGL_2/(\RN,\psi)$ with non-standard test measures follows from the unfolding procedure, as detailed in \cite{SV17}.

Our methods and results are strictly local in nature, but they are motivated by the global problem of establishing relative functoriality for rank one spherical varieties via relative trace formulae. This functoriality is predicated on the canonical map between dual groups $\RG_{\RX}^{\vee}\rightarrow\RG^{\vee}$, which, according to the relative Langlands functoriality conjecture, should induce a functorial transfer of the corresponding automorphic spectra. Globally, the relative trace formula should be viewed as a distribution
\begin{align*}
	\mathrm{RTF}_{\FX}:\CS(\RX\times\RX/\RG)\rightarrow\BC.
\end{align*}
The geometric expansion is given as the evaluation map in \cite{Sak16}. We still need the spectral expansion. On the Kuznetsov trace formula side, we have to extend the Kuznetsov trace formula to these enlarged test measures
\begin{align*}
	\mathrm{KTF}:\CS^-_{L_{\RX}}(\RN_{\psi}\backslash\RG_{\RX}/\RN,\psi)\rightarrow\BC
\end{align*}
as in \cite{Li25}. Subsequently, local transfer operators are employed to facilitate the formulation of a Poisson summation formula, which is instrumental in comparing the geometric sides of the relative trace formulas.
\begin{align*}
	\xymatrix{
\CS^-_{L_{\RX}}(\RN_{\psi}\backslash\RG_{\RX}/\RN,\psi) \ar@{->}[rd]^{\mathrm{KTF}} \ar[rr]^{\CT}&& \CS(\RX\times\RX/\RG),\ar@{->}[ld]^{\mathrm{RTF}_{\CX}}\\
 &\BC & 
}
\end{align*}
see \cite[Section 4]{Sak19}.

For each $\RX$-distinguished cuspidal representation $\pi$ of $\RG$, let $\mathrm{RTF}(\pi)$ be the $\pi$-component in the spectral expansion. For every cuspidal representation $\sigma$ of $\RG_{\RX}$, let $\mathrm{KTF}(\sigma)$ be the $\sigma$-component in the spectral expansion of $\mathrm{KTF}$. Let $\pi_0$ be an $\RX$-distinguished cuspidal representation of $\RG$; in principle, we will be led to the following identity
\[
\sum_{\pi}\mathrm{RTF}(\pi)=\sum_{\sigma}\mathrm{KTF}(\sigma),
\]
where the summation on the left-hand-side is over $\RX$-distinguished cuspidal representations that are nearly isomorphic to $\pi_0$, that is, $\pi_{\nu}\cong\pi_{0,\nu}$ for all but finitely many places, and the summation on the right-hand-side is over all the cuspidal representations $\sigma$ such that the Satake parameter of $\sigma_{\nu}$ is mapped to $\pi_{\nu}$ at each unramified place under the canonical map
\begin{align*}
	\RG_{\RX}^{\vee}\rightarrow\RG^{\vee}.
\end{align*}

\subsection{Relative trace formulae and fundamental lemmas}\label{intro_tfandfl}
We first give an overview of the relative trace formula. Let $k$ be a global field, $\BA$ its ring of adeles, $\RG$ a split reductive group over $k$, and $\RX$ a smooth affine spherical variety over $k$. For the purpose of this paper, we may take
\begin{align*}
	\CF(\RX(\BA)):=\otimes^{\prime}_{\nu}\CF(\RX(k_{\nu})),
\end{align*}
where the tensor product is over all the places of $k$ with respect to the characteristic function $1_{\RX(\Fo_{\RF})}$ of $\RX(\Fo_{\nu})$. The $\RX$-theta series is the following functional on $\CF(\RX(\BA))$:
\begin{align*}
	\Phi\mapsto\theta(\Phi,g):=\sum_{\gamma\in\RX(k)}\Phi(\gamma g)\in\CC^{\infty}([\RG]),
\end{align*}
where the sum is absolutely convergent since $\RX(k)$ is discrete in $\RX(\BA)$ as $\RX$ is affine.

Roughly speaking, the relative trace formula is the Plancherel formula for $\CL^2([\RG])$, applied to the inner product of two theta series:
\begin{align*}
	\mathrm{RTF}_{\FX}:\CF(\RX(\BA))\otimes\CF(\RX(\BA))\ni \Phi_1\otimes\Phi_2\mapsto \int_{[\RG]}^*\theta(\Phi_1,g)\theta(\Phi_2,g)\ud g.
\end{align*}
This inner product does not converge in general and needs to be regularized. In the case that $\RX$ is smooth affine and $\RG$ is split, there is a way to regularize it as described in \cite{Sak16}.

As for the Kuznetesov trace formula, one can define $\mathrm{KTF}$ similarly, but with non-standard spaces of test measures, which we will describe in Section \ref{ssec_localktf}. 

For the purpose of this paper, let us consider the case related to rank one spherical varieties. The list of spherical varieties of rank $1$ consists of a finite number of families, classified by Ahiezer \cite{Ahi83}, also see \cite{KV06} and \cite{Was96}. Up to an action of the center, the affine homogeneous spherical varieties over an algebraically closed field in characteristic $0$ are listed in the following table :
\begin{align*}
	\begin{tabular}{ | m{1cm} | m{3.4cm}| m{1cm} | m{5.5cm}| } 
 \hline
 & $\RX$ & $\RG_{\RX}^{\vee}$ & $L_{\RX}$ \\
 \hline
   $A_n$&   $\GL_n\backslash \PGL_{n+1}$ & $\SL_2$ & $L\left(\std,\frac{n}{2}\right)^2$  \\
\hline
    $B_n$&   $\SO_{2n}\backslash\SO_{2n+1}$ &$\SL_2$  & $L\left(\std,n-\frac{1}{2}\right)L\left(\std,\frac{1}{2}\right)$ \\
\hline
   $C_n$&   $\Sp_{2n-2}\times\Sp_2\backslash\Sp_{2n}$  & $\SL_2$  &$L\left(\std,n-\frac{1}{2}\right)L\left(\std,n-\frac{3}{2}\right)$  \\
\hline
   $D_n$&   $\SO_{2n-1}\backslash\SO_{2n}$  & $\PGL_2$ & $L(\Ad,n-1)$ \\
\hline
   $B_3^{''}$ &   $\RG_2\backslash\mathrm{Spin}_7$ & $\PGL_2$  & $L(\Ad,3)$ \\
\hline
 $D_4^{''}$&   $\mathrm{Spin}_7\backslash\mathrm{Spin}_8$  & $\PGL_2$ & $L(\Ad,n-1)$ \\
\hline
 $F_4$&   $\mathrm{Spin_9}\backslash\RF_4$  & $\SL_2$ & $L\left(\std,\frac{11}{2}\right)L\left(\std,\frac{5}{2}\right)$ \\
\hline
 $G_2$&   $\SL_3\backslash\RG_2$  & $\SL_2$ & $L\left(\std,\frac{5}{2}\right)L\left(\std,\frac{1}{2}\right)$ \\
\hline
\end{tabular}.
\end{align*}
In this paper, we will only consider the split cases. For each spherical variety, $L_{\RX}$ stands for a $\displaystyle{\frac{1}{2}\BZ}$-graded representation $\displaystyle{\rho=\bigoplus_{n\in\frac{1}{2}\BZ} \rho_n }$ of its dual group $\RG_{\RX}^{\vee}$, which is called the $L$-value asssociated to $\RX$, thinking of the $L$-value
\[
\prod_{n}L(\rho_n\circ\varphi,n)
\]
attached to any Langlands parameter $\varphi$ into $\RG_{\RX}^{\vee}$. And the $L$-value here is determined in terms of the geometry of $\RX$. For each $L_{\RX}$, we define a space of non-standard test measures in Section \ref{ssec_localktf}, together with basic vectors 
\begin{align*}
	f_{L_{\RX}}:=	\left\{ \begin{array}{rcl}
f_{L\left(\std,s_1\right)L\left(\std,s_2\right)}  & \mbox{for}&
\RG_{\RX}^{\vee}=\SL_2;\\ f_{L\left(\Ad,s_0\right)}  & \mbox{for} &\RG_{\RX}^{\vee}=\PGL_2.
\end{array}\right.
\end{align*}
Here $s_1,s_2$ and $s_0$ are also listed as in the above table, depending on $\RX$. We take \[\CF_{L_{\RX}}(\RN(k_{\nu}),\psi_{\nu}\backslash\RG_{\RX}(k_{\nu}))\] to be the Schwartz sections of certain bundle over $\overline{\RN\backslash\RG_{\RX}}$, together with basic vector
\begin{align*}
	\Phi_{L_{\RX}}:=\left\{ \begin{array}{rcl}
\Phi_{L\left(\std,s_1\right)L\left(\std,s_2\right)}  & \mbox{for}&
\RG_{\RX}^{\vee}=\SL_2;\\ \Phi_{L\left(\Ad,s_0\right)}  & \mbox{for} &\RG_{\RX}^{\vee}=\PGL_2.
\end{array}\right.
\end{align*}
Here again $\Phi_{L_{\RX}}$ is defined in Section \ref{ssec_localktf}. Then we have the Poincar\'{e} series
\begin{align*}
	\CP(\Phi,g):=\sum_{\gamma\in\RN(k)\backslash\RG_{\RX}(k)}\Phi(\gamma g)
\end{align*}
for $\Phi\in\CF_{\RX}(\RN(\BA),\psi\backslash\RG_{\RX}(\BA)):=\otimes^{\prime}_{\nu}\CF_{\RX}(\RN(k_{\nu}),\psi_{\nu}\backslash\RG_{\RX}(k_{\nu}))$. Let $\CF(\RN(\BA),\psi\backslash\RG(\BA))$ be the space of standard test functions, then the Kuznetsov trace formula is the decomposition of
\begin{align*}
	\mathrm{KTF}:\CF_{\RX}(\RN(\BA),\psi\backslash\RG_{\RX}(\BA))\otimes\CF(\RN(\BA),\psi\backslash\RG_{\RX}(\BA))\ni\Phi_1\otimes\Phi_2\mapsto \int^*_{[\RG_{\RX}]}\CP(\Phi_1,g)\CP(\Phi_2,g)\ud g.
\end{align*}
Of course, the integral here also needs to be regularized.

The ultimate objective of this framework is to establish a rigorous global comparison between $\mathrm{RTF}_{\FX}$ and $\mathrm{KTF}$. Locally, the smooth transfer has already been established in \cite{Sak21}. Let $\pi:\RX\times\RX\rightarrow\RX\times\RX\sslash\RG$ be the canonical quotient map to the affine, invariant-theoretic quotient $\Spec\;\RF[\RX\times\RX]^{\RG}$, and the image $\pi_*\CS(\RX\times\RX)$ of the push-forward measure will be denoted by $\CS(\RX\times\RX/\RG)$. Y. Sakellaridis proved the following theorem in \cite{Sak21}:
\begin{thm}[Theorem 1.3.1 in \cite{Sak21}]\label{thm_TrSak}
	Let $\FC_{\RX}=(\RX\times\RX)\sslash\RG$. There is an isomorphism $\FC_{\RX}\cong\BA^1$, and the map $\RX\times\RX\rightarrow\BA^1$ is smooth away from the preimage of two points of $\BA^1$, which we will call singular. We fix the isomorphism as follows:
	\begin{itemize}
		\item When $\RG_{\RX}^{\vee}=\SL_2$, we take the set of singular points to be $\{0,1\}$, with $\RX^{\diag}\subset\RX\times\RX$ mapping to $1\in\FC_{\RX}\cong\BA^1$.
		\item When $\RG_{\RX}^{\vee}=\PGL_2$, we take the set of singular points to be $\{-2,2\}$, with $\RX^{\diag}\subset\RX\times\RX$ mapping to $2\in\FC_{\RX}\cong\BA^1$.
	\end{itemize}
	Then there is a continuous linear isomorphism:
	\[
	\CT:\CS_{L_{\RX}}^-(\RN,\psi\backslash\RG_{\RX}/\RN,\psi)\rightarrow\CS(\RX\times\RX/\RG),
	\]
	given by the following formula:
	\begin{itemize}
		\item When $\RG_{\RX}^{\vee}=\SL_2$ with $L_{\RX}=L(\std,s_1)L(\std,s_2)$, $s_1\geq s_2$,
		\[
		\CT f(\xi)=|\xi|^{s_1-\frac{1}{2}}\left(  |\cdot|^{\frac{1}{2}-s_1}\psi(\cdot)\ud \cdot   \right)\star \left(  |\cdot|^{\frac{1}{2}-s_2}\psi(\cdot)\ud \cdot   \right)\star f(\xi).
		\]
		\item When $\RG_{\RX}^{\vee}=\PGL_2$ with $L_{\RX}=L(\Ad,s_0)$,
		\[
		\CT f(\zeta)=|\zeta|^{s_0-1}\left(  |\cdot|^{1-s_0}\psi(\cdot)\ud \cdot   \right)\star f(\zeta).
		\]
		\end{itemize}
\end{thm}

\begin{rmk}\label{lem_Tr_Sak}
	Write $f(\xi)=\phi(\xi)\ud \xi$ or $\phi(\zeta)\ud\zeta$, in terms of additive usual Fourier transforms,
	\begin{itemize}
	\item When $\RG_{\RX}^{\vee}=\SL_2$ with $L_{\RX}=L(\std,s_1)L(\std,s_2)$, $s_1\geq s_2$,
		\[
		\CT  f(\xi)= \left(\BF_{\psi^{-1}}\circ\iota\circ|\cdot|^{1+s_1-s_2}\circ\BF_{\psi^{-1}} \circ\iota\circ|\cdot|^{s_2+\frac{1}{2}} \right)(\phi)(\xi)\ud\xi  .
		\]
		\item When $\RG_{\RX}^{\vee}=\PGL_2$ with $L_{\RX}=L(\Ad,s_0)$,
		\[
		\CT f(\zeta)= (\BF_{\psi^{-1}}\circ\iota\circ|\cdot|^{s_0})(\phi)(\zeta)\ud\zeta .
		\]
	\end{itemize}
    For the notations, see Section \ref{ssec_notation}.
\end{rmk}

To achieve the relative Langlands functoriality via comparing the relative trace formula for $\RX$ and the Kuznetsov trace formula, we also need the appropriate fundamental lemma, and the first main result of this paper is the following fundamental lemma, which is a combination of Theorem \ref{thm_d}, Theorem \ref{thm_a}, Theorem \ref{thm_c}, and Theorem \ref{thm_b}:
\begin{thm}\label{thm_main}
		Let $\RX$ be a rank one spherical variety and $\FC_{\RX}=(\RX\times\RX)\sslash\RG$. Fix the isomorphism $\FC_{\RX}\cong\BA^1$ as in Theorem \ref{thm_TrSak}, and under the transfer operator
	\[
	\CT:\CS_{L_{\RX}}^-(\RN,\psi\backslash\RG_{\RX}/\RN,\psi)\rightarrow\CS(\RX\times\RX/\RG)
	\]
there,
\begin{align*}
	\CT^{-1}(\BL_{\RX}^{\circ})=\frac{q^{\dim\RX}}{\#\RX(\kappa_{\RF})\zeta_{\RF}(2)\zeta_{\RF}(s_{\RX})}f_{L_{\RX}}.
\end{align*}
where \begin{align*}
	s_{\RX}=\left\{ \begin{array}{rcl}
s_1+s_2=\frac{\dim \RX}{2}    & \mbox{for} & \RG_{\RX}^{\vee}=\SL_2; \\
2s_0=\dim{\RX}-1 & \mbox{for} & \RG_{\RX}^{\vee}=\PGL_2.
\end{array}\right.
\end{align*}
Here $\BL_{\RX}^{\circ}$ is the push-forward image of the probability measure, which is equal to the characteristic function on $\RX(\Fo_{\RF})\times\RX(\Fo_{\RF})$ times an invariant measure.\end{thm}

\begin{cnj}\label{cnj_un}
	Theorem \ref{thm_main} holds for all rank one spherical varieties.
\end{cnj}

\begin{rmk}
The author is grateful to Yi Shan for noting that Conjecture \ref{cnj_un} for $B_3^{''}$, $D_4^{''}$, and $G_2$ follows from isomorphisms	
\begin{align*}
		\RG_2 \backslash \mathrm{Spin}_7\cong \mathrm{Spin}_7\backslash\mathrm{Spin}_8\cong  \SO_7\backslash\SO_8,\;\mathrm{and}\;\SL_3\backslash\RG_2\cong\SO_6\backslash\SO_7,
	\end{align*}
effectively reducing them to the classical cases established herein. The remaining case of $F_4$ is treated in the recent preprint \cite{LW25}.
\end{rmk}

\begin{rmk}
	If we write $\tau(\RX)$ for the local Tamagawa measure of $\RX$, then we have
	\begin{align*}
		\tau(\SL_2)=\tau(\PGL_2)=1-q^{-2}=\frac{1}{\zeta_{\RF}(2)},
	\end{align*}
	hence we can rewrite the above formula as
	\begin{align*}
	\CT^{-1}(\BL_{\RX}^{\circ})=\frac{\tau(\RG_{\RX})}{\tau(\RX)\zeta_{\RF}(s_{\RX})}f_{L_{\RX}}.
	\end{align*}
\end{rmk}

However, to globally implement the relative trace formula for establishing relative functoriality and the correlation between $\RX$-periods of automorphic forms and $L$-values, the fundamental lemma for the full Hecke algebra is required to effectively isolate the spectral contributions.

 Fix a Borel subgroup $\RB\subset\RG$, the maximal torus $\RA\subset\RB$. Fix a hyperspecial vertex in the Bruhat-Tits building of $\RG$ that lies in the apartment determined by $\RA$. It determines a hyperspecial maximal open compact subgroup $\RK$ of $\RG$. Let $\RX^{\circ}\subset\RX$ be its open orbit under the $\RB$-action. Let $\RP(\RX):=\{g\in\RG\mid \RX^{\circ}g=\RX^{\circ}   \}$. Let $\RN$ be the unipotent radical of $\RB$, then the quotient $\RX^{\circ}\sslash\RN$ is a homogeneous space under the action of $\RA$, whose action factors through the faithful action of a quotient $\RA\twoheadrightarrow\RA_{\RX}$, the universal Cartan of $\RX$. It is known that $\RA_{\RX}$ is a quotient of $\RP(\RX)$ through the Levi $\RL(\RX)$:
\[
\RP(\RX)\rightarrow\RL(\RX)\rightarrow\RA_{\RX}.
\]
This gives rise to a map with finite kernel between the complex dual tori
\begin{align}\label{map_tori}
\RA_{\RX}^{\vee}\rightarrow\RA^{\vee}.
\end{align}
While this map facilitates the study of the unramified spectrum of $\RX$, it necessitates a modification as formulated in \cite{Sak18}. Let $\delta_{(\RX)}^{\frac{1}{2}}$ be the square root of the modular character, which is defined as the quotient of the right by the left Haar measure of the group $\RB\cap\RL(\RX)$, considered as an unramified character of $\RB$ and hence as an element of $\RT^{\vee}$. It is stable under the little Weyl group $\BW_{\RX}$ of $\RX$, and we consider the normalized morphism
\begin{align}\label{map_tori_crt}
\RA_{\RX}^{\vee}\rightarrow\RA^{\vee}:\widetilde{\chi}\mapsto \chi\delta_{(\RX)}^{\frac{1}{2}},
\end{align}
where $\chi$ is the image of $\widetilde{\chi}$ under (\ref{map_tori}). Consider the spherical Hecke algebra $\CH(\RG,\RK)$ of $\RG$. By restriction via the above normalized morphism between tori, we get a morphism of algebras
\[
\lambda_{\RX}:\CH(\RG,\RK)\cong \BC[\RA^{\vee}]^{\BW_{\RG}}\rightarrow\BC[\RA_{\RX}^{\vee}]^{\BW_{\RX}}\cong \CH(\RG_{\RX},\RG_{\RX}(\Fo_{\RF})).
\]

\begin{thm}\label{thm_flhk}
	Let $\RX$ be a rank one spherical variety of classical type and $\FC_{\RX}=(\RX\times\RX)\sslash\RG$. Fix the isomorphism $\FC_{\RX}\cong\BA^1$ as in Theorem \ref{thm_TrSak}, and under the transfer operator
	\[
	\CT:\CS_{L_{\RX}}^-(\RN,\psi\backslash\RG_{\RX}/\RN,\psi)\rightarrow\CS(\RX\times\RX/\RG)
	\]
there,\begin{align*}
	\CT^{-1}(f\star \BL_{\RX}^{\circ} )=\lambda_{\RX}(f)\star\frac{q^{\dim\RX}}{\#\RX(\kappa_{\RF})\zeta_{\RF}(2)\zeta_{\RF}(s_{\RX})}f_{L_{\RX}}.
\end{align*}

\end{thm}

\begin{cnj}\label{cnj_flhk}
	Theorem \ref{thm_flhk} for all rank one spherical varieties.
\end{cnj}

\begin{rmk}\label{rmk_notations}
	Before moving on, let us clarify the notations. First, note that the Hecke algebra does not act on the space of push-forward of measures. However, the Bernstein center does, where the action is on the first copy, and the Bernstein center for the component of the spectrum corresponding to the unramified principal series is isomorphic to the Hecke algebra under the natural map. Of course, there is a more directed description of this action. Let us follow the conversion used in \cite{Sak13}. Let $\Phi_1$ be a Schwartz function on $\RX$ that is fixed by $\RK$ and $h \in\CH(\RG,\RK)$, consider the right convolution, which is the push-forward under the action map
	\[
	h \star \Phi(x):=\int_{\RG}\Phi(xg^{-1})h(g)\ud g.
	\]
	Then for $\Phi_1\otimes\Phi_2$, we interpret the action of $h$ on its push-forward measure as the image of $(h\star\Phi_1)\otimes\Phi_2$. If one wants to consider the action on the second copy, the action of the Bernstein center is also the image of the push-forward of $\Phi_1\otimes ( h^{\vee}\star \Phi_2)$.
	
	When $\RX=\RH\backslash\RG$, under the various canonical identifications
	\[
	\left(\RH\backslash\RG\times\RH\backslash\RG\right)\sslash \RG= \RG\rotatebox{60}{$\sslash$}\left(\RG/\RH\times\RG/\RH\right)=\left(\RH\backslash\RG\right)\sslash\RH=\RH\rotatebox{60}{$\sslash$}(\RG/\RH),
	\]
	One can check that the Hecke algebra action is realized either as the $\star$ action on $\RH\backslash\RG$ or the following right action of $f\in\CH(\RG,\RK)$ on $\RK$-invariant functions on $\RG/\RH$ defined by
	\[
	h \star \Phi(x):=\int_{\RG}h(g)\Phi(gx)\ud x,\;\Phi\in\CF(\RG/\RH)^{\RK}.
	\]

\end{rmk}

\begin{rmk}
In view of the subsequent analysis involving Weil representations, we adopt the left $\RG$-action on $\RG/\RH$ for all cases. Consequently, to ensure consistency with (\ref{map_tori_crt}), the modular character $\delta_{(\RX)}$ is replaced by its inverse, $\delta_{(\RX)}^{-1}$, for the homomorphism between Hecke algebras.

\end{rmk}

\subsection{The organization of the paper}

We begin by reviewing the necessary background in Section \ref{sec_pre}. Section \ref{ssec_localktf} focuses on the non-standard test measures for the Kuznetsov trace formula on $\SL_2$ and $\PGL_2$, and Section \ref{ssec_weil} provides the requisite background on the Weil representations employed in our proofs. 

In Section \ref{sec_trWeil}, we described the construction of certain transfer operators via Weil representations, following \cite{GS12}. These operators are closely related to those studied in \cite{Sak21}. Once the relationship between the transfer operators arising from Weil representations and those defined in \cite{Sak21} is established, the fundamental lemma for the unit elements follows from an explicit computation. Furthermore, the extension to the full Hecke algebra is obtained by demonstrating the compatibility of the Hecke actions, as predicated on the methods of S. Rallis \cite{Ral82}.

The ensuing sections, Section \ref{sec_d}, Section \ref{sec_a}, Section \ref{sec_c}, and Section \ref{sec_b}, are devoted to the case-by-case analysis required to establish the fundamental lemma across all classical types.

\subsection{Notations}\label{ssec_notation}
\begin{itemize}
    \item For a finite set $\RS$, we will use $\#\RS$ to denote the cardinality of $\RS$, and $1_{\RS}$ the characteristic function of $\RS$.
    \item Let $\RF$ be a non-Archimedean local field of characteristic zero, $\Fo_{\RF}$ its ring of integers, and $\Fp_{\RF}$ its maximal ideal. Let $q$ be the cardinality of the residue field $\kappa_{\RF}:=\Fo_{\RF}/\Fp_{\RF}$. Assume $q$ is odd. Fix a uniformizer $\varpi_{\RF}$ and for any $x\in\RF$, denote 
    \[
    \ac(x):=\frac{x}{\varpi_{\RF}^{\val(x)}}\in\Fo_{\RF}^{\times},
    \]
    where $\val:\RF^{\times}\rightarrow\BZ$ is the normalized valuation map such that $\val(\varpi_{\RF})=1$. Fix a non-trivial unramified character $\psi$ of $\RF$.
    \item
    Let
    \[
    \zeta_{\RF}(s):=\frac{1}{1-q^{-s}}
    \]
    be the local zeta function, where $s\in\BC$.
       
       \item  Let $\BD$ be an algebra over $\RF$ of one of the following types:
    \subitem [(1)] the field $\RF$ itself;
    \subitem [(2)] the split etale quadratic extension of $\RF$, i.e., $\BD=\RF\times\RF$;
    \subitem [(3)] the matrix algebra $\Mat_2(\RF)$ of $2\times 2$ matrices over $\RF$. 
    In all the cases, we identify $\RF$ with the subfield of $\BD$ consisting of scalar multiples of the identity. Equip $\BD$ with the $\RF$-involution $\overline{\cdot}$, together with the (reduced) trace and (reduced) norm $\tau,\nu:\BD\rightarrow\RF$ as follows:
    \subitem [(1)] If $\BD=\RF$, let $\overline{x}:=x$, $\tau(x):=x$, and $\nu(x):=x$. 
    \subitem [(2)] If $\BD=\RF\times\RF$, and $x=(a,b)$, let $\overline{x}:=(b,a)$, $\tau(x):=a+b$ and $\nu(x):=ab$.
    \subitem [(3)] If $\BD=\Mat_2(\RF)$, and $x=\begin{pmatrix}
    	a & b \\ c & d
    \end{pmatrix}$, let $\overline{x}:=\begin{pmatrix}
    	d & -b \\ -c & a
    \end{pmatrix}$, $\tau(x):=a+b$ and $\nu(x):=ad-bc$. 
     \item When the meaning is clear from the context, for a variety $\RX$ over $\RF$, we denote the set of $\RF$-rational points $\RX(\RF)$ simply by $\RX$.      
    \item For a smooth variety $\RX$ over $\RF$, we denote by $\CF(\RX)$ the space of Schwartz functions on the rational points of $\RX$. It means smooth functions of compact support. Similarly, we will use $\CS(\RX)$ to denote the space of Schwartz measures. Often $\RX(\RF)$ is a homogeneous $\RG(\RF)$-space for some reductive group $\RG$, and admits an invariant positive measure $\ud x$, then we have $\CS(\RX)=\CF(\RX) \cdot \ud x$.
    \item If $\RX$ is a homogeneous $\RG$-variety that admits an invariant rationally defined non-vanishing top degree differential form, we use $\ud x$ to denote the induced local unramified measure, which satisfies 
    \[
 \tau(\RX):=  \int_{\RX(\Fo_{\RF})}1\ud x=\frac{\#\RX(\kappa_{\RF})}{q^{\dim\RX}},
    \]
    where $\#$ means the cardinality. See \cite[Theorem 2.2.5]{Wei82}. Note that in the case that $\RX=\BG_a$, $\ud x$ is also the self-dual Haar measure on $\RF$. We also equip $\BD$ with the product Haar measure. It might be a little bit confusing between $\RX=\BG_m$ and $\RX=\BG_a$. When we use $\RF$ or $\RF^{\times}$ (instead of $\BG_a$ and $\BG_m$) to denote the integral domain, $\ud x$ is always reserved to denote the additive Haar measure, i.e., for a function $f$ on $\BG_m$, we will write
    \begin{align*}
    	\int_{\BG_m}f(x)\ud x=\int_{\RF^{\times}}f(x)\frac{\ud x}{|x|}
    \end{align*}
    interchangeably. To emphasize the choice of the multiplicative Haar measure on the torus, we adopt the notation $\ud^{\times}x$ where appropriate
        \item For the notation of the Fourier transforms, let $f$ be a measure on $\BG_a$, for $s\in\BC$, we use $(|\cdot|^s\psi(\cdot))\star$ to denote the operator of multiplicative convolution by the measure $ |x|^s\psi(x)\ud x $:
 \[
(|\cdot|^s\psi(\cdot))\star f(y):=\int_{\RF}|x|^s\psi(x)f(x^{-1}y)\ud x.
	\]
If the integral does not converge, then the convolution should be understood as the Fourier transform. If $\phi$ is a function on $\RF$, we also use the usual Fourier transform on the Affine line
\[
\BF_{\psi}(\phi)(x):=\int_{\RF} \phi(y)\psi^{-1}(xy)\ud y.
\]
Then if $f$ is represented by some function $\phi(x)\ud x$, we have
\[
(|\cdot|^s\psi(\cdot))\star f(y)=(|\cdot|^s\circ\BF_{\psi^{-1}}\circ\iota\circ|\cdot|^{1-s})(\phi)(y)\ud y,
\]
where $\iota(\phi)(x):=\phi(x^{-1})$.
\end{itemize}

\subsection{Acknowledgement}

I am deeply grateful to my advisor, Professor Dihua Jiang, for his steadfast guidance and for encouraging a rigorous approach to this research. I would also like to thank Professor Yiannis Sakellaridis for introducing me to this area of study. Special thanks are due to Professor Zhilin Luo and Yi Shan for many helpful discussions and valuable feedback. I am also grateful to Professor Fan Gao for clarifying several technical points regarding covering groups and for providing relevant references. 

A portion of this work was completed during my visit to the Institute for Advanced Study in Mathematics at Zhejiang University; I am indebted to Professor Binyong Sun for the invitation, as well as for the warm hospitality and the stimulating academic environment provided during my stay.

\section{Preliminaries}\label{sec_pre}

\subsection{Test measures for the Kuznetsov formula}\label{ssec_localktf}

For the purpose of this paper, let $\RG$ be a split reductive group over $\RF$ with fixed Borel pairs $(\RB,\RA)$. Let $\RG^\vee$ be the dual group with the corresponding Borel pair $(\RB^\vee,\RA^\vee)$. Let $X^*(\RA)$ be the $\RF$-rational characters of $\RA$ and $X_*(\RA)$ be the $\RF$-rational co-characters of $\RA$. Then $X^*(\RA)=X_*(\RA^\vee)$ and $X^*(\RA^\vee)=X_*(\RA)$. We denote by $X^{*}(\RA)^+$ and $X_{*}(\RA)^+$, respectively, the dominant cones in $X^*(\RA)$ and $X_*(\RA)$, respectively, with the positivity being associated to the chosen Borel subgroup. We denote by $\Delta$ the set of simple roots and $\Delta^{\vee}$ the set of simple coroots, respectively. In this way, we obtain the 
based root datum $(X^*(\RA),\Delta,X_*(\RA),\Delta^\vee)$ of $\RG$ with respect to $(\RB,\RA)$ and the based root datum $(X_{*}(\RA),\Delta^{\vee},X^{*}(\RA),\Delta)$ of $\RG^\vee$ with respect to $(\RB^{\vee},\RA^{\vee})$. 

Fix a hyperspecial vertex in the Bruhat-Tits building of $\RG$ that lies in the apartment determined by $\RA$. It determines a hyperspecial maximal open compact subgroup $\RK$ of $\RG$. Note that $\RK_{\RA}:=\RA\cap \RK$ is a hyperspecial subgroup of $\RA$. Let $\RA^+$ be the image of $X_{*}(\RA)^+$ under the map $\mu\mapsto \mu(\varpi_{\RF})$, where $\varpi_{\RF}\in \RF$ is a fixed uniformizer. We have the Cartan decomposition 
\[
\RG=\RK \RA^+\RK.
\]
Only here we take the Haar measure on $\RG$ such that the volume of $\RK$ is $1$. The spherical Hecke algebra of $\CH(\RG,\RK)$ is defined as the 
convolution algebra of compactly supported $\RK$-bi-invariant $\BC$-valued functions on $\RG$. The same definition applies to $\RA$ 
with respect to $\RK_{\RA}$ and the Weyl group $\BW_{\RG}$ acts on $\CH(\RA,\RK_{\RA})$.

The Satake isomorphism, as is defined in \cite{Sat63} yields an isomorphism of algebras:
\begin{align}\label{Sat}
   \CH(\RG,\RK)\rightarrow \CH(\RA,\RK_{\RA})^{\BW_{\RG}},\quad h \mapsto \left( t \mapsto \delta_{\RB}^{\frac{1}{2}}(t)\int_{\RN} h(tn)\ud n \right),
\end{align}
where $\delta_{\RB}$ is the modular character on $\RB$ and $\ud n$ is the Haar measure on $\RN$ such that the volume of $\RN\cap \RK$ 
is $1$. 

Note that $X_{*}(\RA)\cong \RA/\RK_{\RA}\colon\mu\mapsto \mu(\varpi_{\RF})$, which implies that 
\[
\CH(\RA,\RK_{\RA})^{\BW_{\RG}}\cong \BC[X_{*}(\RA)]^{\BW_{\RG}}=\BC[X^{*}(\RA^{\vee})]^{\BW_{\RG}}.
\]
Write \[\Sat:\CH(\RG,\RK)\rightarrow\CH(\RA,\RK_{\RA})^{\BW_{\RG}}\rightarrow\BC[X^*(\RA^{\vee})]^{\BW_{\RG}}.\]

It is known that for any finite-dimensional algebraic representation $\rho$ of $\RG^{\vee}(\BC)$, 
we must have $\tr\; \rho\in \BC[X^{*}(\RA^{\vee})]^{\BW_{\RG}}$, and 
there exists a canonical one-to-one correspondence between the set $\Pi_{\RF}(\RG,\RK)$ of equivalent classes of irreducible admissible spherical representations of $\RG$ with respect to $\RK$ and the set of conjugacy classes $\Fc$ of $\RA^{\vee}(\BC)/\BW_{\RG}$.

More precisely, let $\Fs\in X^*(\RA)\otimes_{\BZ}\BC/\BW_{\RG}.\left( \frac{2\pi i }{\log q} X^*(\RA)   \right)$, where following \cite{Sat63}, we use $\BW_{\RG}.\left( \frac{2\pi i }{\log q} X^*(\RA)   \right)$ to denote the group of affine transformations of $X^*(\RA)\otimes_{\BZ}\BC$ generated by the Weyl group $\BW_{\RG}$ and by the group of translations defined by $\frac{2\pi i}{\log q}X^*(\RA)$. On the one hand, denote $\chi(\Fs)$ to be the corresponding unramified character of $\RA$, and let $\RI(\chi(\Fs))$ be the normalized induced representation, whose underlying space consists of locally constant functions $\Phi$ on $\RG$ with compact support module $\RB$, satisfying
\[
\Phi(gb)=\delta_{\RB}^{-\frac{1}{2}}(b)\chi(\Fs)(b)\Phi(g),\;\forall g\in\RG,b\in\RB,
\]
with the action of $\RG$ given by $g.\Phi(x):=\Phi(g^{-1}x)$. $\RI(\chi(\Fs))$ admits a unique unramified subquotient, which we denote $\pi(\Fs)$.

On the other hand, view $\Fs\in X_*(\RA^{\vee})\otimes_{\BC}\BC/\BW_{\RG}.\left( \frac{2\pi i }{\log q} X_*(\RA^{\vee})   \right)$ and let $\Fc(\Fs)=q^{\Fs}\in\RA^{\vee}(\BC)/\BW$ be the corresponding conjugacy class. Then we have
\[
\tr_{\;\pi(\Fs)}(h)=\Sat(h)(\Fc(\Fs)) 
\]
for any $h\in\CH(\RG,\RK)$.

\begin{exm}
	Let $\RG=\SL_2$, $\RB:=\left\{ \begin{pmatrix}
		* & * \\ & * 
	\end{pmatrix}  \right\}$  and $\RA:=\left\{ \begin{pmatrix}
		* & \\ & *
	\end{pmatrix}\right\}$. Let $s\in X^*(\RA)\otimes_{\BZ}\BC= X_*(\RA^{\vee})\otimes_{\BZ}\BC\cong \BC$, which gives the unramified character $\chi(s)$ of $\RA$
	\[
   \RA\rightarrow\BC^{\times}:\begin{pmatrix}
		t & \\ & t^{-1}
	\end{pmatrix}\mapsto |t|^s.
	\]
	and the normalized induced representation $\RI(\chi(s))$. Let $\pi(s)$ be the unique irreducible unramified subquotient of $\RI(\chi(s))$. On the other hand, $\Fc(s)$ is represented by $\left( \begin{pmatrix}
	q^s & \\ & 1
\end{pmatrix}    \right)$, in the maximal torus of $\PGL_2$. The Satake transform states that we have an isomorphism
\[
\Sat:\CH(\SL_2,\SL_2(\Fo_{\RF}))\rightarrow\BC[\RA^{\vee}]^{\BZ/2\BZ }=\BC[x,x^{-1}]^{\BZ/2\BZ}
\]
such that for any $h\in\CH(\SL_2,\SL_2(\Fo_{\RF}))$, we have
\[
\tr_{\pi(s)}(h)=\Sat(h)(q^s).
\]
Similarly for $\RG=\PGL_2$.
\end{exm}

Let $\rho$ be an algebraic representation of $\RG^{\vee}(\BC)$. Consider the $L$-value
\[
L(\rho,s):=\frac{1}{\det(1-q^{-s}\rho)}\in\BC(\RG^{\vee})^{\RG^{\vee}},
\]
where $\BC(\RG^{\vee})^{\RG^{\vee}}$ is the set of invariant rational functions on $\RG^{\vee}$ under the adjoint action.
When $\displaystyle{\rho=\bigoplus_{i=1}^m\rho_i}$ is reducible, we can allow $s=(s_i)_{i=1}^m$ to be a collection of complex numbers, and define 
\[
L(\rho,s):=\prod_{i=1}^mL(\rho_i,s_i).
\]
We proceed with a single $s$, and the arguments for general $s$ are similar. The $L$-factor can be written as a formal Taylor series in $q^{-s}$:
\[
L(\rho,s)=\sum_{n=0}^{\infty}q^{-ns}\tr\;(\Sym^n\rho).
\]
The generating series of the local $L$-function is the Whittaker function
\[
\Phi_{L(\rho,s)}:=\sum_{n=0}^{\infty}q^{-ns}\Sat^{-1}(\Sym^n\rho)\star 1_{\RN\backslash\RG(\Fo_{\RF})}.
\]

When $\RG=\PGL_2$ or $\SL_2$, let $\RB$ be the Borel subgroup consisting of upper triangular matrices, $\RA$ the maximal torus consisting of diagonal matrices. Let $\RN$ be the unipotent radical of $\RB$ consisting of strictly upper triangular matrices. Choose $\Bw:=\begin{pmatrix}
	0 & -1 \\ 1 & 0
\end{pmatrix}$ to be the non-trivial Weyl element. Let $\FC$ be the quotient $\RN\backslash\RG\sslash\RN$, and $\FC^{\circ}$ its open subset corresponding to the open Bruhat cell. We embed $\RA$ into it by $t\mapsto [\Bw t]$, the class of the element $\Bw t$. Let
\[
\CS(\RN,{\psi}\backslash\RG/\RN,\psi)
\]
be the space of measures on $\RA$ of the form
\begin{align}\label{eq_twpufor}
f(t)= \frac{ \delta_{\RB}(t)}{1-q^{-2}} \int_{\RN\times\RN}\Phi(n_1\Bw tn_2)\psi^{-1}(n_1n_2)\ud n_1\ud n_2\cdot\ud t,
\end{align}
where recall $\ud n$ is such that $\vol(\RN(\Fo_{\RF}),\ud n)=1$, $\ud t$ is such that $\vol(\RA(\Fo_{\RF}),\ud t)=1-q^{-1}$, and $\Phi$ is a Schwartz function on $\RG$. We also fix the coordinates of $\RA$ by
\[
\xi(t)=e^{\alpha}(t),\;\mathrm{if}\;\RG=\PGL_2,
\]
and
\[
\zeta(t)=e^{\frac{\alpha}{2}}(t),\;\mathrm{if}\;\RG=\SL_2,
\]
where $\alpha$ is the positive root, and we use the exponential notation to denote the corresponding character. Then $\FC^{\circ}$ can be identified with $\BG_m$ or $\BG_m^2$ with sections
\[
\xi\mapsto \begin{pmatrix}
	 0 & -1 \\ \xi & 0
\end{pmatrix}\;\mathrm{if}\;\RG=\PGL_2,
\]
and
\[
\zeta\mapsto \begin{pmatrix}
	 0 & -\zeta^{-1} \\ \zeta & 0
\end{pmatrix}\;\mathrm{if}\;\RG=\SL_2.
\]
The elements of $\CS(\RN,\psi\backslash\RG/\RN,\psi)$, viewed as measures on $\RF$, are bounded, supported away from zero, while in a neighborhood of $0\in\BA^1$, they have a singularity which is called the Kloosterman germ, because they are smooth multiples of the measures
\[
\xi\mapsto \int_{|u|^2=|\xi|} \psi^{-1}\left(  \frac{  u}{\xi}+u^{-1}    \right)\ud u\cdot\ud^{\times}\xi,
\]
respectively,
\[
\zeta\mapsto \int_{|u|\in \pm 1+\Fp_{\RF} } \psi^{-1}\left(\frac{ u+u^{-1}}{\zeta}   \right)\ud u\cdot  |\zeta|  \ud^{\times}\zeta.
\]
For more details, we refer the reader to \cite{Sak13, Sak21, Sak22}.

Now let us consider the twisted-push forward of the non-standard measure $\Phi_{L(\rho,s)}$:
\[
f_{L(\rho,s)}(t):=\frac{\delta_{\RB}(t)}{1-q^{-2}}\int_{\RN}^*\Phi_{L(s,\rho)}(\Bw tn)\psi^{-1}(n)\ud n\cdot \ud t.
\]
 We call $f_{L(\rho,s)}$ the basic vector in the Whittaker space of the $L$-value $L(\rho,s)$.

Specialize to $\rho$ being a sum of copies of the standard representations of $\RG^{\vee}=\SL_2$ when $\RG=\PGL_2$, and the sum of copies of adjoint representations of $\RG^{\vee}=\PGL_2$ when $\RG=\SL_2$, and write 
\[
\rho=\bigoplus_{i=1}^m\rho_i,
\]
correspondingly write $s=(s_i)_{i=1}^m$. We let $\CS^-_{L(\rho,s)}(\RN,\psi\backslash\RG/\RN,\psi)$ be the space of measures on $\FC$ which on any compact set coincide with elements of $\CS(\RN,\psi\backslash\RG/\RN,\psi)$, while in a neighborhood of infinity, in the above coordinates, when all of the $s_i$'s are distinct, are of the form
\[
\sum_j C_j(\zeta^{-1})|\zeta|^{1-s_i}\ud^{\times}\zeta,
\]
in the case $\RG=\SL_2$, and 
\[
\sum_j C_j(\xi^{-1})|\xi|^{\frac{1}{2}-s_i}\ud^{\times}\xi,
\]
in the case $\RG=\PGL_2$, where $C_j$'s are locally constant functions in a neighborhood of zero. When two or more of the $s_i$'s coincide with some complex number $s_0$, the corresponding summands at infinity will be replaced by
\[
\left(C_1(\zeta^{-1})+C_2(\zeta^{-1})\log|\zeta|+C_3(\zeta^{-1})\log^2|\zeta|+\cdots\right)\cdot|\zeta|^{1-s_0}\ud^{\times}\zeta,
\]
and similarly when $\RG=\PGL_2$. According to \cite{Sak22}, $f_{L(\rho,s)}\in\CS^-_{L(\rho,s)}(\RN,\psi\backslash\RG/\RN,\psi)$. Moreover, one can compute explicitly that

\begin{prp}\label{prp_basic}
	Let $\RG=\PGL_2$ and $\rho=\std$ be the standard representation of $\SL_2(\BC)$. For the $L$-value $L\left(\std,s_1\right)L\left(\std,s_2\right)$, we have the following explicit expressions of the basic vector:
		\begin{align*}
		\Phi_{L(\std,s_1)(\std,s_2)}=\left\{ \begin{array}{rcl}
 	 \sum_{m=0}^{\infty}\frac{q^{-m(s+\frac{1}{2})}}{1-q^{-2s}}(m+1) 1_{x_m\RK} & \mbox{if} &  s_1=s_2=s;\\ \sum_{m=0}^{\infty}\left(\frac{q^{-m(s_1+\frac{1}{2})}-q^{-m(s_2+\frac{1}{2})}q^{s_1-s_2}}{(1-q^{s_1-s_2})(1-q^{-(s_1+s_2)})} \right)1_{x_m\RK} & \mbox{if} & s_1\neq s_2.
\end{array}\right.,  
	\end{align*}
where 
	\[
	1_{x_m\RK}\left(n\begin{pmatrix}
		\varpi_{\RF}^{\ell} & 0 \\ 0 & 1
	\end{pmatrix} k \right)=\left\{ \begin{array}{rcl} 0 & \mbox{if} & \ell\neq m; \\  \psi(n) & \mbox{if} & \ell=m. \end{array}\right. 
	\]
	Here $n\in\RN$ and $k\in\PGL_2(\Fo_{\RF})$. And
	\begin{itemize}
		\item when $s_1=s_2=s$, 
			\begin{align*}
		&\frac{f_{L\left(\std,s_1\right)L\left(\std,s_2\right)}\left( \begin{pmatrix}
			\xi & 0 \\ 0 & 1
		\end{pmatrix}  \right)}{|\xi|\ud^{\times}\xi}\\
		&=(1-q^{-2})^{-1}(1-q^{-2s})^{-1}\left\{ \begin{array}{rcl}
|\xi|^{-s-\frac{1}{2}}\left( (1-q^{-2s-1)})\log_q|\xi|+(1-3q^{-2s-1})    \right) & \mbox{for} & |\xi|\geq 1; \\
 -2q^{-s-\frac{1}{2}} & \mbox{for} & |\xi|=q^{-1}; \\
|\xi|^{-1}\int_{|x|^2=|\xi|}  \psi^{-1}\left( \frac{x}{\xi }+ x^{-1} \right)\ud x & \mbox{for} & |\xi|<q^{-1} .
\end{array}\right.
	\end{align*}		where $\ud^{\times}\xi$ is the Haar measure such that $\vol(\RT(\Fo_{\RF}),\ud^{\times}\xi)=1-q^{-1}$.
    \item When $s_1\neq s_2$, 
   \begin{align*}
		&\frac{f_{L\left(\std,s_1\right)L\left(\std,s_2\right)}\left( \begin{pmatrix}
			\xi & 0 \\ 0 & 1
		\end{pmatrix}  \right)}{|\xi|\ud^{\times}\xi}=(1-q^{-2})^{-1}(1-q^{s_1-s_2})^{-1}\left(1-q^{-(s_1+s_2)}\right)^{-1}\\
		&\cdot \left\{ \begin{array}{rcl}
|\xi|^{-s_1-\frac{1}{2}} (1-q^{-2s_1-1)})-|\xi|^{-s_2-\frac{1}{2}}q^{s_1-s_2}(1-q^{-2s_2-1})   & \mbox{for} & |\xi|\geq 1; \\
-q^{-(s_1+\frac{1}{2})}+q^{-(s_2+\frac{1}{2})}q^{s_1-s_2} & \mbox{for} & |\xi|=q^{-1}; \\
(1-q^{s_1-s_2})   |\xi|^{-1}\int_{|x|^2=|\xi|}  \psi^{-1}\left( \frac{x}{\xi }+ x^{-1} \right)\ud x & \mbox{for} & |\xi|<q^{-1} .
\end{array}\right.
	\end{align*}	
	\end{itemize}
	\end{prp}

\begin{proof}
	Let us start with the series
	\begin{align*}
		L\left(\std,s_1\right)L\left(\std,s_2\right)&=\frac{1}{\det(1-q^{- s_1}\std)}\cdot \frac{1}{\det(1-q^{- s_2}\std)}\\
		&=\left(\sum_{m=0}^{\infty} q^{-m s_1 }\tr\;(\Sym^m\std)  \right)\left(\sum_{m=0}^{\infty} q^{-m s_1 }\tr\;(\Sym^m\std)  \right)\\
		&=\sum_{m_1,m_2=0}^{\infty}q^{-m_1s_1-m_2s_2}\tr\;(\Sym^{m_1}\std\otimes\Sym^{m_2}\std).
	\end{align*}
	According to \cite[Section 11.2]{FH91},
	\[
	\Sym^{m_1}\std\otimes\Sym^{m_2}\std=\sum_{\ell=0}^{\min\{m_1,m_2\}}\Sym^{m_1+m_2-\ell}\std.
	\]
	Therefore, we have
	\begin{align*}
		L\left(\std,s_1\right)L\left(\std,s_2\right)&=\frac{1}{\det(1-q^{- s_1}\std)}\cdot \frac{1}{\det(1-q^{- s_2}\std)}\\
		&=\sum_{m_1,m_2=0}^{\infty}q^{-m_1s_1-m_2s_2}\tr\;\left(\sum_{\ell=0}^{\min\{m_1,m_2\}}\Sym^{m_1+m_2-\ell}\std\right)\\
		&=\sum_{m=0}^{\infty}  \left( \sum_{\ell=0}^{\infty}q^{-(m+2\ell)s_1}\sum_{m_2=\ell}^{m+\ell}q^{m_2(s_1-s_2)}   \right)       \tr\;(    \Sym^{m}\std)\\
		&=\left\{ \begin{array}{rcl}
 	 \sum_{m=0}^{\infty}\frac{q^{-ms}}{1-q^{-2s}}(m+1)\tr\;(\Sym^m\std) & \mbox{for} &  s_1=s_2=s;\\ \sum_{m=0}^{\infty}\left(\frac{q^{-ms_1}-q^{-ms_2}q^{s_1-s_2}}{(1-q^{s_1-s_2})(1-q^{-(s_1+s_2)})} \right)\tr\;(\Sym^m\std) & \mbox{for} & s_1\neq s_2.
\end{array}\right.  
	\end{align*}
	The Casselman-Shalika formula tells us
	\[
	\Sat^{-1}(\Sym^m\std)\star 1_{\RN\backslash\RG(\Fo_{\RF})}=q^{-\frac{m}{2}}1_{x_m\RK}.
	\]
 Therefore,
	\begin{align*}
		\Phi_{L(\std,s_1)(\std,s_2)}=\left\{ \begin{array}{rcl}
 	 \sum_{m=0}^{\infty}\frac{q^{-m(s+\frac{1}{2})}}{1-q^{-2s}}(m+1) 1_{x_m\RK} & \mbox{if} &  s_1=s_2=s;\\ \sum_{m=0}^{\infty}\left(\frac{q^{-m(s_1+\frac{1}{2})}-q^{-m(s_2+\frac{1}{2})}q^{s_1-s_2}}{(1-q^{s_1-s_2})(1-q^{-(s_1+s_2)})} \right)1_{x_m\RK} & \mbox{if} & s_1\neq s_2.
\end{array}\right.  
	\end{align*}
According to \cite[Section 6.3]{Sak13}, we have that
\begin{align*}
	\int_{\RN}1_{x_m\RK}\left(\Bw \begin{pmatrix}
		\xi & 0 \\ 0 & 1
	\end{pmatrix}  n  \right)&\psi^{-1}(n)\ud n
	\\ &= \left\{ \begin{array}{rcl}
 	1 & \mbox{if} & |\xi|=q^m;   \\ -1 & \mbox{if} & |\xi|=q^{m-2},m\geq 1; \\
 	|\xi|^{-1}\int_{|x|^2=|\xi|}\psi^{-1}\left(   \frac{x}{\xi}+x^{-1}\right)\ud x & \mbox{if} & |\xi|<1,m=0.
\end{array}\right.  
\end{align*}
	Combine all these, in the case that $s_1=s_2=s$, we have
		\begin{align*}
		&\frac{f_{L\left(\std,s_1\right)L\left(\std,s_2\right)}\left( \begin{pmatrix}
			\xi & 0 \\ 0 & 1
		\end{pmatrix}  \right)}{|\xi|\ud^{\times}\xi}\\
		&=(1-q^{-2})^{-1}(1-q^{-2s})^{-1}\left\{ \begin{array}{rcl}
|\xi|^{-s-\frac{1}{2}}\left( (1-q^{-2s-1)})\log_q|\xi|+(1-3q^{-2s-1})    \right) & \mbox{for} & |\xi|\geq 1; \\
 -2q^{-s-\frac{1}{2}} & \mbox{for} & |\xi|=q^{-1}; \\
|\xi|^{-1}\int_{|x|^2=|\xi|}  \psi^{-1}\left( \frac{x}{\xi }+ x^{-1} \right)\ud x & \mbox{for} & |\xi|<q^{-1} .
\end{array}\right.
	\end{align*}	
	And in the case that $s_1\neq s_2$, we have
		\begin{align*}
		&\frac{f_{L\left(\std,s_1\right)L\left(\std,s_2\right)}\left( \begin{pmatrix}
			\xi & 0 \\ 0 & 1
		\end{pmatrix}  \right)}{|\xi|\ud^{\times}\xi}=(1-q^{-2})^{-1}(1-q^{s_1-s_2})^{-1}\left(1-q^{-(s_1+s_2)}\right)^{-1}\\
		&\cdot \left\{ \begin{array}{rcl}
|\xi|^{-s_1-\frac{1}{2}} (1-q^{-2s_1-1)})-|\xi|^{-s_2-\frac{1}{2}}q^{s_1-s_2}(1-q^{-2s_2-1})   & \mbox{for} & |\xi|\geq 1; \\
-q^{-(s_1+\frac{1}{2})}+q^{-(s_2+\frac{1}{2})}q^{s_1-s_2} & \mbox{for} & |\xi|=q^{-1}; \\
(1-q^{s_1-s_2})   |\xi|^{-1}\int_{|x|^2=|\xi|}  \psi^{-1}\left( \frac{x}{\xi }+ x^{-1} \right)\ud x & \mbox{for} & |\xi|<q^{-1} .
\end{array}\right.
	\end{align*}	
	\end{proof}

\begin{prp}
	Let $\RG=\SL_2$ and $\rho=\Ad$ be the adjoint representation of $\PGL_2(\BC)$. For the $L$-value $L(\Ad,s)$, we have the following explicit expressions of the basic vector:
	\begin{align*}
		\Phi_{L(\Ad,s)}=\sum_{m=0}^{\infty}\frac{q^{-m(s+1)}}{1-q^{-2s}}1_{x_m\RK},
	\end{align*}
where \[
	1_{x_m\RK}\left(n\begin{pmatrix}
		\varpi_{\RF}^{\ell} & 0 \\ 0 & \varpi_{\RF}^{-\ell}
	\end{pmatrix} k \right)=\left\{ \begin{array}{rcl} 0 & \mbox{if} & \ell\neq m; \\  \psi(n) & \mbox{if} & \ell=m. \end{array}\right. 
	\]
	Here $n\in\RN$ and $k\in\RK$. And
\begin{align*}
	&\frac{f_{L(\Ad,s)}\left(\begin{pmatrix}
		\zeta & 0 \\ 0 & \zeta^{-1}
	\end{pmatrix} \right)}{|\zeta|^2\ud^{\times}\zeta}\\
	&=(1-q^{-2})^{-1}(1-q^{-2s})^{-1}\left\{ \begin{array}{rcl}
|\zeta|^{-s-1}(1-q^{-s-1}) & \mbox{for} & |\zeta|\geq 1; \\
	|\zeta|^{-1}\int_{x\in \pm 1+\Fp_{\RF}}\psi^{-1}\left(   \frac{x+x^{-1}}{\xi}\right)\ud x & \mbox{for} & |\zeta|\leq q^{-1} .
\end{array}\right.
	\end{align*}
\end{prp}

\begin{proof}
	Let $\rho_k$ be the descent of $\Sym^{2k}\std$ of $\SL_2(\BC)$ to $\PGL_2(\BC)$. Then according to \cite{Rud90}, we have
	\[
	\Sym^m\Ad=\sum_{k\leq m, k\equiv m \mod 2}\rho_k.
	\]
	Therefore,
	\begin{align*}
	L(\Ad,s)&=\frac{1}{\det(1-q^{-s}\Ad)} =\sum_{m=0}^{\infty}q^{-ms}\;\Sym^m\Ad= \sum_{m=0}^{\infty}   \left(\sum_{k=m,k\equiv m \mod 2}^{\infty} q^{-ks}\right)\rho_m       \\
	&=\sum_{m=0}^{\infty}\frac{q^{-ms}}{1-q^{-2s}}\rho_m.
	\end{align*}
	Using the Casselman-Shalika formula again, we have
	\[
	\Sat^{-1}(\rho_m)\star1_{\RN\backslash\RG(\Fo_{\RF})}=q^{-m}1_{x_{m}\RK}.
	\]

Therefore,
	\begin{align*}
		\Phi_{L(\Ad,s)}=\sum_{m=0}^{\infty}\frac{q^{-m(s+1)}}{1-q^{-2s}}1_{x_m\RK}.
	\end{align*}
	Similar to the computation in \cite{Sak13}, it is easy to see that
	\begin{align*}
	\int_{\RN}1_{x_m\RK}\left(\Bw \begin{pmatrix}
		\zeta & 0 \\ 0 & \zeta^{-1}
	\end{pmatrix}  n  \right)&\psi^{-1}(n)\ud n \\&= \left\{ \begin{array}{rcl}
 	1 & \mbox{if} & |\zeta|=q^m;   \\ -1 & \mbox{if} & |\zeta|=q^{m-1},m\geq 1; \\
 	|\zeta|^{-1}\int_{x\in \pm 1+\Fp_{\RF}}\psi^{-1}\left(   \frac{x+x^{-1}}{\xi}\right)\ud x & \mbox{if} & |\xi|<1,m=0.
\end{array}\right.  
\end{align*}
Combine them, and we have
\begin{align*}
	&\frac{f_{L(\Ad,s)}\left(\begin{pmatrix}
		\zeta & 0 \\ 0 & \zeta^{-1}
	\end{pmatrix} \right)}{|\zeta|^2\ud^{\times}\zeta}
	\\&=(1-q^{-2})^{-1}(1-q^{-2s})^{-1}\left\{ \begin{array}{rcl}
|\zeta|^{-s-1}(1-q^{-s-1}) & \mbox{for} & |\zeta|\geq 1; \\
	|\zeta|^{-1}\int_{x\in \pm 1+\Fp_{\RF}}\psi^{-1}\left(   \frac{x+x^{-1}}{\xi}\right)\ud x & \mbox{for} & |\zeta|\leq q^{-1} .
\end{array}\right.
	\end{align*}
\end{proof}

\subsection{$\epsilon$-Hermitian modules, unitary groups and Weil representations}\label{ssec_weil}

Let $\BV$ and $\BW$ be two right $\BD$-modules, which are finite dimensional over $\RF$. We will denote the set of right $\BD$-module morphisms by 
\[
\Hom_{\BD}(\BV,\BW):=\{T:\BV\rightarrow\BW\mid T(v_1a+v_2b)=T(v_1)a+T(v_2)b,\;\forall v_1,v_2\in\BV,a,b\in\BD\}.
\]
If $\BV=\BW$, we will denote this set by $\End_{\BD}(\BV)$. Set $\GL_{\BD}(\BV)$ to be the set of elements $T$ in $\End_{\BD}(\BV)$ such that $T$ is invertible.

\begin{dfn}
	Let $\epsilon=\pm 1$. We say that $(\BV,\BB)$ is a right $\epsilon$-Hermitian $\BD$-module, if $\BV$ is a right $\BD$-module, and $\BB:\BV\times\BV\rightarrow\BD$ is a map such that
	\begin{itemize}
		\item [(1)] $\BB$ is sesquilinear, i.e., for all $v_1,v_2,v_3\in\BV$, $a,v\in\BD$,
		\[
		\BB(v_1,v_2a+v_3b)=\BB(v_1,v_2)a+\BB(v_1,v_3)b
		\]
		and
		\[
		\BB(v_1a+v_2b,v_3)=\overline{a}\BB(v_1,v_3)+\overline{b}\BB(v_2,v_3).
		\]
		\item [(2)] $\BB$ is $\epsilon$-$Hermitian$, i.e., for all $v_1,v_2\in\BV$,
		\[
		\BB(v_1,v_2)=\epsilon\overline{\BB(v_2,v_1)}.
		\]
		\item [(3)] $\BB$ is non-degenerate.
	\end{itemize}
	We call $\BB$ an $\epsilon$-Hermitian form. Usually, $1$-Hermitian $\BD$-modules (resp. forms) are also called Hermitian modules (resp. forms), and $-1$-Hermitian modules (resp. forms) are called skew-Hermitian modules (resp. forms).
\end{dfn}

Given a right $\epsilon$-Hermitian $\BD$-module $(\BV,\BB)$, we write 
\[
\RU(\BV,\BB):=\{  g\in\GL_{\BD}(\BV) \mid \BB(gv_1,gv_2)=\BB(v_1,v_2),\;\forall v_1,v_2\in\BV  \}
\]
for the unitary group of $(\BV,\BB)$. Then $\RU(\BV,\BB)$ acts on $\BV$ on the left. When $\BB$ is clear from the context, we will also write $\RU(\BV)$ for simplicity.

Given a right $\epsilon$-Hermitian $\BD$-module $(\BV,\BB)$, we can construct a left $\BD$-module $\overline{\BV}$ together with $\overline{\BB}:\overline{\BV}\times \overline{\BV}\rightarrow\BD$ in the following way: as a set, $\overline{\BV}$ is the set of symbols $\{ \overline{v}\mid v\in\BV\}$. Then we give $\overline{\BV}$ a left $\BD$-module structure by setting, for all $v_1,v_2\in\BV$, $a\in\BD$,
\[
\overline{v_1}+\overline{v_2}:=\overline{v_1+v_2},\; a\overline{v}:=\overline{v\overline{a}},
\]
and
\[
\overline{\BB}:\overline{\BV}\times\overline{\BV}\rightarrow\BD:(\overline{v_1},\overline{v_2})\mapsto \overline{\BB(v_2,v_1)}.
\]

Let us briefly recall the classification of (skew-)Hermitian modules and the corresponding unitary groups.

\begin{exm}
	If $\BD=\RF$, the Hermitian $\BD$-modules are classfied by the non-degenerate quadratic forms over $\RF$, which are determined by the invariants: the discriminant $\disc(\BV)\in \RF^{\times}/(\RF^{\times})^2$, the Hasse-Witt invariant $\epsilon(\BV)=\pm 1$, and the dimension $\dim_{\RF}(\BV)$. For more details, see \cite{Ser73}. In this case, we have $\RU(\BV,\BB)=\RO(\BV,\BB)$. When $\BV=\RF^n$, and $\BB$ is given by
	\[
	\RF^{n}\times\RF^n:\left(\begin{pmatrix}
		x_1 \\ \vdots \\ x_n
	\end{pmatrix}  ,    \begin{pmatrix}
		y_1\\ \vdots \\ y_n
	\end{pmatrix}  \right)\mapsto\sum_{i=1}^nx_iy_{n+1-i},
	\]
	we call it the standard Hermitian model of rank $n$ over $\RF$. When $n=2m+1$ is odd, we will add one more $x_{m+1}y_{m+1}$ for the purpose of the paper.
	
	The skew-Hermitian forms are classified by non-degenrate alternative formrs over $\RF$, which is unique up to isomorphism, see \cite[Lemma 1.1.5]{GW09}. For the purpose of this paper, when $\BW=\RF^2$, and $\BB$ is given by
	\[
	\RF^2\times\RF^2\rightarrow\RF:\left( \begin{pmatrix}
		x_1 \\ x_2
	\end{pmatrix} ,   \begin{pmatrix}
		y_1 \\ y_2
	\end{pmatrix}  \right)\mapsto x_1y_2-x_2y_1,
	\]
	we call it the standard skew-Hermitian model of rank $2$ over $\RF$. We also write $e:=\begin{pmatrix}
		1 \\ 0
	\end{pmatrix}$ and $f:=\begin{pmatrix}
		0 \\ 1
	\end{pmatrix}$, and call $\BW= e\BD+f \BD$ the standard polarization.
	\end{exm} 
\begin{exm}
	If $\BD=\RF\times\RF$. Since $\BV$ is an $\RF\times\RF$-module, and hence has the form $\RV\times\RV^{\vee}$, for some $\RF$-vector space $\RV$ and its $\RF$-dual space $\RV^{\vee}$. Up to isomorphism, any Hermitian $\BD$-module is isomorphic to the one defined by 
	\[
	\BB:\left((v_1,v_1^{\vee}),(v_2,v_2^{\vee})\right) \mapsto(  \langle v_1,v_2^{\vee}\rangle,\langle v_2,v_1^{\vee}\rangle     )\in\BD=\RF\times\RF, 
	\]
	where $\langle \cdot,\cdot\rangle :\RV\times\RV^{\vee}\rightarrow\RF$ is the canonical pair between $\RV$ and $\RV^{\vee}$. 
	
	In this case, it is easy to see 
	\[
	\RU(\BV,\BB)\cong\GL(\RV),
	\]
	where the embedding is given by $\GL(\RV)\ni g\mapsto\begin{pmatrix}
		g & 0 \\ 0 & {^t g^{-1}}
	\end{pmatrix}$ in terms of matrices. We also fix a basis of $\RV$ and its dual basis, then $\RV=\RF^n$ and $\BB$ is given by
	\[
	\left(  \left(  \begin{pmatrix}
		x_1\\ \vdots \\ x_n
	\end{pmatrix},\begin{pmatrix}
		y_1 \\ \vdots \\ y_n
	\end{pmatrix}   \right),  \left(  \begin{pmatrix}
		\xi_1\\ \vdots \\ \xi_n
	\end{pmatrix},\begin{pmatrix}
		\zeta_1 \\ \vdots \\ \zeta_n
	\end{pmatrix}   \right)      \right)\mapsto \left(  \sum_{i=1}^nx_i\zeta_{i},\sum_{i=1}^n y_i\xi_{i}  \right).
	\]
	We call it the standard Hermitian model of rank $n$ over $\RF\times\RF$.
	
	Similarly, any skew Hermitian space is of the form $\BW=\RW\times\RW^{\vee}$ with
	\[
	\BB:\left((w_1,w_1^{\vee}),(w_2,w_2^{\vee})\right) \mapsto(  \langle w_1,w_2^{\vee}\rangle, -\langle w_2,w_1^{\vee}\rangle     )\in\BD=\RF\times\RF, 
	\]
	up to isomorphism. Then $\RU(\BW,\BB)\cong\GL(\RW)$ as well. In the special case that $\BW=\RW\times\RW^{\vee}$ for a two dimensional vector space $\RW$ over $\RF$, we have $\RU(\BW,\BB)\cong\GL_2$. Choose a basis $\{e_1,e_2\}$ of $\RW$ and its dual basis $\{e_1^{\vee},e_2^{\vee}\}\in\RW^{\vee}$, and let 
	\[
	e=(e_1,e_2^{\vee}),\;f=(-e_2,e_1^{\vee}),
	\]   	
	then we have $\BB(e,f)=1$ and $\BB(e,e)=\BB(f,f)=0$. We call $\BW$ the standard skew-Hermitian model of rank $2$ over $\RF\times\RF$ and $\BW=e\BD+f\BD$ the standard polarization of $\BW$.
	\end{exm}

    \begin{exm}
	If $\BD=\Mat_2(\RF)$. According to the structure theory of simple central algebras, $\BV$ has the form 
	\[
\RF^2\oplus\RF^2\oplus\cdots\oplus\RF^2\cong\RF^n\otimes_{\RF}\RF^2,
	\]
	where we view $\RF^2$ as the space of row vectors, and $\Mat_2(\RF)$ acts on the right.
	
	Let $\BB$ be an $\epsilon$-Hermitian form on $\BV$, then $\BB$ defines a map $\sigma_{\BB}:\End_{\BD}(\BV)\rightarrow\End_{\BD}(\BV)$, which is called the adjoint involution, characterized by
	\[
	\BB(v_1,f(v_2))=\BB(\sigma_{\BB}(f)(v_1),v_2),\;\forall v_1,v_2\in\BV,\;f\in\End_{\BD}(\BV).
	\]
	This map $\BB\mapsto\sigma_{\BB}$ is a one-to-one correspondence between $\epsilon$-Hermitian forms on $\BV$ up to a scalar in $\RF^{\times}$, and involutions $\sigma$ on $\End_{\BD}(\BV)$ such that $\sigma(x)=x$ for all $x\in\RF$. See \cite[Chapter 1, Section 4
	]{KMRT98}. Accoding to the Schur-Weyl duality, we know $\End_{\BD}(\BV)=\End_{\RF}(\RF^n)=\Mat_n(\RF)$, acting on $\BV=\RF^n\otimes_{\RF}\RF^2$ on the left. Then the $\epsilon$-Hermitian forms are classified by convolutions on $\Mat_n$ preserving $\RF$.
	
	Let us consider the special case that $\BV=\BD^n$, where $\BD$ acts on the right. Consider the Hermitian form
	\[
	\BB:\BD^n\times\BD^n\rightarrow\BD: ((x_1,\cdots,x_n),(y_1,\cdots,y_n))\mapsto \sum_{i=1}^n \overline{x_i}y_i.
	\]
	If we identify $\BD\cong\RF^2\otimes_{\RF}\RF^2$, and $\BD^n\cong\RF^{2n}\otimes_{\RF}\RF^2$ as above, then the above form is given by
	\begin{align*}
	\left( \begin{pmatrix}
		v_1 \\ \vdots \\ v_{2n}
	\end{pmatrix} \otimes (x_1,x_2) \right. &, \left.\begin{pmatrix}
		w_1 \\ \vdots \\ w_{2n}
	\end{pmatrix} \otimes (y_1,y_2) \right)    \\
	&\mapsto (v_1,\cdots,v_{2n})\cdot  \begin{pmatrix}
		J &  & \\ & \ddots  & \\ & & J
	\end{pmatrix}   \cdot\begin{pmatrix}
		w_1 \\ \vdots \\w_{2n} 
	\end{pmatrix} \cdot J^{-1} \begin{pmatrix}
		x_1 \\ x_2
	\end{pmatrix}  (y_1,y_2),
	\end{align*}
	where $J=\begin{pmatrix}
		0 & 1 \\ -1 & 0
	\end{pmatrix}$.
	At this time, we have $\RU(\BV,\BB)=\Sp_{2n}$, where $\Sp_{2n}$ is defined by the symplectic form given by $\diag(J,\cdots,J)$, and $\Sp_{2n}$ acts on $\BD^n=\RF^{2n}\otimes_{\RF}\RF^2$ on the left. And we call it the standard Hermitian model of rank $n$ over $\BD$.
	
	Another special case is that $\BW=\BD^2$, and consider the skew-Hermitian form
	\[
	\BB:\BD^2\times\BD^2\rightarrow\BD: ((x_1,x_2),(y_1,y_2))\mapsto \overline{x_1}y_2-\overline{x_2}y_1.
	\]
	Then identify $\BD\cong\RF^2\otimes_{\RF}\RF^2$, and $\BD^n\cong\RF^{2n}\otimes_{\RF}\RF^2$ as above, this form is given by
	\begin{align*}
	\left( \begin{pmatrix}
		v_1 \\ \vdots \\ v_{4}
	\end{pmatrix} \otimes (x_1,x_2) \right. &, \left.\begin{pmatrix}
		w_1 \\ \vdots \\ w_{4}
	\end{pmatrix} \otimes (y_1,y_2) \right)    \\
	&\mapsto (v_1,\cdots,v_{4}) \begin{pmatrix}
		  0  & J \\ -J & 0 
	\end{pmatrix}  \cdot\begin{pmatrix}
		w_1 \\ \vdots \\w_{4} 
	\end{pmatrix} \cdot J \begin{pmatrix}
		x_1 \\ x_2
	\end{pmatrix}  (y_1,y_2),
	\end{align*}
	and hence $\RU(\BW,\BB)=\RO_4$. Let $e=(\RI_2,0)\in\BD^2$ and $f=(0,\RI_2)\in\BD^2$, then we have $\BB(e,f)=1$, and $\BB(e,e)=\BB(f,f)=0$. In this case, we call it the standard skew-Hermitian model of rank $2$ over $\BD$. 
	\end{exm}

Let $(\BV,\BB_{\BV})$ be a Hermitian $\BD$-module, and $(\BW,\BB_{\BW})$ be a skew Hermitian $\BD$-module. On the $\RF$-vector space $\BV\otimes_{\BD}\overline{\BW}$, one can define an alternating form $\BB$ by setting
\[
\BB(v_1\otimes w_1,v_2\otimes w_2):=\tau ( \BB_{\BV}(v_1,v_2)\cdot \overline{\BB_{\BW}} (\overline{w_2},\overline{w_1})   )\in\RF 
\]
for $v_1,v_2\in\BV$ and $w_1,w_2\in\overline{\BW}$. Let 
\[
\Sp(\BV\otimes_{\BD}\overline{\BW},\BB) 
\]
be the symplectic group for the above alternating form. Moreover, there is a natural map
\[
\RU(\BV,\BB_{\BV})\times\RU(\BW,\BB_{\BW})\rightarrow\Sp(\RV\otimes_{\BD}\overline{\BW},\BB ): (g_1,g_2)\mapsto (  v\otimes \overline{w} \mapsto g_1v\otimes \overline{g_2w}  ).
\]
Let $\BH=\BH(\BV\otimes_{\BD}\overline{\BW})=\BV\otimes_{\BD}\overline{\BW}\ltimes\RF$ be the Hsisenberg group associated with $\BV\otimes_{\BD} \overline{\BW}$. 

Assume now $\BV$ is the standard Hermitian model of rank $n$ over $\BD$, $\BW$ is the standard skew-Hermitian model of rank $2$ over $\BD$, and $\BV\otimes_{\BD}\overline{\BW}=\BX\oplus\BY$ be the complete polarization that corresponds to the standard polarization $\BW=e\BD\oplus f\BD$. Let $(\omega_{\psi,\BX,\BY},\CF(\BY))$ be the Schrodinger model of the Weil representation of $\BH$. According to Stone-von-Neumann theorem, for each $g\in\Sp(\BV\otimes_{\BD}\overline{\BW},\BB)$, there is an isomorphism $M=M(g)\in\GL(\CF(\BY))$ of $\CF(\BY)$ such that 
\[
M\circ\omega_{\psi,\BX,\BY}(h)=\omega_{\psi,\BX,\BY}(gh)\circ M.
\]
Define
\[
\Mp(\BV\otimes_{\BD}\overline{\BW})_{\psi,\CF(\BY)}^{\BC^{\times}}:=\{ (g,M)\in\Sp(\BV\otimes_{\BD} \overline{\BW},\BB )\times \GL(\CF(\BY)) \mid M\circ\omega_{\psi,\BX,\BY}(h)=\omega_{\psi,\BX,\BY}(gh)\circ M     \}.
\]
Let $c_{\BX}^{\psi}:\Sp(\BV\otimes_{\BD}\overline{\BW}\BB)\times\Sp(\BV\otimes_{\BD}\overline{\BW},\BB)\rightarrow\mu_8$ be the Leray cocycle, and define
\[
\Mp(\BV\otimes_{\BD}\overline{\BW})_{\psi,\BX}^{\BC^1}:=\Sp(\BV\otimes_{\BD}\overline{\BW})\times \BC^1
\]
as a set but with group structure given by the Leray cocycle $c_{\BX}^{\psi}$, namely
\[
(g,z)\cdot (g^{\prime},z^{\prime}):=(gg^{\prime},c_{\BX}^{\psi}(g,g^{\prime})zz^{\prime}).
\]
Let $\mathbf{r}=\mathbf{r}_{\BX,\BY}$ be the standard Weil operator; then we have the Weil representation
\[
\omega_{\psi,\BX,\BY}:\Mp(\BV\otimes_{\BD}\overline{\BW})_{\psi,\BX}^{\BC^1}\rightarrow \GL(\CF(\BY)):(g,z)\mapsto z\mathbf{r}(g),
\]
and the action of $\Mp(\BV\otimes_{\BD}\overline{\BW})_{\psi,\BX}^{\BC^1}$ is explicitly described as follows: for any $\Phi\in\CF(\BY)$, and $y\in\BY$, we have for 
\begin{itemize}
	\item \[\omega_{\psi,\BX,\BY} \left(
		g ,1\right) \Phi(y)= |\det g|_{\BV}^{\frac{1}{2}} \Phi(g^{-1}y) ,\;\forall g\in\RU(\BV,\BB_{\BV}), \]
		where $|\det g|_{\BV}=1$ except when $\BD=\RF$ and $\RU(\BV,\BW_{\BV})=\RO(\BV)$, in which case $|\det g|_{\BV}=\pm 1$ depending on whether $g$ is in $\SO(\BV)$ or not.
	\item \[
          \omega_{\psi,\BX,\BY}\left( 1,z   \right)\Phi(y)=z\Phi(y),\;\forall z\in\BC^1.
          \]
\end{itemize}
For the action of $\RU(\BW,\BB_{\BW})$, it is given as follows:
\begin{itemize}
	\item For $\BD=\RF$, we have $\RU(\BW,\BB_{\BW})=\SL_2$, and
	\subitem \[ \omega_{\psi,\BX,\BY}\left( \begin{pmatrix}
		\zeta & 0 \\ 0 &  \zeta^{-1}
	\end{pmatrix} ,1       \right) \Phi(y)=|\zeta|^{\frac{\dim_{\RF}\BV}{2}} \Phi(\zeta y),\;\forall \zeta\in\RF^{\times},   \]
	\subitem \[ \omega_{\psi,\BX,\BY}\left( \begin{pmatrix}
		1 & u \\ 0 & 1
	\end{pmatrix}  ,1   \right)\Phi(y) =  \psi\left( \frac{ u \cdot \BB(y,y)}{2}  \right)\Phi(y) ,\;\forall u \in\RF ,   \]
	\subitem \[  \omega_{\psi,\BX,\BY}\left( \begin{pmatrix}
		0 & -1 \\ 1 & 0
	\end{pmatrix}, 1    \right)  \Phi(y)= \int_{\BV}\Phi(x)\psi^{-1}(\BB(x,y) )\ud x .    \]
		\item For $\BD=\RF\times\RF$, we have $\RU(\BW,\BB_{\BW})=\GL_2$, and 
	\subitem \[\omega_{\psi,\BX,\BY} \left( \begin{pmatrix}
		t_1 & 0 \\ 0 & t_2
	\end{pmatrix},1     \right)\Phi(y)=\left|  \frac{t_1}{t_2} \right|^{\frac{ \dim_{\RF}\BV}{4}}\Phi\left(   y (t_1,t_2^{-1})  \right),\;\forall t_1,t_2\in\RF^{\times},
	\]
	\subitem \[  \omega_{\psi,\BX,\BY}\left( \begin{pmatrix}
		1 & u \\ 0 & 1
	\end{pmatrix}    ,1  \right) \Phi(y)=\psi\left(  \frac{u\cdot \tau(\BB(y,y))}{2}   \right) \Phi(y),\;\forall u\in\RF,     \]
	\subitem \[    \omega_{\psi,\BX,\BY}\left( \begin{pmatrix}
		0 & -1 \\ 1 & 0
	\end{pmatrix}, 1    \right)  \Phi(y)= \int_{\BV}\Phi(x)\psi^{-1}(\tau(\BB(x,y)) )\ud x .        \] 
	
	\item For $\BD=\Mat_2(\RF)$, we have $\RU(\BW,\BB_{\BW})=\RO_4$, and
	\subitem \[  
	\omega_{\psi,\BX,\BY}\left( \begin{pmatrix}
		g & 0\\ 0 & g
	\end{pmatrix} ,1 \right)\Phi( v\otimes \overline{f})=\Phi(vg\otimes\overline{f}),\;\forall g\in\SL_2(\RF)\subset\BD,
	\]
	\[
	\omega_{\psi,\BX,\BY}\left(   \begin{pmatrix}
		t_1 & 0 & 0 & 0 \\ 0 & t_2 & 0 & 0 \\  0 & 0 & t_2^{-1} & 0 \\ 0 & 0 & 0 & t_1^{-1}
	\end{pmatrix} ,1  \right) \Phi(v\otimes\overline{f})=|t_1t_2|^{\frac{\dim_{\RF}\BV}{4}}\Phi\left( v\begin{pmatrix}
		t_1 & 0 \\ 0 & t_2
	\end{pmatrix} \otimes\overline{f}   \right),\;\forall t_1,t_2\in\RF^{\times} ,    \]
	\subitem \[\omega_{\psi,\BX,\BY} \left( \begin{pmatrix}  1 & 0 & u & 0 \\ 0 & 1 & 0 & u \\ 0 & 0 & 1 & 0 \\ 0 & 0 & 0 & 1 \end{pmatrix},1  \right) \Phi(y)=\psi\left( \frac{n\cdot \tau (\BB_{\BV}(y,y))}{2}   \right) \Phi(y) ,\;\forall u \in\RF,     \]
	\subitem \[  \omega_{\psi,\BX,\BY} \left(  \begin{pmatrix}
		0 & 0 & -1 & 0 \\  0 & 0 & 0 & -1 \\ 1 & 0 & 0 & 0 \\ 0 & 1 & 0 & 0
	\end{pmatrix},1  \right)\Phi(y)=\int_{\BV}\Phi(x)\psi^{-1}(\tau(\BB(x,y))   )\ud x.    \]
\end{itemize}
For more details, see \cite{GKT25}.

In the following, for simplicity, write $\omega_{\psi}:=\omega_{\psi,\BX,\BY}$. In order to analyze the Hecke algebra action, the following remarks on the splitting are in order. In the case that $\BD=\RF$ and $\dim_{\RF}\BV$ is even, according to \cite[Theorem 11.6]{GKT25}, the map
	\[
	\SO(\BV)\times\SL_2\rightarrow\Mp(\BV\otimes_{\BD}\BW)_{\psi,\BX}^{\BC^1}:(g,h)\mapsto (g\otimes h,1)
	\]
	is a splitting of $\SO(\BV)\times\SL_2$ to $\Mp(\BV\otimes_{\BD}\BW)_{\psi,\BX}^{\BC^1}$, hence the Weil representation is an honest representation of $\SO(\BV)\times\SL_2$ by restriction. The same holds for the case when $\BD=\Mat_2(\RF)$.
	
	In the case that $\BD=\RF\times\RF$, this map $(g,h)\mapsto (g\otimes h,1)$ is also the Leray splitting according to \cite[Proposition 12.21]{GKT25}. 
	
	The only non-trivial case is when $\BD=\RF$ and $\dim_{\RF}\BV$ is odd. In this case, let
	\[
	\widetilde{\SL_2}=\Mp(\BW)_{\BX}^{\mu_2}:=\Sp(\BW)\times\mu_2=\SL_2\times\mu_2
	\]
	as a set, with group multiplication given by
	\[
	(g_1,\epsilon_1)(g_2,\epsilon_2):=(g_1g_2,\epsilon_1\epsilon_2 c_{\BX}(g_1,g_2)),
	\]
	where $c_{\BX}$ is the Rao cocycle. In the case of $\SL_2$, recall that we have the Rao function 
	\[
	\mathbf{x}\left( \begin{pmatrix}
		a & b \\ c & d
	\end{pmatrix} \right):=\left\{ \begin{array}{rcl}
c & \mbox{if}
& c\neq 0 ;\\ d & \mbox{if} & c=0,
\end{array}\right.
	\]
	and the Rao cocycle is given by 
	\[
	c_{\BX}(g_1,g_2):=(\mathbf{x}(g_1),\mathbf{x}(g_2))_{\RF}(-\mathbf{x}(g_1)\mathbf{x}(g_2),\mathbf{x}(g_1g_2))_{\RF},
	\]
	where $(\cdot,\cdot)_{\RF}$ is the Hilbert symbol of $\RF$. Note that this $c_{\BX}$ corresponds to the $\alpha$ used in \cite{Jac87}. Then we have a group homomorphism \cite[Theorem 11.6]{GKT25}
	\[
	\widetilde{\SL_2}\rightarrow\Mp(\BV\otimes_{\BD}\BW)_{\psi,\BX}^{\BC^1}:(g,\epsilon)\mapsto (1\otimes g, \epsilon\gamma(\mathbf{x}(g),\psi_{\frac{1}{2}})^{-\dim_{\RF}(\BV)} (\mathbf{x}(g),\det\BV)_{\RF}        ),
	\]
	where $\det(\BV)=(-1)^{\frac{\dim_{\RF}\BV-1}{2}}$ and $\gamma(\cdot,\psi_{\frac{1}{2} })$ is the normalized Weil index as in \cite[Section 3.3]{GKT25}. Following \cite{Jac87}, we employ the normalized cocycle $\beta$ for the metaplectic group as defined in \cite{Jac87}. To avoid confusion, we introduced the following notation to distinguish it from the conventions used in the previous sections. Write\[
\Mp_2:=\SL_2\times\{\pm1 \}
\]
as a set, with multiplication given by
\[
(g_1,\epsilon_1)(g_2,\epsilon_2):=(g_1g_2,\epsilon_1\epsilon_2\beta(g_1,g_2)),
\]
where $\beta(g_1,g_2):=\alpha(g_1,g_2)s(g_1)s(g_2)s(g_1 g_2)^{-1}$, and for $g=\begin{pmatrix}
	a & b \\ c & d
\end{pmatrix}$, we define
\begin{align*}
	s(g):=\left\{ \begin{array}{rcl}
 (c,d)_{\RF}& \mathrm{if}\; cd\neq 0\;\mathrm{and}\;\val(c)\equiv 1\mod 2 ;\\ 1 & \mathrm{otherwise}.
\end{array}\right.
\end{align*}
In this case $k\mapsto (k,1)$ is an isomorphism of $\SL_2(\Fo_{\RF})$ onto a subgroup $\RK^{\prime}$ of $\Mp_2$. And $\Mp_2\rightarrow\widetilde{\SL_2}:(g,\epsilon)\mapsto (g,\epsilon s(g))$ is a group isomorphism. We view $\SL_2$ as a subset of $\Mp_2$ via $g\mapsto (g,1)$. Let $\CH(\Mp_2,\RK^{\prime})$ be the spherical Hecke algebra of $\Mp_2$, which consists of $\RK^{\prime}$ bi-invariant functions with compact support. Then there is an isomorphism \[\CH(\PGL_2,\PGL_2(\Fo_{\RF}))\cong\CH(\Mp_2,\RK^{\prime}):h\mapsto h^{\prime}. \]
as fixed in \cite[Section 2]{Jac87}, which we review in Section \ref{ssec_jac} for the reader's convenience. We will identify $\RU(\BW,\BB_{\BW})=\SL_2$ as a subset of $\Mp_2$ using the set-theoretical section $g\mapsto (g,1)$.

\begin{rmk}
	For the sake of clarity, we provide a brief comparison between our notation and that employed in \cite{Jac87}. The additive character $\psi$ used in \cite{Jac87} should be $\psi_{-2}(x):=\psi(-2x)$ in our notaion. We take $\epsilon =1$ and $\eta=1$ in \cite{Jac87}. Note that the normalization of the Weil constant in \cite{Jac87} differs by a factor of $-2$ from the convention adopted in \cite{GKT25}. So the Weil constant in \cite{Jac87} is $\gamma(\cdot,\psi_{(-2)(-\frac{1}{2})})=\gamma(\cdot,\psi)$ in our notation, and the $\mu$ function in \cite{Jac87} is still $\gamma(\cdot,\psi)(\cdot,-1)_{\RF}= \gamma(\cdot,\psi)^{-1}$.
\end{rmk}

For $\RU(\BW,\BB_{\BW})$, write
\[
\RT:=\left\{ \begin{array}{rcl}
 \left\{ \begin{pmatrix}
 	\zeta & 0 \\ 0 & \zeta^{-1}
 \end{pmatrix} \mid \zeta\in\RF^{\times}    \right\} & \mbox{if} & \BD=\RF; \\ 
      \left\{ \begin{pmatrix}
 	t_1 & 0 \\ 0 & t_2 
 \end{pmatrix} \mid t_1,t_2\in\RF^{\times}    \right\} & \mbox{if} & \BD=\RF\times\RF ;    \\
 \left\{  \begin{pmatrix}  \xi  & 0 & 0 & 0 \\ 0 & 1 & 0 & 0 \\ 0 & 0 & 1 & 0 \\ 0 & 0 & 0 & \xi^{-1} \end{pmatrix} \mid \xi \in\RF^{\times}    \right\}   & \mbox{if} & \BD= \Mat_2(\RF).  \\
 \end{array}\right.
\]
\[
\RN:=\left\{ \begin{array}{rcl}
 \left\{ \begin{pmatrix}
 	1 & u \\ 0 & 1
 \end{pmatrix} \mid u \in\RF    \right\} & \mbox{if} & \BD=\RF,\;\RF\times\RF; \\ \left\{  \begin{pmatrix}  1 & 0 & u & 0 \\ 0 & 1 & 0 & u \\ 0 & 0 & 1 & 0 \\ 0 & 0 & 0 & 1 \end{pmatrix} \mid u\in\RF     \right\}   & \mbox{if} & \BD= \Mat_2(\RF) . \\
 \end{array}\right.
\]
We identify $\RN$ with $\BG_a$ in the obvious way. And finally write
\[
\Bw:=\left\{ \begin{array}{rcl}
 \begin{pmatrix}
 	0 & -1 \\ 1 & 0
 \end{pmatrix}  & \mbox{if} & \BD=\RF,\;\RF\times\RF; \\  \begin{pmatrix}  0 & 0 & -1 & 0 \\ 0 & 0 & 0 & -1 \\ 1 & 0 & 0 & 0 \\ 0 & 1 & 0 & 0 \end{pmatrix}    & \mbox{if} & \BD= \Mat_2(\RF).  \\
 \end{array}\right.
\]

\section{Transfer operators from Weil representations}\label{sec_trWeil}
In this section, let $(\BV,\BB_{\BV})$ be the standard Hermitian model of rank $n$ and $(\BW,\BB_{\BW})$ be the standard sekw-Hermitian model of rank $2$.
Fix the following $v_0\in\BV$ such that $\displaystyle{\frac{\tau(\BB_{\BV}(v_0,v_0))}{2}=1\in\RF}$:
\[
v_0:=\left    \{ \begin{array}{rcl}
   \begin{pmatrix}
   	1 \\ 0 \\ \vdots \\ 0 \\ 1
   \end{pmatrix} \in \RF^n & \mbox{if}& \;\BD=\RF ; \\ 
     \left( \begin{pmatrix}
   	0 \\ 0 \\ \vdots  \\ 1
   \end{pmatrix} , \begin{pmatrix}
   	0 \\ 0 \\ \vdots \\1 
   \end{pmatrix} \right) \in \RV \times\RV^{\vee}=\RF^n\times\RF^n & \mbox{if} & \;  \BD=\RF\times\RF; \\
    \begin{pmatrix}
   	0 \\ \vdots \\ 0 \\ \RI_2
   \end{pmatrix}\in\BD^n    & \mbox{if} & \;  \BD= \Mat_2(\RF).
 \end{array}\right.
\]
Set
\[
\RX:=\left\{v\in\BV\mid  \frac{\tau( \BB_{\BV}(v,v))}{2}=1 \right\}\subset \BV.
\]
Then $\RX$ is a closed subvariety of $\BV$.

\begin{lem}\label{lem_transitive}
	$\RU(\BV)$ acts transitively on $\RX$. Moreover, let
	\[
	v_0^{\perp}:=\{ v\in\BV\mid \BB_{\BV}(v,v_0)=0\}.
	\]
	The stabilizer of $v_0$ in $\RU(\BV)$ is $\RU(v_0^{\perp})$, where the Hermitian form on $v_0^{\perp}$ is the restriction of $\BB_{\BV}$ to $v_0^{\perp}\times v_0^{\perp}$. Therefore, we have
	\[
	\RX\cong \left    \{ \begin{array}{rcl}
   \SO_n/\SO_{n-1} & \mbox{if}& \;\BD=\RF ; \\ 
   \GL_n /\GL_{n-1} & \mbox{if}&  \;  \BD=\RF\times\RF; \\
   \Sp_{2n}/\Sp_{2n-2}   & \mbox{if}  &\;  \BD= \Mat_2(\RF).
 \end{array}\right.
	\]
\end{lem}

\begin{proof}
	Let us prove it case by case.

\begin{itemize}
	\item [(1)] If $\BD=\RF$, it is Witt's theorem. See \cite[Chapter 4]{Ser73}. In this case, we have $\RU(\BV)=\RO(\BV,\BB_{\BV})$ and $\RU(v_0^{\perp})=\RO(v_0^{\perp},\BB_{\BV})$, where we still write the restriction of $\BB_{\BV}$ to $v_0^{\perp}\times v_0^{\perp}$ by $\BB_{\BV}$. Since we can always find some element $g\in\RO_n\setminus\SO_n$ such that $gv_0=v_0$, $\SO_n$ acts on $\RX$ transitively as well.
	\item [(2)] If $\BD=\RF\times\RF$, and $\BV=\RF^n\times\RF^n$, then we have
	\begin{align*}
	\RX&=\{ (v,v^{\vee})\in\RF^n\times\RF^n\mid \BB( (v,v^{\vee}),(v,v^{\vee}))=(\langle v,v^{\vee}\rangle,\langle v,v^{\vee}\rangle) =(1,1)  \}\\
	&=\{(v,v^{\vee})\in\RF^n\times\RF^n  \mid \langle v,v^{\vee}\rangle =1  \}.
	\end{align*}
	For simplicity, write $e_i\in\RV=\RF^n$ for the standard vector with $1$ in the $i$-th component and $0$ otherwise for $1\leq i\leq n$. Write $\{e_i^{\vee}\}_{i=1}^n\subset \RV^{\vee}=\RF^n$ for the dual basis, which can be identified with $\{e_i\}_{i=1}^n$ since we identify $\RV^{\vee}=\RF^n$ via the standard pair between $\RF^n$ and $\RF^n$.
	For $g\in\GL_n$, if $ge_n=e_n$, then $g$ is of the form
	\[
	\begin{pmatrix}
		* & \cdots & *  & 0\\
		\vdots & \ddots & \vdots & \vdots \\
		* & \cdots & * & 0 \\
		*  & \cdots & * & 1
	\end{pmatrix}.
	\]
	Similarly, ${^t}g^{-1}e_n^{\vee}=e_n^{\vee}$ implies that $g$ is of the form
	\[
	\begin{pmatrix}
		* & \cdots & *  &  * \\
		\vdots & \ddots & \vdots & \vdots \\
		* & \cdots & * & *  \\
		0  & \cdots & 0 & 1
	\end{pmatrix}.
	\]
	Then $g(e_n,e_n^{\vee})=(ge_n,{^tg^{-1}}e_n^{\vee})=(e_n,e_n^{\vee})$ implies that $g\in\GL_{n-1}=\GL((e_n^{\vee})^{\perp} )\cong \RU ( (e_n,e_n^{\vee})^{\perp} )$. As for the transitivity, suppose $(v,v^{\vee})\in\BV$ such that $\langle v,v^{\vee}\rangle =1$. Let $\{e_1^{\prime},\cdots,e_{n-1}^{\prime}\}$ be a basis of $\{ w\in\RV\mid \langle w,v^{\vee} \rangle =0 \}$. Let $g\in\GL_n$ be such that $ge_n=v$ and $ge_i=e_i^{\prime}$ for $1\leq i\leq n-1$. Then we have for any $1\leq i\leq n-1$
	\[
	\langle e_i^{\prime},{^tg^{-1}}e_n^{\vee}\rangle =\langle g^{-1}e_i^{\prime},e_n^{\vee}\rangle =\langle e_i,e_n^{\vee}\rangle =0,
	\]
	and 
	\[
	\langle v,{^tg^{-1}}e_n^{\vee}\rangle =\langle g^{-1}v,e_n^{\vee}\rangle =\langle e_n,e_n^{\vee}\rangle =1.
	\]
	Therefore, we have $^tg^{-1}e_n^{\vee}=v^{\vee}$.
	\item [(3)] If $\BD=\Mat_2(\RF)$, write $v_0=\begin{pmatrix}
		0 \\ \vdots \\ 1 \\ 0
	\end{pmatrix}\otimes (1,0)+\begin{pmatrix}
		0 \\ \vdots \\ 0 \\ 1
	\end{pmatrix}\otimes (0,1)$ under the isomorphism $\BD^n\cong\RF^{2n}\otimes_{\RF}\RF^2$. Once we identitfy $\RU(\BV)\cong\Sp_{2n}$, then the stabilizer of $v_0$ is 
	\[
	\left\{  g\in\Sp_{2n}\mid g \begin{pmatrix}
		0 \\ \vdots \\ 1 \\ 0
	\end{pmatrix}=\begin{pmatrix}
		0 \\ \vdots \\ 1 \\ 0
	\end{pmatrix}, g\begin{pmatrix}
		0 \\ \vdots \\ 0 \\ 1
	\end{pmatrix}=\begin{pmatrix}
		0 \\ \vdots \\ 0 \\ 1
	\end{pmatrix} \right \}=\Sp_{2n-2},
	\]
	which is exactly $\RU(v_0^{\perp})$. To see the action is transitive, note that any $v\in\BV$ can be written as the form $v_1\otimes(1,0)+v_2\otimes(0,1)$, and 
	\[
	\BB(v,v)= {^tv_1}  \begin{pmatrix}
		J &  & \\ & \ddots  & \\ & & J
	\end{pmatrix} v_2,
	\]
	it suffices to show that for any $v_1,v_2\in\RF^{2n}$ such that $\displaystyle{{^tv_1}  \begin{pmatrix}
		J &  & \\ & \ddots  & \\ & & J
	\end{pmatrix} v_2}$=1, there is some $g\in\Sp_{2n}$ such that 
	\[
	g\begin{pmatrix}
		0 \\ \vdots \\ 1 \\ 0
	\end{pmatrix}=v_1, \; g\begin{pmatrix}
		0 \\ \vdots \\ 0 \\ 1
	\end{pmatrix}=v_2.
	\]
	According to \cite[Section 6.1]{Jac74}, we can extend $\{v_1,v_2\}$ to a basis $\{v_1,v_2,\cdots,v_{2n} \}$ of $\RF^{2n}$ such that the symplectic matrix with respect to this basis is also given by
	\[
	\begin{pmatrix}
		J &  & \\ & \ddots  & \\ & & J
	\end{pmatrix}.
	\] 
	Write $e_i$ for the $i$-th standard vector of $\RF^{2n}$ with $1$ in the $i$-th entry and $0$ otherwise. Consider $g\in\GL_{2n}$ such that $ge_{2n-1}=v_1,ge_{2n}=v_2$, and $ge_i=v_{i-2}$ for $3\leq i\leq 2n$, then we have $g\in\Sp_{2n}$.
	\end{itemize}

\end{proof}

There is a natural restriction map
\[
\res:\CF(\BV)\rightarrow\CF(\RX),
\]
which is surjective since $\RX\hookrightarrow\BV$ is a closed embedding. On the other hand, for any $g\in\RU(\BW,\BB_{\BW})$, $g$ acts on $\CF(\BV)$ via the Weil representation. Write 
\[
\Omega:\CF(\BV)\rightarrow\CC^{\infty}(\RU(\BW,\BB_{\BW})):\Phi\mapsto (g\mapsto \omega_{\psi}(g)\Phi(v_0)).
\]

The first easy lemma is

\begin{lem}\label{lem_RegOrb}
	Given a function $\Phi\in\CF(\BV)$, the following regulized integral
		\begin{align}
		\BT(\Phi)(t):=\int_{\RN}^*  \Omega(\Phi)(\Bw tu)   \psi^{-1}(u)\ud u:=\lim_{k\rightarrow+\infty}\int_{\Fp_{\RF}^{-k}}\Omega(\Phi)(\Bw t u) \psi^{-1}(u)\ud u
	\end{align}
	stabilize for sufficiently large $k$ and $t\in\RT$, where we identify $\RN$ with $\BG_a$ in the usual way. We denote this space of such functions by 
	\[
	\Orb_{\BV}(\RT).
	\]	
\end{lem}

\begin{proof}
First note that $\Omega(\Phi)(\Bw tu)$ differs $\omega_{\psi}(\Bw)\omega_{\psi}(t)\omega_{\psi}(u)\Phi(v_0)$ only by a factor of $t$, so let us consider 
\[
\lim_{k\rightarrow+\infty}\int_{\Fp_{\RF}^{-k}}\omega_{\psi}(\Bw)\omega_{\psi}(t)\omega_{\psi}(u)\Phi(v_0)\psi^{-1}(u)\ud u 
\]
in the following.

Note that for any $k\in\BZ$, we have
\begin{align*}
	&\int_{\Fp_{\RF}^{-k}} \omega_{\psi} (\Bw) \omega_{\psi} (t)\omega_{\psi}(u)  \Phi (v_0 )\psi^{-1}(u)\ud u \\
	&=\int_{\Fp_{\RF}^{-k}}\int_{\BV} \omega_{\psi} ( t)\omega_{\psi}( u)  \Phi (v)\psi^{-1}(\tau(\BB_{\BV}(v_0,v)))     \ud v \cdot \psi^{-1}(u)\ud u \\
	&=\int_{\Fp_{\RF}^{-k}} \int_{\BV} \omega_{\psi} ( t) \psi\left( \frac{u\cdot \tau(\BB_{\BV}(v,v))}{2} \right)   \Phi (v)\psi^{-1}(\tau(\BB_{\BV}(v_0,v)))     \ud v \cdot \psi^{-1}(u)\ud u
\end{align*}
Since $\displaystyle{v\mapsto \omega_{\psi} ( t) \psi\left( \frac{u\cdot \tau(\BB_{\BV}(v,v))}{2} \right)   \Phi (v)}$ is a Schwartz function on $\BV$, we are allowed to change the order in the integration. Moreover, in all cases, it can be checked easily that there are $x(t),y(t)\in\RF^{\times}$, $n(\BV)\in\BZ$ and a smooth map $z:\RT\times\BV\rightarrow\BV$ such that 
\[
 \omega_{\psi} ( t) \psi\left( \frac{u \cdot \tau(\BB_{\BV}(v,v))}{2} \right)   \Phi (v)=|y(t)|^{n(\BV)}\psi\left(  \frac{u \cdot x(t)\tau(\BB_{\BV}(v,v)  )}{2}  \right)\Phi(z(t,v)).
\]

Then the above integral becomes
\begin{align*}
	\int_{\BV} |y(t)|^{n(\BV)}\Phi(z(t,v))\psi^{-1}(\tau (\BB_{\BV}(v_0,v)))  \int_{\Fp_{\RF}^{-k}}\psi\left( u \left( \frac{ x(t)   \tau(\BB(v,v))}{2}  -1\right)\right) \ud u \ud v .
\end{align*}
Note that the inner integral
\[
\int_{\Fp_{\RF}^{-k}}\psi\left( u \left( \frac{ x(t)   \tau(\BB(v,v))}{2}  -1\right)\right) \ud u
\]
is non-zero only if 
\[
\frac{\tau (\BB(v,v))}{2}\in x(t)^{-1}(1+\Fp_{\RF}^{k}),
\]
in which case, we have
\[
\int_{\Fp_{\RF}^{-k}}\psi\left( u \left( \frac{ x(t)   \tau(\BB(v,v))}{2}  -1\right)\right) \ud u =q^{k}.
\]

Consider the map $\displaystyle{\BV\rightarrow\RF:w\mapsto  \frac{\tau(\BB(w,w))}{2}}$, which is a submersion outside the zero vector. According to \cite[Chapter 7]{Igu00}, the Haar measure $\ud v$ on $\BV$ and $\ud x$ on $\RF$ induce measures $|\omega_{x}|$ on each fiber 
\[
\BV_x:=\left\{ v\in\BV \mid \frac{\tau(\BB(v,v))}{2}=x \right\},
\]
such that for any $\Phi\in\BV$,
\[
\int_{\BV}\Phi(v)\ud v =\int_{\RF} \int_{\BV_x}\Phi(v)|\omega_x(v)| x.
\]
Then the desired integral is 
\begin{align*}
&q^k \int_{v \in \BV,\frac{\tau (\BB(v,v))}{2}\in x(t)^{-1}(1+\Fp_{\RF}^{k})
} |y(t)|^{n(\BV)}\Phi(z(t,v))\psi^{-1}(\tau(\BB_{\BV}(v_0,v))) \ud v\\
&=q^k\int_{x(t)^{-1}(1+\Fp_{\RF}^{k} )} \int_{\BV_x}|y(t)|^{n(\BV)}\Phi(z(t,v))\psi^{-1}(\tau(\BB_{\BV}(v_0,v)) )|\omega_x(v) | \ud x,
\end{align*}
and since $\Phi$ is locally constant of compact support, when $k$ is large enough, over $x(t)^{-1}(1+\Fp_{\RF})^k$, we have
\[
\int_{\BV_x}\Phi(z(t,v))\psi^{-1}(\tau(\BB_{\BV}(v_0,v)) )|\omega_x(v) |=\int_{\BV_{x(t)^{-1}}}\Phi(z(t,v))\psi^{-1}(\tau(\BB_{\BV}(v_0,v)) )|\omega_{x(t)^{-1}}(v) |,
\]
and hence
\begin{align*}
&q^k\int_{x(t)^{-1}(1+\Fp_{\RF}^{k-d} )} \int_{\BV_x}|y(t)|^{n(\BV)}\Phi(z(t,v))\psi^{-1}(\tau(\BB_{\BV}(v_0,v)) )|\omega_x(v) | \ud x\\
&=|x(t)|^{-1} |y(t)|^{n(\BV)} \int_{\BV_{x(t)^{-1}}}\Phi(z(t,v))\psi^{-1}(\tau(\BB_{\BV}(v_0,v)) )|\omega_{x(t)^{-1}}(v) |.
\end{align*}
\end{proof}

Now consider the $\RU(v_0^{\perp})$-invariant morphism
\[
\RX\rightarrow\BD:v \mapsto \BB(v_0,v).
\]
For any $x \in\BD$, the preimage of $x$ is 
\[
\RX_{x}=\{  v\in\RX\mid \BB(v_0,v)= x    \}=\left\{ x v_0 + u\mid u\in v_0^{\perp},\frac{\tau(\BB(u,u))}{2}=1-\nu(x)  \right\}.
\]
Outside $\BD^{\heartsuit}:=\{ x \in\BD \mid \nu (x )=1 \}$, the map $\RX\rightarrow\BD$ is a submersion and by Lemma \ref{lem_transitive} again, $\RU(v_0^{\perp})$ acts transitively on $\RX_{x}$. Write $\RX^{\heartsuit}$ to be the preimage of $\BD^{\heartsuit}$, then we have
\[
\RX^{\heartsuit}/\RU(v_0^{\perp})\cong\BD^{\heartsuit}.
\]

The following lemma is well-known for experts, but we also include the proof here for completeness.
\begin{lem}\label{lem_git}
	\begin{itemize}
		\item [(1)] When $\BD=\RF$, 
		\begin{align*}
			\pi:\RX\cong \SO_n/\SO_{n-1}\rightarrow\BA^1
		\end{align*}
		can be identified with the canonical quotient map \[\SO_n/\SO_{n-1}\rightarrow\SO_{n-1}\rotatebox{60}{$\sslash$}\SO_n/\SO_{n-1}.\]
		\item [(2)] When $\BD=\RF\times\RF$, $\RX\rightarrow\BD$ induces morphisms
\[
\begin{tikzcd}
\GL_n/\GL_{n-1} \arrow[r, ""] \arrow[d, ""]
&\BA^2 \arrow[d, ""] \\
\PGL_n/\GL_{n-1} \arrow[r, "\pi"]
&\BA^1
\end{tikzcd},
\]
and $\pi$ can be identified with the canonical quotient map \[
\PGL_n/\GL_{n-1}\rightarrow\GL_{n-1}\rotatebox{60}{$\sslash$}\PGL_n/\GL_{n-1}.\]
\item [(3)] When $\BD=\Mat_2(\RF)$, $\RX\rightarrow\BD$ induces morphisms
\[
\begin{tikzcd}
\Sp_{2n}/\Sp_{2n-2} \arrow[r, ""] \arrow[d, ""]
&\Mat_2 \arrow[d, "\det"] \\
\Sp_{2n}/\Sp_{2n-2}\times\Sp_2 \arrow[r, "\pi"]
&\BA^1
\end{tikzcd},
\]
and $\pi$ can be identified with the canonical quotient map \[
\Sp_{2n}/\Sp_{2n-2}\times\Sp_2 \rightarrow\Sp_{2n-2}\times\Sp_2 \rotatebox{60}{$\sslash$}\Sp_{2n}/\Sp_{2n-2}\times\Sp_2.\]
	\end{itemize}

\end{lem}

\begin{proof}
 
 Write 
 \begin{align*}
 	(\RG,\RH):=\left\{ \begin{array}{rcl}
(\SO_n,\SO_{n-1}) & \mbox{if} & \BD=\RF; \\ 
     (\PGL_n,\GL_{n-1}) & \mbox{if} & \BD=\RF\times\RF ;    \\
(\Sp_{2n},\Sp_{2n-2}\times\Sp_2)  & \mbox{if} & \BD= \Mat_2(\RF).  \\
 \end{array}\right.
 \end{align*}
 And $\RX=\RG/\RH$ temporarily in this lemma. It has been proved in \cite{Sak21} that we have an isomorphism
 \begin{align*}
 	\RH\rotatebox{60}{$\sslash$}\RG/\RH=\RA_{\RX}\sslash\RW_{\RX},
 \end{align*}
 where $\RA_{\RX}$ is the universal Cartan of $\RX$ and $\RW_{\RX}$ is the little Weyl group, which acts on $\RA_{\RX}$ by inversion. Therefore, in all cases it suffices to show that the restriction of $\pi$ to the universal Cartan is a generator of $\RF[\RA_{\RX}]^{\RW_{\RX}}$. Write $\RB$ for the Borel subgrou pf $\RG$ consisting of upper triangular matrices in the case that $\BD=\RF$ and $\RF\times\RF$, and let $\RB$ be the Borel fixing the complete anisotropic flag of $\{ \{e_1\}\subset\{e_1+e_3\}\subset \cdots \subset \{e_1+\cdots+e_{2n-1}\}\}$ in the case $\BD=\Mat_2(\RF)$.
 
 For $(1)$, we first claim that $\RB v_0\subset\RX$ is an open $\RB$-orbit. In fact, it is easy to compute that the stabilizers of $v_0$ in $\RB$ are of the form
 \begin{align*}
 	\begin{pmatrix}
 		1 & 0 & 0 \\ 0 & g & 0 \\ 
 		0 & 0 & 1
 	\end{pmatrix},
 \end{align*}
 where $g$ is an element in the Borel subgroup of $\SO(n-2)$. Then, when $n=2k$ is even, we have
 \begin{align*}
 	\dim\RB v_0=k^2-(k-1)^2=2k-1=\dim\RX,
 \end{align*}
 and when $n=2k+1$ is odd,
 \begin{align*}
 	\dim\RB v_0=k(k+1)-(k-1)k=2k=\dim\RX.
 \end{align*}
 Then the universal Cartan can be identified with 
 \begin{align*}
 	\RA_{\RX}=\{ \diag\{t,1,\cdots,1,t^{-1}\}  \},
 \end{align*}
 and the restriction of $\pi$ to $ \diag\{t,1,\cdots,1,t^{-1}\}  v_0$ is exactly $t+t^{-1}$.
 
 For $(2)$, let $x_0:=\begin{pmatrix}
 	1 & 0 & -1 \\ 0 & \RI_{n-2}& 0 \\ 1 & 0 & 1
 \end{pmatrix}$, similarly one can easily compute that 
 \begin{align*}
 	\dim\RB v_0=2n-2=\dim\RX,
 \end{align*}
then we can identify the universal Cartan as 
\begin{align*}
	\RA_{\RX}=x_0^{-1}\{ \diag\{t,1,\cdots,1, 1\}\}x_0,  
\end{align*}
and the restriction of $\pi$ to $x_0^{-1}\diag\{t,1,\cdots,1, 1\}x_0$ is 
\begin{align*}
	\tau\left( \BB \left(   \left( \begin{pmatrix}
		-t \\ 0 \\ \vdots\\ 0 \\1
	\end{pmatrix}, \begin{pmatrix}
		-\frac{1}{2t} \\ 0 \\ \vdots\\ 0 \\ \frac{1}{2}
	\end{pmatrix}\right),  \left( \begin{pmatrix}
		-1 \\ 0 \\ \vdots\\ 0 \\1
	\end{pmatrix}, \begin{pmatrix}
		-\frac{1}{2} \\ 0 \\ \vdots\\ 0 \\ \frac{1}{2}
	\end{pmatrix}\right)   \right) \right)=\left(\frac{t}{2}+\frac{1}{2}\right)\left(\frac{1}{2t}+\frac{1}{2} \right)=\frac{t+t^{-1}+2}{4}.
\end{align*}

For $(3)$, note that $x_0=\frac{1}{4}\begin{pmatrix}
	 -1 & 0 & 0  &   1 & 0 \\ 0 & -1 & 0 & 0 & 1 \\0 & 0 & \RI_{2n-4} & 0 & 0 \\ 1 & 0 & 0 & 1 & 0 \\0 & 1 & 0 & 0 & 1
\end{pmatrix}$ is in the open $\RB$-orbit since $x_0$ sending the subspace generated by $\{e_n,f_n\}$ to the space generated by $\{e_1+e_n,f_1+f_n\}$, and the stabilizer in $\RB$ preserving the space generated by $\{e_1+e_n,f_1+f_n\}$ has dimension $n^2+n-(4n-4)$.

If we identify the universal Cartan as $x_0^{-1}\diag\{t,t^{-1},1,\cdots,1\}x_0$, then it is clear that the restriction of $\pi$ is given by
\begin{align*}
	\frac{1}{4}\det\left( \begin{pmatrix}
		t^{-1} & 0 \\ 0 & t
	\end{pmatrix}+\begin{pmatrix} 1 & 0 \\ 0 & 1\end{pmatrix}    \right)=\frac{t+t^{-1}+2}{4}.
\end{align*}

\end{proof}

According to \cite[Chapter 7]{Igu00} again, the measure $|\omega_1|$ on $\RX=\BV_1$ and the Haar measure on $\BD$ give rise to measures $|\mu_{x}|$ on fibers $\RX_{x}$.
\begin{dfn}
	Let $\Phi$ be a Schwartz function on $\RX$. We define its orbital integral with respect to the action of $\RU(v_0^{\perp})$, which is a function on $\BD^{\heartsuit}$, to be 
	\[
	\CI(\Phi):\BD^{\heartsuit}\ni x \mapsto \left\{ \begin{array}{rcl}
\int_{\RX_{ 4x-2  }}\Phi(v)|\mu_{4x-2}(v)| & \mbox{for} \;\BD=\RF \;\mbox{and} \;\dim_{\RF}\BV\equiv 1\mod 2
\\ \int_{\RX_{x}}\Phi(v)|\mu_{x}(v)| & \mbox{otherwise}.
\end{array}\right.
	\]
	We denote this space of orbital integrals by $\CF(\FX)$.
\end{dfn}

\begin{rmk}
	We choose $\displaystyle{x=\frac{\BB(v_0,v)}{4}+\frac{1}{2}}$, as our coordinates in the case $\BD=\RF$ and $\dim_{\RF}\BV$ is odd, to make it compatible with the coordinates chosen in Theorem \ref{thm_TrSak}.
\end{rmk}

\begin{prp}\label{prp_TranOp}
	The morphism $\CF(\BV)\rightarrow \Orb_{\BV}(\RT)$ factors through 
	\[
	\CF(\BV)\rightarrow\CF(\RX)\rightarrow\CF(\FX),
	\]
	i.e., we have a commutative diagram
\[
\begin{tikzcd}
\CF(\BV) \arrow[rr, "\BT"] \arrow[d, "\res"] &  & \Orb_{\BV}(\RT) \\
\CF(\RX)\arrow[d, "\CI"]                      &  &   \\
\CF(\FX) \arrow[rruu, "\widetilde{\BT}"]              &  &  
\end{tikzcd}
\]
Moreover, the morphism $\widetilde{\BT}:\CF(\FX)\rightarrow\Orb_{\BV}(\RT)$ is given by the following formulae:
\begin{itemize}
	\item [(1)] if $\BD=\RF$ and $\dim_{\RF}\BV$ is even, we have
	\[
	\widetilde{\BT}(\phi) \begin{pmatrix}
		\zeta & 0 \\  0 & \zeta^{-1}
	\end{pmatrix}  =|\zeta|^{-\frac{\dim_{\RF}\BV}{2}}\int_{\RF^{\heartsuit}}\phi(x)\psi^{-1}(\zeta^{-1} x)\ud x;
	\]
	\item [(2)]
	if $\BD=\RF$ and $\dim_{\RF}\BV$ is odd, we have
	\[
	\widetilde{\BT}(\phi) \begin{pmatrix}
		\zeta & 0 \\  0 & \zeta^{-1}
	\end{pmatrix} =\gamma(t,\psi)|\zeta|^{-\frac{\dim_{\RF}\BV}{2}}\psi^{-1}\left( \frac{-2}{\zeta} \right)\int_{\RF^{\heartsuit}} \phi(x)\psi^{-1}(4\zeta^{-1}x)\ud x;
	\]
	\item [(3)] if $\BD=\RF\times\RF$, wex have
	\[
	\widetilde{\BT}(\phi)\left( \begin{pmatrix}
		t_1 & 0 \\  0 & t_2
	\end{pmatrix} \right) =  \left|\frac{t_1}{ t_2}\right|^{-\frac{\dim_{\RF}\BV}{4}}  \int_{ (\RF\times\RF)^{\heartsuit}}\phi(x_1,x_2)  \psi^{-1}(t_2x_1+t_1^{-1}x_2)\ud x_1\ud x_2.;	\]
	\item [(4)] if $\BD=\Mat_2(\RF)$, we have
	\[
	\widetilde{\BT}(\phi)  \begin{pmatrix}  \xi  & 0 & 0 & 0 \\ 0 & 1 & 0 & 0 \\ 0 & 0 & 1 & 0 \\ 0 & 0 & 0 & \xi^{-1} \end{pmatrix}  =|\xi|^{-\frac{\dim_{\RF}\BV}{4}}\int_{\Mat_2(\RF)^{\heartsuit}}\phi(x)\psi^{-1}\left( \tau\left(  x\begin{pmatrix}
			\xi^{-1} & 0 \\ 0 & 1
		\end{pmatrix}\right) \right)\ud x.
	\]
\end{itemize}
\end{prp}

\begin{proof}
	When $\BD=\RF$ and $\dim_{\RF}\BV$ is even, this is \cite[Proposition 11.1]{GW21}. We also include the proof here for the reader's convenience. Let $\Phi\in\CF(\BV)$, for any integer $k$, we have
\begin{align*}
  \BT(\Phi)(t)&=\int_{\Fp_{\RF}^{-k}}\omega_{\psi}\left(\Bw t\begin{pmatrix}
			1 & u\\ 0 & 1
		\end{pmatrix} \right)\Phi(v_0)\psi^{-1}(u)\ud u \\
		&=\int_{\Fp_{\RF}^{-k}}\int_{\BV}\omega_{\psi}\left( t\begin{pmatrix}
		1 & u \\ 0 & 1
	\end{pmatrix} \right)\Phi (v)\psi^{-1}(\BB(v_0,v))     \ud v \cdot \psi^{-1}(u)\ud u. \\
\end{align*}
Write $t=\begin{pmatrix}
	\zeta & 0 \\ 0 & \zeta^{-1}
\end{pmatrix}$, we have
\begin{align*}
\omega_{\psi} \left( \begin{pmatrix}
	\zeta & 0 \\ 0 & \zeta^{-1}
\end{pmatrix}\begin{pmatrix}
		1 & u \\ 0 & 1
	\end{pmatrix} \right)\Phi (v)&=|\zeta|^{\frac{\dim_{\RF}\BV}{2}}\omega_{\psi} \left( \begin{pmatrix}
		1 & u \\ 0 & 1
	\end{pmatrix} \right)\Phi (\zeta v)\\
	&=|\zeta |^{\frac{\dim_{\RF}\BV}{2}}\psi\left(  \frac{u \cdot  \BB(\zeta v,\zeta v)}{2}       \right)\Phi(\zeta v).
\end{align*}
Then we have
	\begin{align*}
 & \int_{\Fp_{\RF}^{-k}}\omega_{\psi}\left(\Bw t \begin{pmatrix}
			1 & u \\ 0 & 1
		\end{pmatrix} \right)\Phi(v_0)\psi^{-1}(u)\ud u \\
		&=\int_{\Fp_{\RF}^{-k}}|\zeta|^{\frac{\dim_{\RF}\BV}{2}} \int_{\BV}  \psi\left(  \frac{u \cdot \BB(\zeta v,\zeta v)}{2}       \right)\Phi(\zeta v )\psi^{-1}(\BB(v_0,v))     \ud v \cdot \psi^{-1}(u)\ud u \\
		&=|\zeta|^{-\frac{\dim_{\RF}\BV}{2}} \int_{\BV}\Phi( v)\psi^{-1}(\BB(v_0,\zeta^{-1}v)) \int_{\Fp_{\RF}^{-k}}\psi\left( u \left(\frac{ \BB( v, v)}{2} -1\right)      \right)\ud u\ud v. 
\end{align*}
Note that
	\[
	 \int_{\Fp_{\RF}^{-k}}\psi\left( u \left(\frac{ (\BB( v, v)}{2} -1\right)      \right)\ud n =\left\{ \begin{array}{rcl}
q^{k} &
\mathrm{if}\;\frac{ \BB( v, v)}{2} \in 	1+\Fp_{\RF}^{k}  \\ 0 & \mathrm{otherwise}. 
\end{array}\right.
	\]
Then, when $k$ is large enough, we have
	\begin{align*}
 & \int_{\Fp_{\RF}^{-k}}\omega_{\psi}(\Bw )\omega_{\psi}(t)\omega_{\psi}\left( \begin{pmatrix}
			1 & u \\ 0 & 1
		\end{pmatrix} \right)\Phi(v_0)\psi^{-1}(u)\ud u \\
		&=  |\zeta|^{-\frac{\dim_{\RF}\BV}{2}} \int_{\RX}\Phi( v)\psi^{-1}(\BB(v_0,\zeta^{-1}v))|\omega_1(v)| \\
		&= |\zeta|^{-\frac{\dim_{\RF}\BV}{2}}\int_{\RF^{\heartsuit}}\int_{\RX_{x}}\Phi(v) |\mu_{x}(v)| \psi^{-1}(\zeta^{-1} x )\ud x \\
		&= |\zeta|^{-\frac{\dim_{\RF}\BV}{2}}\int_{\RF^{\heartsuit}} \CI(\res(\Phi))(x )\psi^{-1}(\zeta^{-1} x)\ud x .
\end{align*}

In the case that $\dim_{\RF}\BV$ is odd, using our coordinate, we have
\begin{align*}
	\widetilde{\BT}(\phi) \begin{pmatrix}
		\zeta & 0 \\  0 & \zeta^{-1}
	\end{pmatrix} &=\gamma(t,\psi)|\zeta|^{-\frac{\dim_{\RF}\BV}{2}}\int_{\RF^{\heartsuit}}\CI(\res(\Phi))(x)\psi^{-1}(\zeta^{-1}(4x-2))\ud x \\
	&=\gamma(t,\psi)|\zeta|^{-\frac{\dim_{\RF}\BV}{2}}\psi^{-1}\left( \frac{-2}{\zeta} \right)\int_{\RF^{\heartsuit}} \CI(\res(\Phi))(x)\psi^{-1}(4\zeta^{-1}x)\ud x.
\end{align*}

When $\BD=\RF\times\RF$, for $\Phi\in\CF(\BV)$ and $t=\begin{pmatrix}
	t_1 & 0   \\ 0 &  t_2 
\end{pmatrix}\in\RT$, we have
\begin{align*}
	 & \int_{\Fp_{\RF}^{-k}}\omega_{\psi}\left( \Bw t\begin{pmatrix}
			1 & u \\ 0 & 1
		\end{pmatrix} \right)\Phi(v_0)\psi^{-1}(u)\ud u \\
		&=\int_{\Fp_{\RF}^{-k}}\int_{\BV} \omega_{\psi} \left(\begin{pmatrix}
	t_1 & 0 \\ 0 & t_2 
\end{pmatrix}\begin{pmatrix}
		1 & u \\ 0 & 1
	\end{pmatrix} \right)\Phi (v )\psi^{-1}(\tau(\BB(v_0,v)))     \ud v \cdot \psi^{-1}(u)\ud u \\
\end{align*}
	Write $\BV=\RF^n\times\RF^n$ and its elements as $(v,v^{\vee})$, then we have
	\[
	\omega_{\psi} \left(\begin{pmatrix}
	t_1 & 0 \\ 0 & t_2 
\end{pmatrix}  \begin{pmatrix}
		1 & u \\ 0 & 1
	\end{pmatrix} \right)\Phi ((v,v^{\vee}))=\left| \frac{t_1}{t_2}  \right|^{\frac{\dim_{\RF}\BV}{4}}\psi( n\langle t_1 v,t_2^{-1}v^{\vee}\rangle )\Phi((t_1v,t_2^{-1}v^{\vee})),
	\]
	hence
	\begin{align*}
	 &  \int_{\Fp_{\RF}^{-k}}\omega_{\psi}\left( \Bw t\begin{pmatrix}
			1 & u \\ 0 & 1
		\end{pmatrix} \right)\Phi(v_0)\psi^{-1}(u)\ud u \\
		&=\int_{\Fp_{\RF}^{-k}}\int_{\RF^n\times\RF^n}        \left|\frac{t_1}{ t_2}\right|^{\frac{n}{2}}\psi( u \langle v,v^{\vee}\rangle )\Phi((v,v^{\vee}))\cdot  \psi^{-1}(\tau(\BB(v_0,(t_1^{-1}v,t_2 v^{\vee}))))  \ud v\ud v^{\vee}\cdot \psi^{-1}(u)\ud u \\
		&=  \left|\frac{t_1}{ t_2}\right|^{-\frac{n}{2}} \int_{\RF^n\times\RF^n}\Phi((v,v^{\vee})) \psi^{-1}( \langle t_1^{-1}v,e_n^{\vee }\rangle +\langle e_n,t_2v^{\vee}\rangle  ) \int_{\Fp_{\RF}^{-k}}     \psi\left(u\left(   \langle v,v^{\vee}\rangle -1\right)  \right)\ud u  \ud v \ud v^{\vee} 
\end{align*}
Since we have
	\[
	 \int_{\Fp_{\RF}^{-k}}\psi\left(u  \left(   \langle v,v^{\vee}\rangle -1\right)  \right)\ud u  =\left\{ \begin{array}{rcl}
q^k  &
\mathrm{if}\; \langle v ,v^{\vee} \rangle \in 	1+\Fp_{\RF}^{k}  \\ 0 & \mathrm{otherwise},
\end{array}\right.
	\]
hence when $k$ is large enough,
\begin{align*}
	 &  \int_{\Fp_{\RF}^{-k}}\omega_{\psi}\left( \Bw t\begin{pmatrix}
			1 & u \\ 0 & 1
		\end{pmatrix} \right)\Phi(v_0)\psi^{-1}(u)\ud u\\
		&=  \left|\frac{t_1}{ t_2}\right|^{-\frac{n}{2}}  \int_{\RX}\Phi((v,v^{\vee})) \psi^{-1}( \langle t_1^{-1}v,e_n^{\vee }\rangle +\langle e_n,t_2v^{\vee}\rangle  )   |\omega_1((v,v^{\vee}))|\\
		&=\left|\frac{t_1}{ t_2}\right|^{-\frac{n}{2}}  \int_{(\RF\times\RF)^{\heartsuit}}\CI(\res(\Phi))(x_1,x_2)  \psi^{-1}(t_2x_1+t_1^{-1}x_2)\ud x_1\ud x_2.
\end{align*}

If $\BD=\Mat_2(\RF)$, for $\Phi\in\CF(\BV)$ and $t= \begin{pmatrix}  \xi  & 0 & 0 & 0 \\ 0 & 1 & 0 & 0 \\ 0 & 0 & 1 & 0 \\ 0 & 0 & 0 & \xi^{-1} \end{pmatrix} \in\RT$, we have
\begin{align*}
	 &  \int_{\Fp_{\RF}^{-k}}\omega_{\psi}\left( \Bw t\begin{pmatrix}
			1 & u \\ 0 & 1
		\end{pmatrix} \right)\Phi(v_0)\psi^{-1}(u)\ud u\\
		&=\int_{\Fp_{\RF}^{-k}}\int_{\BV} \omega_{\psi,\BX,\BY}\left(  \begin{pmatrix}  \xi  & 0 & 0 & 0 \\ 0 & 1 & 0 & 0 \\ 0 & 0 & 1 & 0 \\ 0 & 0 & 0 & \xi^{-1} \end{pmatrix} n \right)\Phi (v )\psi^{-1}(\tau(\BB(v_0,v)))     \ud v \cdot \psi^{-1}(u)\ud u \\
\end{align*}
Since we have
	\[
	\omega_{\psi}\left(tu\right)\Phi (v)=\left|\xi \right|^{\frac{\dim_{\RF}\BV}{4}}\psi\left( \frac{u \cdot \tau\left(\BB_{\BV}\left(v\begin{pmatrix}
		\xi  & 0 \\ 0 & 1
	\end{pmatrix},v\begin{pmatrix}
		\xi & 0 \\ 0 & 1
	\end{pmatrix} \right)\right)  }{2}  \right)\Phi\left(v \begin{pmatrix}
		\xi & 0 \\ 0 & 1
	\end{pmatrix} \right),
	\]
	then
	\small{
	\begin{align*}
	 &  \int_{\Fp_{\RF}^{-k}}\omega_{\psi}\left( \Bw t\begin{pmatrix}
			1 & u \\ 0 & 1
		\end{pmatrix} \right)\Phi(v_0)\psi^{-1}(u)\ud u \\
	 &=\int_{\Fp_{\RF}^{-k}}\int_{\BV}        \left|\xi \right|^{-\frac{\dim_{\RF}\BV}{4}}\psi\left(\frac{u\cdot \tau\left(\BB_{\BV}(v,v)\right)}{2}  \right)\Phi(v)  \cdot  \psi^{-1}\left(\tau\left(\BB\left(v_0, v \begin{pmatrix}
	 	\xi ^{-1} & 0 \\ 0 & 1
	 \end{pmatrix}  \right)\right) \right) \ud v \cdot \psi^{-1}(u)\ud u \\
		&=  \left|\xi \right|^{-\frac{\dim_{\RF}\BV}{4}} \int_{\BV}\Phi(v) \psi^{-1}\left( \tau \left(\BB_{\BV}(v_0,v) \begin{pmatrix}
			\xi^{-1} & 0 \\ 0 & 1
		\end{pmatrix}  \right) \right) \int_{\Fp_{\RF}^{-k}}     \psi\left(u\left(   \frac{\tau(\BB_{\BV}(v,v))}{2} -1\right)  \right)\ud u  \ud v \end{align*}}
Again
	\[
	 \int_{\Fp_{\RF}^{-k}}\psi\left(u\left(   \frac{\tau(\BB_{\BV}(v,v))}{2} -1\right)  \right)\ud u  =\left\{ \begin{array}{rcl}
q^k&
\mathrm{if}\; \frac{\tau(\BB_{\BV}(v,v))}{2} \in 	1+\Fp_{\RF}^{k}  \\ 0 & \mathrm{otherwise},
\end{array}\right.
	\]
when $k$ is large enough,
\begin{align*}
	 &  \int_{\Fp_{\RF}^{-k}}\omega_{\psi}\left( \Bw t\begin{pmatrix}
			1 & u \\ 0 & 1
		\end{pmatrix} \right)\Phi(v_0)\psi^{-1}(u)\ud u\\
		&=  \left|\xi\right|^{-\frac{\dim_{\RF}\BV}{4}}  \int_{\RX}\Phi(v)) \psi^{-1}\left( \tau \left(\BB_{\BV}(v_0,v)\begin{pmatrix}
			\xi^{-1} & 0 \\ 0 & 1
		\end{pmatrix}  \right) \right) |\omega_1(v)|\\
		&=\left|\xi\right|^{-\frac{\dim_{\RF}\BV}{4}}  \int_{\Mat_2(\RF)^{\heartsuit}}\CI(\res(\Phi))(x)  \psi^{-1}\left(\tau \left(  x  \begin{pmatrix}
			\xi^{-1} & 0 \\ 0 & 1
		\end{pmatrix} \right)\right)\ud x.
		\end{align*}

\end{proof}

To facilitate a comparison with the transfer operators established in Theorem \ref{thm_TrSak}, we examine each case in detail. We begin with the case $\BD=\RF$, which is the most direct. Let $\CS(\FX)=\CF(\FX)\ud x$ and $\CS_{\BV}(\RT):=\Orb_{\BV}(\RT)|\zeta^2|\ud^{\times}\zeta$ denote the respective spaces of measures. Retaining the notation $\overline{\BT}$ for the extension to $\CS(\FX)$.

\begin{prp}\label{prp_tr_weil_d}
	In the case where $\dim\BV$ is even, let $\CT$ denote the transfer established in Theorem \ref{thm_TrSak}. Then, on the space of push-forward measures, we have the identity
    \[
\widetilde{\BT}\circ\CT=\Id.
\]  
\end{prp}

\begin{proof}
According to Proposition \ref{prp_TranOp}, for $\phi\in\CF(\FX)$, we have
\begin{align}\label{eq_tr_d_n}
	\widetilde{\BT}(\phi)(\zeta)= |\zeta |^{-\frac{\dim_{\BD}\BV}{2}}\int_{\BD^{\heartsuit}}\phi(x)\psi^{-1}(\zeta^{-1}x)\ud x 
	= (|\cdot|^{-n}\cdot  \iota \circ\BF_{\psi}(\phi))(\zeta).
	\end{align}
\end{proof}

\begin{prp}\label{prp_tr_weil_b}
   When $\dim_{\RF}\BV$ is odd, we have
    \[
    \widetilde{\BT}\circ\CT=\gamma(\cdot,\psi) \psi\left(  \frac{2}{\cdot} \right)\cdot \left(|\cdot|^{\frac{1}{2}}\circ \circ\BF_{\psi^{-1}}\circ\iota\circ|\cdot|\right)\left( \frac{\cdot}{4}  \right).
    \]
\end{prp}

\begin{proof}
   The result is also a direct consequence of Proposition \ref{prp_TranOp}.
    \end{proof}

Next, consider the case that $\BD=\RF\times\RF$. Consider the following $\BG_m$-action on $\BA^2$:
\[
z(x_1,x_2):=(zx_1,z^{-1}x_2),\;\forall z\in\BG_m,\;(x_1,x_2)\in\BA^2.
\]
It is clear that $\BA^2\sslash\BG_m=\BA^1$ via the map $(x_1,x_2)\mapsto x_1x_2$.\begin{lem}\label{lem_abs_a_lft}
	For any $\phi\in\CF(\FX)$, $x_1,x_2\in\RF^{\times}$ such that $x_1x_2\neq 0,1$, 
	\[
	\int_{\BG_m}\phi(z^{-1}x_1,zx_2 )   \ud z
	\] 
	is absolutely converget.\end{lem}

\begin{proof}
	Let $\Phi\in\CF(\RX)$ and $\phi=\CI(\Phi)$. Note that for fixed $(x_1,x_2)\in(\RF\times\RF)^{\heartsuit}$ such that $x_1x_2\neq 0,1$, \[\phi(z^{-1}x_1,zx_2)\neq 0\] only when the support of $\Phi$ has a non-empty intersection with $\RX_{(z^{-1}x_1,zx_2)}$. Since the support of $\Phi$ is compact, and $\RX\rightarrow\RF$ is a submersion away from $0,1$, the integral\[
\int_{\BG_m}\phi(z^{-1}x_1,zx_2 )   \ud z
\]
is actually over some compact subsect of $\RF^{\times}$, hence is absolutely convergent. 
\end{proof}

On the other hand, consider the center integrals of elements in $\Orb_{\BV}(\RT)$, we also have
\begin{lem}\label{lem_conv_a}
	For $f\in \Orb_{\BV}(\RT) $,
\[
\int_{\BG_m}   f \left(  \begin{pmatrix}
	\xi & 0 \\ 0 & 1
\end{pmatrix}  \begin{pmatrix}
	z & 0 \\ 0 & z
\end{pmatrix}  \right)\ud z
\]
is absolutely convergent for $\xi\in\RF^{\times}$.
\end{lem}

\begin{proof}
	According to the proof of Proposition \ref{prp_TranOp}, the choice of $k$ in the regularized integral is independent of $t$ and only depends on the function $\Phi$. Let us fix such a $k$. We first claim that when $|z|$ is large enough, 
	\[
	 f \left(  \begin{pmatrix}
	\xi & 0 \\ 0 & 1
\end{pmatrix}  \begin{pmatrix}
	z & 0 \\ 0 & z
\end{pmatrix}  \right)=0.
	\]
	In fact, we may assume $f=\widetilde{\BT}(\phi)$ for some $\phi=\CI(\res(\Phi))$, and $\Phi=\Phi_1\otimes\cdots\Phi_n\otimes\Psi_1\otimes\Psi_n$ is a pure tensor. Then we have
	\begin{align*}
		& f \left(  \begin{pmatrix}
	\xi & 0 \\ 0 & 1
\end{pmatrix}  \begin{pmatrix}
	z & 0 \\ 0 & z
\end{pmatrix}  \right)\\
&=\int_{\Fp_{\RF}^{-k}}\int_{\RF^n\times\RF^n}        \left|\xi\right|^{-\frac{\dim_{\BD}\BV}{2}}\psi( u(x_1y_1+\cdots x_ny_n) )\Phi_1(x_1)\cdots\Phi_n(x_n)\Psi_1(y_1)\cdots\Psi_n(y_n)\\
& \cdot  \psi^{-1}(\xi^{-1}z^{-1}x_n+zy_n)  \ud x_1\cdots \ud x_n\ud y_1\cdots\ud y_n  \cdot \psi^{-1}(u)\ud u \\
&=\int_{\Fp_{\RF}^{-k}}\int_{\RF^n}|\xi|^{-\frac{\dim_{\BD}\BV}{2}}\BF_{\psi^{-1}}(\Phi_1)(uy_1)\cdots\BF_{\psi^{-1}}(\Phi_{n-1})(uy_{n-1})\BF_{\psi^{-1}}(\Phi_n)(-\xi^{-1}z^{-1}+u y_n)\\
&\cdot 	\Psi_1(y_1)\cdots\Psi_n(y_n)\psi^{-1}(zy_n)\ud y_1\cdots \ud y_n\psi^{-1}(u)\ud u.
\end{align*}
For fixed $\xi$, when $z$ is large enough, we have $\BF_{\psi^{-1}}(\Phi_n)(\xi^{-1}z^{-1}-uy_n)=\BF_{\psi^{-1}}(\Phi_n)(uy_n)$, and
\begin{align*}
	\int_{\RF}\BF_{\psi^{-1}}(\Phi_n)(-\xi^{-1}z^{-1}+u y_n)\Psi_n(y_n)\psi^{-1}(zy_n) \ud y_n=\int_{\RF}\BF_{\psi^{-1}}(\Phi_n)(u y_n)\Psi_n(y_n)\psi^{-1}(zy_n) \ud y_n.
\end{align*}
It is not hard to see that when $|u|\leq q^k$, 
\[
\int_{\RF}\BF_{\psi^{-1}}(\Phi_n)(u y_n)\Psi_n(y_n)\psi^{-1}(zy_n) \ud y_n=0
\]
when $|z|$ is large enough. Similarly by changing $z$ to $z^{-1}$, we see
	\[
	 f \left(  \begin{pmatrix}
	\xi & 0 \\ 0 & 1
\end{pmatrix}  \begin{pmatrix}
	z & 0 \\ 0 & z
\end{pmatrix}  \right)=0.
	\]
	when $|z|$ is small enough, hence \[
\int_{\BG_m}   f \left(  \begin{pmatrix}
	\xi & 0 \\ 0 & 1
\end{pmatrix}  \begin{pmatrix}
	z & 0 \\ 0 & z
\end{pmatrix}  \right)\ud z
\]
is absolutely convergent for $\xi\in\RF^{\times}$. 

\end{proof}

\begin{dfn}
\begin{itemize}
	\item [(1)] For any $\phi\in\CF(\FX)$, and $x \in\RF\setminus\{0,1\}$, write
	\[
	\widetilde{\phi}(x):= \int_{\RF}\phi(z^{-1} x,z)\ud^{\times}z.
	\]
	We will use $\CF(\overline{\FX})$ to denote the space of such functions and $\CS(\overline{\FX}):=\CF(\overline{\FX})\ud x$.
	\item [(2)] Let $\Orb_{\BV}(\overline{\RT})$ be the space of functions obtained from $\Orb(\RT)$ by taking integrals over the center, namely, they are functions of $\xi\in\RF^{\times}$ of the form
\[
\widetilde{f}\left(\xi \right):= \int_{\BG_m}   f \left(  \begin{pmatrix}
	\xi & 0 \\ 0 & 1
\end{pmatrix}  \begin{pmatrix}
	z & 0 \\ 0 & z
\end{pmatrix}  \right) \ud z,\;f\in\Orb_{\BV}(\RT) .
\]
Write $\CS_{\BV}(\overline{\RT}):=\Orb_{\BV}(\overline{\RT})|\xi|\ud^{\times}\xi$.
\end{itemize}
	
\end{dfn}

\begin{lem}\label{lem_Tr_Weil_A}
We have a commutative diagram
		\begin{align*}
	\xymatrix{
\CF(\FX)   \ar@{->}[d]^{} \ar@{->}[r]  &   \Orb_{\BV}(\RT)\ar@{->}[d]^{} \\
\CS(\overline{\FX})  \ar@{->}[r]  & \CS_{\BV}(\overline{\RT}),
}
\end{align*}
where the bottom arrow is given by
\[
\overline{\BT}:|\cdot|^{-\frac{n}{2}}   \circ       \iota\circ  \BF_{\psi}\circ |\cdot|^{-1}\circ\iota\circ \BF_{\psi}
\]
in the sense of measures.
\end{lem}

\begin{proof}
	Let $\varphi\in\CC_c^{\infty}(\RF^{\times})$ be a test function, and $f=\widetilde{\BT}(\phi)$ for $\phi\in\CF(\FX)$, then we have
	\begin{align*}
		\langle \widetilde{f},\varphi\rangle=\int_{\RF}\widetilde{f}(\xi)\varphi(\xi)\ud \xi.
	\end{align*}
	Using Fubini's theorem, we have
	\begin{align*}
		\langle \widetilde{f},\varphi\rangle&= \int_{\RF} \int_{\RF^{\times}}\left|\xi \right|^{-\frac{n}{2}} \int_{(\RF\times\RF)^{\heartsuit}} \phi(\xi_1,\xi_2)\psi^{-1}( z  \xi_1+\xi^{-1}z^{-1}\xi_2)\ud \xi_1\ud \xi_2  \frac{\ud z}{|z|} \varphi(\xi)\ud \xi \\
		&=\int_{\RF^{\times}}\int_{(\RF\times\RF)^{\heartsuit}} \phi(\xi_1,\xi_2)\psi^{-1}( z  \xi_1+\xi^{-1}z^{-1}\xi_2) \varphi(\xi)|\xi|^{-\frac{n}{2}}\ud \xi     \ud \xi_1\ud \xi_2  \frac{\ud z}{|z|}\\
		&=\int_{\RF^{\times}}\int_{(\RF\times\RF)^{\heartsuit}}\phi(\xi_1,\xi_2)\psi^{-1}(z\xi_1)\BF_{\psi}( |\cdot|^{-2}\circ\iota\circ|\cdot|^{-\frac{n}{2}}\varphi   )(z^{-1}\xi_2)\ud\xi_1\ud\xi_2\frac{\ud z}{|z|}.
	\end{align*}
	As $\BF_{\psi}( |\cdot|^{-2}\circ\iota\circ|\cdot|^{-\frac{n}{2}}\varphi   )$ vanishes near $0$ and $\infty$, we can apply the Fubini's theorem again to obtain
	\begin{align*}
		\langle \widetilde{f},\varphi\rangle&=\int_{(\RF\times\RF)^{\heartsuit}}\phi(\xi_1,\xi_2)\int_{\RF^{\times}}\BF_{\psi}( |\cdot|^{-2}\circ\iota\circ|\cdot|^{-\frac{n}{2}}\varphi   )(z^{-1})|z^{-1}|^{-1}\psi^{-1}(\xi_1\xi_2 z)\frac{\ud z}{|z|^2}\ud\xi_1\ud\xi_2\\
		&=\langle \widetilde{\phi},\BF_{\psi}\circ |\cdot|^{-2}\circ\iota\circ  |\cdot|^{-1}\circ\BF_{\psi}\circ|\cdot|^{-2}\circ\iota \circ |\cdot|^{-\frac{n}{2}} \varphi \rangle .
	\end{align*}
	Note that for a measure $\mu$ and a test function $\varphi$, we have
	\begin{align*}
		\langle \iota (\mu),\varphi\rangle =\langle \mu,|\cdot|^{-2}\circ\iota (\varphi)\rangle , 
	\end{align*}
	hence 
	\begin{align*}
		&\langle \widetilde{\phi},\BF_{\psi}\circ |\cdot|^{-2}\circ\iota\circ  |\cdot|^{-1}\circ\BF_{\psi}\circ|\cdot|^{-2}\circ\iota \circ |\cdot|^{-\frac{n}{2}} \varphi  \rangle\\
		&=\langle |\cdot|^{-\frac{n}{2}} \circ \iota \circ \BF_{\psi} \circ |\cdot|^{-1}\circ  \iota \circ  \BF_{\psi},\varphi\rangle .
	\end{align*}
\end{proof}

\begin{prp}\label{prp_tr_weil_a}
	Let $\CT$ be the transfer operator in Theorem \ref{thm_TrSak}, then we have \[
\overline{\BT}\circ\CT=	 \Id.
\] 
\end{prp}

\begin{proof}
	This is a direct consequence of Lemma \ref{lem_Tr_Weil_A}.
\end{proof}

The last case is when $\BD=\Mat_2(\RF)$. Consider the action of $\SL_2$ on $\Mat_2$ by left or right multiplication. Then according to \cite[Theorem 5.3.3]{GW09}, $\Mat_2\sslash\SL_2=\BA^1$ via $g\mapsto \det g$. Write $\CS(\FX):=\CF(\FX)\ud x$ for measures on $\Mat_2(\RF)^{\heartsuit}$ by multiplying an additive Haar measure.
\begin{lem}\label{lem_c_integrable}
	For $\phi\in\CF(\FX)$, and $x\in\RF\setminus\{0,1\}$,
	\begin{align*}
		\widetilde{\phi}(x):=\int_{\SL_2}\phi \left( \begin{pmatrix}
			x & 0 \\ 0 & 1
		\end{pmatrix}  g   \right)\ud g
	\end{align*}
	is absolutely convergent.
\end{lem}

\begin{proof}
This is similar to Lemma \ref{lem_abs_a_lft}.
\end{proof}

In this case, we have to make some modifications to the transfer operator $\widetilde{\BT}$. Now, consider the embedding 
\begin{align*}
	\SL_2\hookrightarrow\SO_4:g\mapsto \begin{pmatrix}
		g & 0 \\ 0 & g
	\end{pmatrix}.
\end{align*}
Consider the adjoint action of $\PGL_2$ on $\Fg\Fl_2=\Mat_2(\RF)$, which is equipped with the quadratic form given by the determinant. The ordered basis
\[
\left\{\begin{pmatrix}
	0 & 1 \\ 0& 0 
\end{pmatrix},   \begin{pmatrix}
	1 & 0 \\ 0 & 0
\end{pmatrix}  ,\begin{pmatrix}
	0 & 0 \\ 0 & -1
\end{pmatrix} ,\begin{pmatrix}
	0 & 0 \\ 1 & 0
\end{pmatrix}\right\}
\]
gives an embedding $\PGL_2\rightarrow\RO_4$ such that 
\begin{align*}
	\begin{pmatrix}
		t & \\ & 1
	\end{pmatrix} \mapsto \begin{pmatrix}
		t & & & \\ & 1 & & \\ &  & 1 & \\ & & & t^{-1}
	\end{pmatrix},\;\begin{pmatrix}
		1 & n \\ & 1
	\end{pmatrix}\mapsto \begin{pmatrix}
		1 & & n & \\ & 1 & & n \\ & & 1 & \\ & & & 1
	\end{pmatrix},
\end{align*}
and
\begin{align*}
	\begin{pmatrix}
		& -1 \\ 1 & 
	\end{pmatrix}\mapsto \begin{pmatrix}
		& & & -1 \\ & & 1 & \\ & 1 & & \\ -1 & & &
	\end{pmatrix}.
\end{align*}
We also identify $\PGL_2$ with its image in $\RO_4$. Note that $\PGL_2$ normalizes $\SL_2$. Then, let us consider the function
\begin{align*}
\overline{\Omega}(\Phi):\PGL_2&\rightarrow\BC\\
g&\mapsto \int_{\SL_2}\Omega_{\psi}(\Phi)(gg^{\prime})\ud g^{\prime}.
\end{align*}
For $\Phi \in \CF(\BV)$, replace $\Phi$ by $\omega_{\psi}(g)\Phi$ for $g\in\SL_2$ in Proposition \ref{prp_TranOp},we have
\begin{align*}
		&\int_{\RN}^*\omega_{\psi}(\Bw)\omega_{\psi}\left( \begin{pmatrix}
 		\xi & 0 & 0 & 0 \\ 0 & 1 & 0 & 0 \\ 0 & 0 & 1 & 0\\ 0 & 0 & 0 & \xi^{-1}
 	\end{pmatrix}  \right)\omega_{\psi}(u) \omega_{\psi}(g) \Phi(v_0)\psi^{-1}(u)\ud u\\
		&=|\xi|^{-\frac{\dim_{\RF}\BV}{4}}\int_{\Mat_2(\RF)^{\heartsuit}}\CI(\res(\Phi))(x)\psi^{-1}\left(\tau  \left(  xg^{-1}  \begin{pmatrix}
			\xi^{-1} & 0 \\ 0 & 1
		\end{pmatrix}  \right)\right)\ud x=:\widehat{\BT}(\phi)(\xi,g),
\end{align*}
where $\phi=\CI(\Res(\Phi))$. Denote the space of such functions by $\Orb_{\BV}(\RT\times\SL_2)$.

\begin{lem}\label{lem_conv_c}
	For $\phi\in\CF(\FX)$,
	\begin{align*}
		\int_{\SL_2}\widehat{\BT}(\phi)(\xi,g)\ud g
	\end{align*}
	is absolutely convergent for $\xi\in\RF^{\times}$.
\end{lem}

\begin{proof}
	Let $\phi=\CI(\res(\Phi))$ with $\Phi\in\CF(\BV)$. Assume $\RK^{\prime}\subset\RK=\SL_2(\Fo_{\RF})$ such that $\omega_{\psi}(k)\Phi=\Phi$ for any $k\in\RK^{\prime}$. According to the Iwasawa decomposition, we have
	\begin{align*}
		\int_{\SL_2}\widehat{\BT}(\phi)(\xi,g)\ud g&=\int_{\RF\times\RF^{\times}}\int_{\RK}\widehat{\BT}(\phi)\left(\xi, \begin{pmatrix}
			1 & x \\ 0 & 1
		\end{pmatrix} \begin{pmatrix}
			m & 0 \\ 0 & m^{-1}
		\end{pmatrix} k\right)|m|^{-2}\ud x\frac{\ud m}{|m|}\ud k\\
		&=\sum_{k_i\in \RK/\RK^{\prime}}\vol(\RK^{\prime})\int_{\RF\times\RF^{\times}}\widehat{\BT}(\phi_i)\left(\xi, \begin{pmatrix}
			1 & x \\ 0 & 1
		\end{pmatrix} \begin{pmatrix}
			m & 0 \\ 0 & m^{-1}
		\end{pmatrix} \right)|m|^{-2}\ud x\frac{\ud m}{|m|},
	\end{align*}
	where $\phi_i=\CI(\res( \omega_{\psi}(k_i)\Phi ))$. For simplicity of notations, replace $\Phi$ by $\omega_{\psi}(k_i)\Phi$, it suffices to show that
	\begin{align*}
		\int_{\RF\times\RF^{\times}}\widehat{\BT}(\phi)\left(\xi, \begin{pmatrix}
			1 & x \\ 0 & 1
		\end{pmatrix} \begin{pmatrix}
			m & 0 \\ 0 & m^{-1}
		\end{pmatrix} \right)|m|^{-2}\ud x\frac{\ud m}{|m|}
	\end{align*}
	is absolutely convergent.
	
	Note that we have

\begin{align*}
	 & \begin{pmatrix}
	  	0 & 0 & -1 & 0 \\ 0 & 0 & 0 & -1 \\1 & 0 & 0 & 0 \\0 & 1 & 0 & 0
	  \end{pmatrix}   \begin{pmatrix}
 		\xi & 0 & 0 & 0 \\ 0 & 1 & 0 & 0 \\ 0 & 0 & 1 & 0\\ 0 & 0 & 0 & \xi^{-1}
 	\end{pmatrix} \begin{pmatrix}
 		1 & 0 &  u & 0 \\ 0 & 1 & 0 & u \\ 0 & 0 & 1 & 0\\ 0 & 0 & 0 & 1
 	\end{pmatrix} \begin{pmatrix}
 		1 & x  & 0 & 0\\ 0 & 1 & 0 & 0 \\ 0 & 0 & 1 & x \\ 0 & 0 & 0 & 1
 	\end{pmatrix}  \begin{pmatrix}
 		m & 0 & 0 & 0 \\ 0 & m^{-1} & 0 & 0 \\ 0 & 0 & m & 0 \\ 0 & 0 & 0 & m^{-1}
 	\end{pmatrix}\\
 	&=\begin{pmatrix}
 		1 & \xi x  & 0 & 0\\ 0 & 1 & 0 & 0 \\ 0 & 0 & 1 & \xi x \\ 0 & 0 & 0 & 1
 	\end{pmatrix} \begin{pmatrix}
	  	0 & 0 & -1 & 0 \\ 0 & 0 & 0 & -1 \\1 & 0 & 0 & 0 \\0 & 1 & 0 & 0
	  \end{pmatrix}  \begin{pmatrix}
 		\xi m & 0 & 0 & 0 \\ 0 & m^{-1} & 0 & 0 \\ 0 & 0 & m & 0 \\ 0 & 0 & 0 & \xi^{-1}m^{-1} 	\end{pmatrix}\begin{pmatrix}
 		1 & 0 &  u & 0 \\ 0 & 1 & 0 &  u \\ 0 & 0 & 1 & 0\\ 0 & 0 & 0 & 1
 	\end{pmatrix} ,
\end{align*}
then 
\begin{align*}
	&\omega_{\psi}\left( \Bw\begin{pmatrix}
 		\xi & 0 & 0 & 0 \\ 0 & 1 & 0 & 0 \\ 0 & 0 & 1 & 0\\ 0 & 0 & 0 & \xi^{-1}
 	\end{pmatrix} \begin{pmatrix}
 		1 & 0 &  u & 0 \\ 0 & 1 & 0 & u \\ 0 & 0 & 1 & 0\\ 0 & 0 & 0 & 1
 	\end{pmatrix} \begin{pmatrix}
 		1 & x  & 0 & 0\\ 0 & 1 & 0 & 0 \\ 0 & 0 & 1 & x \\ 0 & 0 & 0 & 1
 	\end{pmatrix}  \begin{pmatrix}
 		m & 0 & 0 & 0 \\ 0 & m^{-1} & 0 & 0 \\ 0 & 0 & m & 0 \\ 0 & 0 & 0 & m^{-1}
 	\end{pmatrix} \right)\Phi(v_0)\\
 	&=\omega_{\psi}\left(  \Bw  \begin{pmatrix}
 		\xi m & 0 & 0 & 0 \\ 0 & m^{-1} & 0 & 0 \\ 0 & 0 & m & 0 \\ 0 & 0 & 0 & \xi^{-1}m^{-1} 	\end{pmatrix}\begin{pmatrix}
 		1 & 0 &  u & 0 \\ 0 & 1 & 0 &  u \\ 0 & 0 & 1 & 0\\ 0 & 0 & 0 & 1
 	\end{pmatrix}  \right)\Phi\left(v_0\begin{pmatrix}
 		1 & \xi x \\ 0 & 1
 	\end{pmatrix}\right)\\
 	&=\int_{\BV}\omega_{\psi}\left(  \begin{pmatrix}
 		\xi m & 0 & 0 & 0 \\ 0 & m^{-1} & 0 & 0 \\ 0 & 0 & m & 0 \\ 0 & 0 & 0 & \xi^{-1}m^{-1} 	\end{pmatrix}\begin{pmatrix}
 		1 & 0 &  u & 0 \\ 0 & 1 & 0 & u \\ 0 & 0 & 1 & 0\\ 0 & 0 & 0 & 1
 	\end{pmatrix} \right)\Phi(v)\\
 	&\cdot \psi^{-1}\left(\tau\left(\BB_{\BV}\left( v_0\begin{pmatrix}
 		1 & \xi x \\ 0 & 1
 	\end{pmatrix} ,v  \right) \right)\right)\ud v \\
 	&=\int_{\BV}|\xi m\cdot m^{-1}|^n\psi\left(  \frac{u\cdot \BB_{\BV}\left(v\begin{pmatrix}
	\xi m & 0 \\ 0 & m^{-1}
\end{pmatrix},v\begin{pmatrix}
	\xi m & 0 \\ 0 & m^{-1}
\end{pmatrix}\right)}{2}      \right)\Phi\left(v\begin{pmatrix}
	\xi m & 0 \\ 0 & m^{-1}
\end{pmatrix}\right)\\
&\psi^{-1}\left(\tau\left(\BB_{\BV}\left( v_0\begin{pmatrix}
 		1 & \xi x \\ 0 & 1
 	\end{pmatrix}  ,v \right) \right)\right)\ud v\\
 	&=\int_{\BV}|\xi |^{-n}\psi\left(\frac{ u \cdot \BB_{\BV}(v,v)}{2}   \right)\Phi(v)\psi^{-1}\left(\tau\left(\BB_{\BV}\left( v_0\begin{pmatrix}
 		1 & \xi x \\ 0 & 1
 	\end{pmatrix} ,v \begin{pmatrix}
 		\xi^{-1} m^{-1} & 0 \\ 0 & m
 	\end{pmatrix}    \right) \right)\right)\ud v.
\end{align*}
Write $v=\left( \begin{pmatrix}
	x_1 & z_1 \\ w_1 & y_1
\end{pmatrix} ,\cdots,\begin{pmatrix}
	x_n & z_n \\w_n & y_n
\end{pmatrix}   \right)$, then
\begin{align*}
	\tau\left(\BB_{\BV}\left(  v_0\begin{pmatrix}
 		1 & \xi x \\ 0 & 1
 	\end{pmatrix} ,v \begin{pmatrix}
 		\xi^{-1} m^{-1} & 0 \\ 0 & m
 	\end{pmatrix}   \right) \right)&=\tau\left( \begin{pmatrix}
 		1 & -\xi x\\ 0 & 1
 	\end{pmatrix}  \begin{pmatrix}
 		x_n & z_n \\ w_n & y_n
 	\end{pmatrix} \begin{pmatrix}
 		\xi^{-1} m^{-1} & 0 \\  0& m^{-1} 
 	\end{pmatrix} \right)\\
 	&=\tau \left(    \begin{pmatrix}
 		\xi^{-1}m^{-1}x_n-xm^{-1}w_n & *\\ * & my_n
 	\end{pmatrix}    \right)\\
 	&=\xi^{-1}m^{-1}x_n+my_n-m^{-1}xw_n,
\end{align*}
so
\begin{align*}
	&\omega_{\psi}\left( \Bw\begin{pmatrix}
 		\xi & 0 & 0 & 0 \\ 0 & 1 & 0 & 0 \\ 0 & 0 & 1 & 0\\ 0 & 0 & 0 & \xi^{-1}
 	\end{pmatrix} \begin{pmatrix}
 		1 & 0 &  u & 0 \\ 0 & 1 & 0 & u \\ 0 & 0 & 1 & 0\\ 0 & 0 & 0 & 1
 	\end{pmatrix} \begin{pmatrix}
 		1 & x  & 0 & 0\\ 0 & 1 & 0 & 0 \\ 0 & 0 & 1 & x \\ 0 & 0 & 0 & 1
 	\end{pmatrix}  \begin{pmatrix}
 		m & 0 & 0 & 0 \\ 0 & m^{-1} & 0 & 0 \\ 0 & 0 & m & 0 \\ 0 & 0 & 0 & m^{-1}
 	\end{pmatrix} \right)\Phi(v_0)\\
 	&=|\xi |^{-n}\int_{\RF^{4n}}\Phi\left(\left( \begin{pmatrix}
	x_1 & z_1 \\ w_1 & y_1
\end{pmatrix} ,\cdots,\begin{pmatrix}
	x_n & z_n \\w_n & y_n
\end{pmatrix}   \right)\right)\\
&\cdot \psi(u(x_1y_1+\cdots +x_ny_n-z_1w_1-\cdots -z_nw_n))\psi^{-1}(\xi^{-1}m^{-1}x_n+m y_n-m^{-1}xw_n).
\end{align*}
Without loss of generality, we may assume \[\Phi=\Phi_1(x_1,\cdots,x_n,y_1,\cdots,y_n)\Phi_2(z_1,\cdots,z_n,w_1,\cdots,w_n)\]
with $\Phi_1,\Phi_2\in\CF(\RF^n)$. Then for some $k$ large enough,
\begin{align*}
	&\widehat{\BT}(\phi)\left(\xi,  \begin{pmatrix}
		1 & x \\ 0 & 1
	\end{pmatrix} \begin{pmatrix}
		m & 0 \\ 0 & m^{-1}
	\end{pmatrix} \right)\\
	&=|\xi|^{-n}\int_{\Fp_{\RF}^{-k}}\int_{\RF^{2n}}\Phi_1(x_1,\cdots,x_n,y_1,\cdots,y_n)\psi(u(x_1y_1+\cdots+x_ny_n))\psi^{-1}(\xi^{-1}m^{-1}x_n+my_n)\\
	&\cdot\int_{\RF^{2n}}\Phi_2(z_1,\cdots,z_n,w_1,\cdots,w_n)\psi(-u(z_1w_1+\cdots+z_nw_n))\psi^{-1}(-m^{-1}xw_m)\psi^{-1}(u)\ud u.
\end{align*}
From the proof of Lemma \ref{lem_conv_a},
\begin{align*}
	\int_{\RF^{2n}}\Phi_1(x_1,\cdots,x_n,y_1,\cdots,y_n)\psi(u(x_1y_1+\cdots+x_ny_n))\psi^{-1}(\xi^{-1}m^{-1}x_n+my_n)\ud x_1\cdots\ud y_n\neq 0
\end{align*}
implies that $M_1 \leq |m|\leq M_2$ for some $M_1,M_2>0$. Then since
\begin{align*}
	&\int_{\RF^{2n}}\Phi_2(z_1,\cdots,z_n,w_1,\cdots,w_n)\psi^{-1}(-u(z_1w_1+\cdots+z_nw_n))\psi^{-1}(-m^{-1}x w_n)\ud z_1\cdots \ud w_n\\
	&=\int_{\RF^n}\BF_{\psi,1}\circ\cdots \circ\BF_{\psi,n}\Phi_2(-uw_1,\cdots,-uw_{n-1},-uw_n,w_1,\cdots,w_n)\psi^{-1}(-m^{-1}xw_n)\ud w_1\cdots\ud w_n,
\end{align*}
which is non-zero only when $|x|\leq A$ for some $A>0$ depending on $\Phi_2$, $\Fp_{\RF}^{-k}$, and $M$. Therefore,
\begin{align*}
		\int_{\SL_2}\widehat{\BT}(\phi)(\xi,g)\ud g
	\end{align*}
	is absolutely convergent for $\xi\in\RF^{\times}$.

\end{proof}

\begin{dfn}
	\begin{itemize}
		\item [(1)] For $\phi\in\CS(\FX)$, let $\widetilde{\phi}$ be the push-forward of the measure $\phi$ to $\BA^1$ along the determinant map. We will use $\CS(\overline{\FX})$ to denote the space of such measures.
		\item [(2)] Let $\Orb_{\BV}(\overline{\RT})$ be the space of functions 
		\begin{align*}
	\xi\mapsto \int_{\SL_2}\widehat{\BT}(\phi)(\xi,g)\ud g,
		\end{align*}
		and $\CS_{\BV}(\overline{\RT}):=\Orb_{\BV}(\overline{\RT})|\xi|\ud^{\times}\xi$.
	\end{itemize}
\end{dfn}

\begin{rmk}
	Observe that
		\[
	\begin{pmatrix}
		& 1 & & \\ -1 & & & \\ & & & 1 \\ & & -1 & 
	\end{pmatrix}\begin{pmatrix}
		 & & -1 & \\ & & & -1 \\ 1 & & & \\ & 1 & &
	\end{pmatrix}=\begin{pmatrix}
		& & & -1 \\ & & 1 & \\ & 1 & & \\ -1 & & &
	\end{pmatrix}
	\]
	with $\begin{pmatrix}
		& 1 & & \\ -1 & & & \\ & & & 1 \\ & & -1 & 
	\end{pmatrix}\in\SL_2(\RF)$, hence on the maximal torus of $\PGL_2$,
	\[
	\int_{\RN}^*\overline{\Omega}(\Phi)\left(\Bw\begin{pmatrix}
		\xi & \\ & 1
	\end{pmatrix}  n\right)\psi^{-1}(n)\ud n=\int_{\SL_2}\widehat{\BT}(\CI(\Res(\Phi)))(\xi,g)\ud g.
	\]
\end{rmk}

\begin{lem}\label{lem_Tr_Weil_C}
We have a commutative diagram
		\begin{align*}
	\xymatrix{
\CF(\FX)   \ar@{->}[d]^{} \ar@{->}[r]  &   \Orb_{\BV}(\RT\times\SL_2)\ar@{->}[d]^{} \\
\CS(\overline{\FX})  \ar@{->}[r]  & \CS_{\BV}(\overline{\RT}),
}
\end{align*}
where the bottom arrow is given by
\[
\overline{\BT}:|\cdot|^{-n+1}\circ\iota\circ\BF_{\psi}\circ|\cdot|^{-2}\circ\iota\circ\BF_{\psi}
\]
in the sense of measures or distributions.
\end{lem}

\begin{proof}
	Let $\varphi\in\CC_c^{\infty}(\RF^{\times})$ be a test function, and $f=\widehat{\BT}(\phi)$ with $\phi\in\CF(\FX)$, then we have
	\begin{align*}
		\langle \widetilde{f},\varphi\rangle =\int_{\RF}|\xi|^{-n}\int_{\SL_2}\int_{\Mat_2(\RF)^{\heartsuit}}\phi(x)\psi^{-1}\left( \tr \left( \begin{pmatrix}
			1 & 0 \\ 0 & \zeta^{-1}
		\end{pmatrix} g x  \right)\right)\ud x \ud g\cdot  \varphi(\xi)\ud \xi.
	\end{align*}
	Write $x=\begin{pmatrix}
		x_1 & x_2 \\ x_3 & x_4
	\end{pmatrix}$ and using the Brahut decomposition to write \[g=\begin{pmatrix}
		1 & 0 \\ u & 1 
	\end{pmatrix}\begin{pmatrix}
		t & 0 \\ 0 & t^{-1}
	\end{pmatrix}\begin{pmatrix}
		 1 & v \\ 0 & 1
	\end{pmatrix}.\] Then
	\begin{align*}
		\tau \left(  x g^{-1} \begin{pmatrix}
			\xi^{-1} & 0 \\ 0 & 1
		\end{pmatrix}  \right)&=\tr \left(  \begin{pmatrix}
		   \xi^{-1}tuvx_1+\xi^{-1}t^{-1}x_1-\xi^{-1}tux_2   	 & * \\ * & -tvx_3+tx_4		\end{pmatrix}    \right)\\
		&=	   \xi^{-1}tuvx_1+\xi^{-1}t^{-1}x_1-\xi^{-1}tux_2   	-tvx_3+tx_4	
	\end{align*}
	As $\varphi$ is a Schwartz function, we can apply Fubini's theorem, then
	\begin{align*}
		\langle \widetilde{f},\varphi\rangle &=\int_{\RF^{\times}\times\RF\times\RF}\int_{\Mat_2(\RF)^{\heartsuit}}\phi(x)\psi^{-1}(	tx_4 -tvx_3)\\
		&\cdot \int_{\RF}\varphi(\zeta)|\zeta|^{-n}\psi^{-1}( \xi^{-1} tuvx_1 + \xi^{-1} t^{-1}x_1 - \xi^{-1}tux_2)\ud \xi \ud x_1\ud x_2\ud x_3\ud x_4 |t|^2\frac{\ud t}{|t|}\ud u \ud v \\
	\end{align*}
	Since 
	\begin{align*}
		 \int_{\RF}\varphi(\xi)|\xi|^{-n}\psi^{-1}( \xi^{-1} tuvx_1 + \xi^{-1} t^{-1}x_1 - \xi^{-1}tux_2)\ud \zeta \\
		  =( \BF_{\psi}\circ |\cdot |^{-2} \circ\iota \circ  |\cdot|^{-n})(\varphi)(tuvx_1+t^{-1}x_1-tux_2)
	\end{align*}
	and $ (\BF_{\psi}\circ |\cdot |^{-2} \circ\iota \circ  |\cdot|^{-n})(\varphi)$ also vanishes at $0$ and $\infty$, we can first integral over $v$ and replace $v$ by $\displaystyle{ \frac{v+tux_2-t^{-1}x_1}{tux_1}   }$, then
	\begin{align*}
		&\int_{\RF}( \BF_{\psi}\circ |\cdot |^{-2} \circ\iota \circ  |\cdot|^{-n})(\varphi)(tuvx_1+t^{-1}x_1-tux_2)\psi^{-1}(-tvx_3)\ud v\\
		&=\int_{\RF}( \BF_{\psi}\circ |\cdot |^{-2} \circ\iota \circ  |\cdot|^{-n})(\varphi)(v)\psi^{-1}\left( -v\frac{x_3}{ux_1}  \right )\frac{1}{|tux_1|}  \ud v \cdot \psi^{-1}\left(-\frac{-t^{-1}x_1x_3+tux_2x_3}{ux_1}    \right)\\
		&=(|\cdot|^{-2}\circ\iota\circ  |\cdot|^{-n})(\varphi)\left(\frac{x_3}{ux_1}\right)\frac{1}{|tux_1|}\psi^{-1}\left(-\frac{-t^{-1}x_1x_3+tux_2x_3}{ux_1}    \right).
	\end{align*}
	We can further integral over $u$ first, and make a change of variable by replacing $u$ by $\displaystyle{\frac{x_3}{ux_1}}$, then
	\begin{align*}
		&\int_{\RF}(|\cdot|^{-2}\circ\iota\circ  |\cdot|^{-n})(\varphi)\left(\frac{x_3}{ux_1}\right)\frac{1}{|u|}\psi^{-1}\left( \frac{t^{-1}x_3}{u} \right)\ud u\\
		&=\int_{\RF}(|\cdot|^{-2}\circ\iota\circ  |\cdot|^{-n})(\varphi)(u)\psi^{-1}(t^{-1}x_1 u)\frac{\ud u}{|u|}\\
		&= (\BF_{\psi}\circ |\cdot|^{-1}\circ |\cdot|^{-2}\circ\iota\circ  |\cdot|^{-n})(\varphi)\left(  \frac{x_1}{t}  \right).
	\end{align*} 
	Finally, replace $t$ by $tx_1$, then we have
	\begin{align*}
		&\int_{\RF}(\BF_{\psi}\circ |\cdot|^{-1}\circ |\cdot|^{-2}\circ\iota\circ  |\cdot|^{-n})(\varphi)\left(  \frac{x_1}{t}  \right) \frac{1}{| t x_1|} \psi^{-1}\left(t\frac{x_1x_4-x_2x_3}{x_1}\right) |t|^2 \frac{\ud t}{|t|}\\
		&=\int_{\RF}(\BF_{\psi}\circ |\cdot|^{-1}\circ |\cdot|^{-2}\circ\iota\circ  |\cdot|^{-n})(\varphi)\left(  \frac{1}{t}  \right)\psi^{-1}(t(x_1x_4-x_2x_3))\ud t\\
		&=(\BF_{\psi}\circ \iota \circ \BF_{\psi^{-1}}\circ |\cdot|^{-1}\circ |\cdot|^{-2}\circ\iota\circ  |\cdot|^{-n})(\varphi)(x_1x_4-x_2x_3), 
	\end{align*}
	hence
	\begin{align*}
		\langle \widetilde{f},\varphi \rangle &=\langle \widetilde{\phi},(\BF_{\psi}\circ \iota \circ \BF_{\psi}\circ |\cdot|^{-1}\circ |\cdot|^{-2}\circ\iota\circ  |\cdot|^{-n})(\varphi) \rangle \\
		&=\langle     (  |\cdot |^{-n}\circ    \iota\circ  |\cdot|^{-1}\circ \BF_{\psi}\circ |\cdot |^{-2}\circ \iota \circ \BF_{\psi} )(\phi),\varphi\rangle \\
		&=\langle ( |\cdot |^{-n+1}\circ \iota \circ \BF_{\psi}\circ |\cdot|^{-2}\circ \iota \circ \BF_{\psi})(\phi),\varphi\rangle .
	\end{align*}
\end{proof}

\begin{prp}\label{prp_tr_weil_c}
	Let $\CT$ be the transfer operator in Theorem \ref{thm_TrSak}, then we have \[
\overline{\BT}\circ\CT=	 \Id.
\] 
\end{prp}

\begin{proof}
	It follows from Lemma \ref{lem_Tr_Weil_C}.
\end{proof}

\section{Fundamental lemma of type $D_n$}\label{sec_d}

The results of this section are fundamentally contained in the works of S. Rallis \cite{Ral82} and W. Gan-X. Lei \cite{GW21}. While we provide a reformulation of these results in the present framework, the author wishes to emphasize that the credit for the fundamental lemma for the full Hecke algebra in the $D_n$ case belongs to S. Rallis. We include the details here to ensure the manuscript remains self-contained and to clarify the compatibility with our notation.

In this section, let $\BD=\RF$, and $\dim_{\RF}\BV$ is even such that there is an ordered basis $\{e_1,\cdots,e_n,f_n,\cdots,f_1\}$ satisfying
	\[
	\BB_{\BV}(e_i,e_j)=\BB_{\BV}(f_i,f_j)=0,\;\mathrm{and}\;\BB_{\BV}(e_i,f_j)=\delta_{ij}\; \forall \;1\leq i,j\leq n,
	\]
	where $\delta_{ij}$ is the Kronecker symbol. Under the isomorphism $\BV\cong\RF^{2n}$ under the above basis, we have
		\[
	\BB_{\BV}:\RF^{2n}\times\RF^{2n}\rightarrow\RF:\left(\begin{pmatrix}
		x_1 \\ \vdots \\ x_{2n}
	\end{pmatrix}  ,    \begin{pmatrix}
		y_1\\ \vdots \\ y_{2n}
	\end{pmatrix}  \right)\mapsto\sum_{i=1}^{2n}x_iy_{2n+1-i}.
	\]
Take $\BW$ to be a $2$-dimensional symplectic space over $\RF$, and fix the ordered basis $\{e,f\}$ of $\BW$ such that $\BB_{\BW}(e,f)=1$. Under the isomorphism $\BW\cong\RF^2$ under this fixed basis, we have
	\[
	\BB_{\BW}:\RF^2\times\RF^2\rightarrow\RF:\left( \begin{pmatrix}
		x_1 \\ x_2
	\end{pmatrix} ,   \begin{pmatrix}
		y_1 \\ y_2
	\end{pmatrix}  \right)\mapsto x_1y_2-x_2y_1.
	\]
	In this case, $\RU(\BV,\BB_{\BV})=\SO(\BV)$ and $\RU(\BW,\BB_{\BW})=\SL_2$. Recall that we take $v_0=e_1+f_1=\begin{pmatrix}
		1 \\ 0 \\ \vdots \\ 0 \\ 1
	\end{pmatrix}\in\RF^{2n}$ and
	\[\RX=\SO(\BV)/\SO(v_0^{\perp})\cong \{ v\in\BV\mid \BB_{\BV}(v,v)=2    \}\subset\BV:g\mapsto g v_0.\]
	Let $\Phi_0$ be the characteristic function of $\BY(\Fo_{\RF})$.

\subsection{Matching of unit elements}

\begin{lem}\label{lem_volb}
	\[
	\CI(\Res(\Phi_0))(x)\ud x= \frac{1}{\zeta_{\RF}(n)}\BL_{\RX}^{\circ} =\frac{\#\RX(\kappa_{\RF})}{q^{\dim\RX}}\BL_{\RX}^{\circ}. 	\]
\end{lem}

\begin{proof}
	
 Denote	\[
	F(x):=\int_{\BV_x}\Phi_0(v)\omega_x(v),\;x\neq 0.
	\]
	Then we only need to compute $F(1)$. To do this, let
	\[
	F^*(x):=\int_{\BV}\Phi_0(v)\psi\left(x \cdot \frac{\tau (\BB_{\BV}(v,v))}{2} \right)\ud v.
	\]
	According to \cite[Section 8.3, Theorem 8.3.1]{Igu00}, we know that
	\[
	F^*(x)=\max\{1,|x|\}^{-n},
	\]
	and
	\begin{align*}
	F(0)&=\lim_{k\rightarrow\infty}\int_{|x|\leq q^k}F^*(x)\psi(-x)\ud x\\
	&= \int_{|x|\leq 1}1\ud x +   \lim_{k\rightarrow\infty}\int_{q\leq |x|\leq q^k}|x|^{-n} \psi(-x)  \ud x\\
	&=1-q^{-n}.
	\end{align*}
	
	On the other hand, due to \cite[Chapter 11]{Tay92}, 
\[
\#\RO(\BV)(\kappa_{\RF})=2q^{n(n-1)}(q^n-1)\prod_{i=0}^{n-1}(q^{2i}-1)
\]
and
\[
\#\RO( v_0^{\perp} )(\kappa_{\RF})=2q^{(n-1)^2}\prod_{i=0}^{n-1}(q^{2i}-1),
\]
hence 
\[
\frac{\#\RX(\kappa_{\RF})}{q^{\dim\RX}}=1-q^{-n}=\zeta_{\RF}(n)^{-1}.
\]

\end{proof}

\begin{lem}\label{lem_matunit_a}
	\[
	\Omega(\Phi_0)= \frac{1}{\zeta_{\RF}(2n-2)}\Phi_{L_{\RX}.}  
	\]
\end{lem}

\begin{proof}

The identity follows from the fact that both functions are $(\RN, \psi)$-left invariant and $\SL_2(\Fo_{\RF})$-right invariant; since they agree on the torus $\RT$, they must coincide everywhere by the Iwasawa decomposition.
\end{proof}

As a consequence,

\begin{thm}[fundamental lemma for the unit element]\label{thm_d}
	Let $\BL_{\RX}^{\circ}\in\CS(\SO_{2n-1}\backslash\SO_{2n}/\SO_{2n-1})$ be the basic measure on $\BA^1$, then we have
	\[
	\CT^{-1}\left(\BL_{\RX}^{\circ}\right)=\frac{q^{\dim\RX}}{\#\overline{\RX}(\kappa_{\RF})\zeta_{\RF}(2)\zeta_{\RF}(s_{\RX})}f_{L_{\RX}}=\frac{\zeta_{\RF}(n)}{\zeta_{\RF}(2)\zeta_{\RF}(2n-2)}f_{L\left(\Ad,n-1\right)}
	\]
\end{thm}

\subsection{Matching of Hecke elements}

Let $\RB$ be the Borel subgroup of $\SO(\BV)$ fixing the isotropic flag
\[
\langle e_1\rangle \subset \langle e_1,e_2\rangle \subset \cdots \subset \langle e_1,\cdots,e_n\rangle,
\]
and $\RA$ the maximal torus consisting of elements fixing each line $\langle e_i\rangle $ for every $1\leq i\leq n$. 
\begin{lem}\label{lem_open_b}
\[
	\RX^{\circ}=\{v\in\RX\mid \BB_{\BV}(x,e_1)\neq 0\}.
	\]
Then $v_0\in\RX^{\circ}$ and $\RX^{\circ}$ is an open $\RB$-orbit in $\RX$.	
\end{lem}

\begin{proof}
	It is clear that $\RX^{\circ}$ is open in $\RX$, stable under the action of $\RB$ and $v_0\in\RX^{\circ}$. So we only need to show $\RX^{\circ}=\RB v_0$. 
	
	Indeed, let $\displaystyle{ x=\sum_{i=1}^n x_i e_i+\sum_{i=1}^n y_i f_i \in \RX^{\circ}  }$, then $y_1\neq 0$. By the action of $\RA$, we may assume $y_1=1$. Under the ordered basis $\{e_1,\cdots,e_n,f_n,\cdots,f_1\}$, we have
	\[
	\begin{pmatrix}
		s_n \left(\tensor*[^t]{  \begin{pmatrix} 1 & & & -y_n \\  & 1 & & -y_{n-1} \\ & & \rotatebox{45}{$\vdots$} & \vdots  \\ & & & 1 \end{pmatrix}  }{^{-1}}  \right)
 s_n & 0_n \\ 0_n &  \begin{pmatrix} 1 & & & -y_n \\  & 1 & & -y_{n-1} \\ & & \rotatebox{45}{$\vdots$} & \vdots  \\ & & & 1 \end{pmatrix} 	\end{pmatrix}\begin{pmatrix}
 	x_1  \\ x_2 \\ \vdots \\ x_n \\ y_n \\ y_{n-1} \\ \vdots \\ 1 
 \end{pmatrix}=\begin{pmatrix}
 	1 \\ x_2^{\prime} \\ \vdots \\ x_n^{\prime} \\ 0 \\ 0 \\ \vdots \\ 1
 \end{pmatrix},
	\]
	for some $x_2^{\prime},\cdots,x_n^{\prime}\in\RF$. And here $s_n=\begin{pmatrix}
		& & 1 \\ & \rotatebox{135}{$\vdots$} \\ 1 & & 
	\end{pmatrix}$. Then one can take
	\[
	\begin{pmatrix}
		\RI_n   &     \begin{pmatrix} x_n^{\prime} & x_{n-1}^{\prime} & \cdots & x_2^{\prime} & 0\\0 & 0 & \cdots & 0 & -x_2^{\prime}\\ \vdots &\vdots & \vdots & \vdots&  \vdots \\ 0&0 &\cdots  & 0 & -x_{n-1}^{\prime} \\ 0 & 0  & 0 & 0 & -x_n^{\prime}\end{pmatrix}      \\ 0_n &  \RI_n \end{pmatrix}\begin{pmatrix}
 	1  \\ x_2^{\prime} \\ \vdots \\ x_n^{\prime} \\ 0  \\ 0 \\ \vdots \\ 1 
 \end{pmatrix}=\begin{pmatrix}
 	1 \\ 0  \\ \vdots \\ 0 \\ 0 \\ 0 \\ \vdots \\ 1
 \end{pmatrix}.
\]
\end{proof}

\begin{cor}
	$\RP(\RX)$ is the maximal proper parabolic subgroup fixing the isotropic flag
	\[
	\langle e_1\rangle.
	\]

Then under the above basis $\{ e_1,\cdots, e_n,f_n,\cdots, f_1\}$, the Levi $\RL(\RX)$ of $\RP(\RX)$ is 
\[
\RL(\RX)=\left\{  \begin{pmatrix}
	t & & \\ & g & \\ & & t^{-1}
\end{pmatrix}  \mid t\in\BG_m,g\in\SO(v_0^{\perp})    \right\}\cong\BG_m\times\SO(v_0^{\perp}),
\]
and hence
\[
\delta_{(\RX)}^{\frac{1}{2}} \left(  \begin{pmatrix}
	t_1 & & & & & \\ & \rotatebox{45}{$\vdots$} & & & & \\ & & t_n & & & \\ & & & t_n^{-1} & & \\ & & & & \rotatebox{45}{$\vdots$} & \\ & & & & & t_1^{-1} 
\end{pmatrix}  \right)=|t_2|^{n-2}|t_3|^{n-3}\cdots |t_n|^0.
\]
Moreover, the quotient map $\RA\rightarrow\RA_{\RX}\cong\BG_m$ is given by 
\[
\begin{pmatrix}
	t_1 & & & & & \\ & \rotatebox{45}{$\vdots$} & & & & \\ & & t_n & & & \\ & & & t_n^{-1} & & \\ & & & & \rotatebox{45}{$\vdots$} & \\ & & & & & t_1^{-1} 
\end{pmatrix} \mapsto t_1.
\]
Then the normalized map $\BC[\RA^{\vee}]^{\BW_{\RG}}\rightarrow\BC[\RA_{\RX}^{\vee}]^{\BW_{\RX}}$ is given by
\begin{align}\label{hkduald}
	\BC[x_1^{\pm},\cdots,x_n^{\pm}]^{\BW_{\SO_{2n}}}\rightarrow\BC[x^{\pm}]^{\BZ/2\BZ}:x_1\mapsto x ,\;x_j\mapsto q^{j-n}\;\mathrm{for}\;2\leq j\leq n.
\end{align}
\end{cor}

Now consider the nilpotent cone $\displaystyle{\Sigma:=\{ v\in\BV\mid \BB_{\BV}(v,v)=0  \}}$.
\begin{lem}[Lemma 3.1 in \cite{Ral82}]\label{lem_nil_d}
	\[
	\Sigma^{\circ}:=\left\{ g\begin{pmatrix}
		1 \\ 0 \\ \vdots \\ 0
	\end{pmatrix}\mid g\in\SO(\BV) \right\}
	\]
	is an open dense $\SO(\BV)$-orbit of $\Sigma$.
\end{lem}

\begin{proof}
	According to Witt's Theorem, $\Sigma=\Sigma^{\circ}\bigsqcup\{0\}$, see \cite[Chapter 4]{Ser73}. See also \cite[Lemma 3.1]{Ral82}.
\end{proof}

Let $\omega_{\psi}$ be the Weil representation of $\SO(\BV)\times\SL_2$ on $\CF(\BY)$.

\begin{lem}[Proposition 2.2 in \cite{Ral82}]\label{lem_nilipvan_d}
	Let $\Phi\in\CF(\BY)^{\SL_2(\Fo_{\RF})}$, and $\omega_{\psi}(g)\Phi|_{\Sigma}=0$ for all $g\in\SL_2$, then $\Phi=0$.
\end{lem}

\begin{proof}
This is a special case of \cite[Proposition 2.2]{Ral82}. But let us include the proof here for the convenience of the reader. Note that as a representation of $\SL_2$, the dual of the Jacquet module
\[
\CF(\BY)_{\RN}:=\CF(\BY)/\left\langle \Phi-\omega_{\psi}(n)\Phi\mid \Phi\in\CF(\BY),n\in \left\{\begin{pmatrix}
	1 & * \\ & 1
\end{pmatrix} \right\}  \right\rangle
\]
is isomorphic to the distributions on $\BY$ supported on $\Sigma$. Then the space of functions vanishing on $\Sigma$ is isomorphic to the Jacquet module $\CF(\BY)_{\RN}$. Since for any $g\in\SL_2$, $\omega_{\psi}(g)\Phi|_{\Sigma}=0$, the $\SL_2$-module generated by $\Phi$ is cuspidal. But we know any such module cannot be unramified, hence $\Phi=0$.
\end{proof}

Now let $\sigma\in\BC$ be such that $\Re(\sigma)>1$. Consider the following intertwining operators introduced in \cite{Ral82}:
\[
\RZ_{\sigma}:\CF(\BY)\rightarrow \CC^{\infty}(\SO_{2n}\times\SL_2) :\Phi \mapsto (g_1,g_2)\mapsto \int_{\RF} \omega_{\psi}(g_1,g_2)^{-1}\Phi\left( \begin{pmatrix}
	a \\ 0 \\ \vdots \\ 0
\end{pmatrix} \right)|a|^{\sigma} \frac{\ud a}{|a|},
\]
which is absolutely convergent according to \cite{Tat67}.

\begin{prp}[Lemma 4.1 and Remark 4.4 in \cite{Ral82}]\label{prp_cpthk_d}
	When $\Re(\sigma)>1$,
	\begin{itemize}
		\item [(1)] Let $\Fs=(s_1,s_2,\cdots,s_n)\in X^*(\RA)\otimes_{\BZ}\BC=\BC^n$, write 
\[
\chi(\Fs): \RT\rightarrow\BC:  \begin{pmatrix}
	t_1 & & & & & \\ & \rotatebox{45}{$\vdots$} & & & & \\ & & t_n & & & \\ & & & t_n^{-1} & & \\ & & & & \rotatebox{45}{$\vdots$} & \\ & & & & & t_1^{-1} 
\end{pmatrix} \mapsto \prod_{i=1}^n|t_i|^{s_i}
\]
and $\RI(\chi(\Fs))$ the correponding normalized induced representation of $\SO_{2n}$. Then $\RZ_{\sigma}$ factors through $\RI(\chi(n-1-\sigma,n-2,n-3,\cdots,0) )\otimes\RI(\chi(\sigma-n+1))$ as representations of $\SO_{2n}\times\SL_2$.
\item [(2)] Write $\RK$ for the maximal compact subgroup of $\SO(\BV)$ fixing the standard lattice $\Fo_{\RF}^{2n}$ under the ordered basis $\{e_1,\cdots,e_n,f_n,\cdots,f_1\}$. Let \[\lambda_{\RX}:\CH(\SO(\BV),\RK)\rightarrow\CH(\SL_2,\SL_2(\Fo_{\RF}))\] be the morphism of the Hecke algebra corresponding to (\ref{hkduald}). Let $\Phi_0$ be the characteristic function of $\BY(\Fo_{\RF})$, then we have 
	\[
\omega_{\psi}(h^{\vee})\Phi_0=\omega_{\psi}(\lambda_{\RX}(h))\Phi_0.
\]
for all $h\in\CH(\SO(\BV),\RK)$. 
	\end{itemize}
\end{prp}

\begin{proof}
	$(1)$ is simply a matter of computing the effect of $\RA\times\RT$ on $\RZ_{\sigma}$ when $\Re(\sigma)>1$.
	
	As for $(2)$, a first observation is that for any $f\in\CH(\SO(\BV),\RK)$,
	\begin{align*}
&\RZ_{\sigma}(\omega_{\psi}(h^{\vee})\Phi_0  - \omega_{\psi}(\lambda_{\RX}(h))\Phi_0   )\\
&=\left(\Sat(h^{\vee})(n-1-\sigma,n-2,\cdots,0)-\Sat (\lambda_{\RX}(h))(\sigma-n+1)\right)\RZ_{\sigma}(\Phi_0)=0
\end{align*}
when $\Re(\sigma)>1$. Hence for any $g_1\in\SO(\BV)$ and $g_2\in\SL_2$, as a function on $\RF$,
\[
a\mapsto \omega_{\psi}(g_1,g_2)^{-1}(\omega_{\psi}(h^{\vee})\Phi_0  - \omega_{\psi}(\lambda_{\RX}(h))\Phi_0 )  \left( \begin{pmatrix}
	a \\ 0 \\ \vdots \\ 0
\end{pmatrix} \right)
\]
is $\Fo_{\RF}^{\times}$-invariant and its Mellin transform is zero. Hence according to the Mellin inversion \cite[Chapter 1]{IJ78}, \[
 \omega_{\psi}(g_1,g_2)^{-1}(\omega_{\psi}(h^{\vee})\Phi_0  - \omega_{\psi}(\lambda_{\RX}(h))\Phi_0 )  \left( \begin{pmatrix}
	a \\ 0 \\ \vdots \\ 0
\end{pmatrix} \right)\]
for all $a\in\RF$. By varing $g_1\in\SO(\BV)$ and Lemma \ref{lem_nil_d}, we see for any $g_2\in\SL_2$, 
\[
\omega_{\psi}(g_2)(\omega_{\psi}(h^{\vee})\Phi_0  - \omega_{\psi}(\lambda_{\RX}(h))\Phi_0 )|_{\Sigma}=0.
\]
Now as $\displaystyle{\omega_{\psi}(h^{\vee})\Phi_0  - \omega_{\psi}(\lambda_{\RX}(h))\Phi_0 \in  \CF(\BV)^{\SL_2(\Fo_{\RF})}  }$, according to Lemma \ref{lem_nilipvan_d}, we obtain that 
\[
\omega_{\psi}(h^{\vee})\Phi_0  = \omega_{\psi}(\lambda_{\RX}(h))\Phi_0 .
\]
 \end{proof}

\begin{thm}
	Conjecture \ref{cnj_flhk} is true for $\RX$ of type $\RD_n$.
\end{thm}

\begin{proof}
	As the restriction map $\CF(\BV)\rightarrow\CF(\RX)$ is obviously $\SO(\BV)$-equivariant, hence for any $h\in\CH(\SO(\BV),\RK)$,
	\[
	\CI\circ\res(\omega_{\psi}(h^{\vee})\Phi_0)\ud x= \CI(h\star ( \res(\Phi_0)  )  )\ud x = h\star \frac{1}{\zeta_{\RF}(n)}\BL_{\RX}^{\circ}.
	\]
	On the other hand, as functions on $\SL_2$, since we know from Lemma \ref{lem_matunit_a}
	\[
	\Omega(\Phi_0)(g)=\omega_{\psi}(g)\Phi_0(v_0)=\frac{1}{\zeta_{\RF}(2n-2)}\Phi_{L\left(\Ad,n-1\right)}(g),
	\]
	hence we have 
	\[ \Omega( \omega_{\psi}(\lambda_{\RX}(h))\Phi_0)=(\lambda_{\RX}(h))^{\vee}\star  \frac{1}{\zeta_{\RF}(2n-2)}\Phi_{L\left(\Ad,n-1\right)}(g)  =  \lambda_{\RX}(h)\star  \frac{1}{\zeta_{\RF}(2n-2)}\Phi_{L\left(\Ad,n-1\right)} (g) .   \]
	Therefore, Conjecture \ref{cnj_flhk} is true for $\RX$ of type $\RD_n$.
\end{proof}

\section{Fundamental lemma of type $A_{n-1}$}\label{sec_a}

In this section, take $\BD=\RF\times\RF$, $\BV=\RV\times\RV^{\vee}$, where $\RV$ is a vector space over $\RF$ and $\RV^{\vee}$ its dual space. Fix an ordered basis $\{e_1,\cdots,e_n\}$ of $\RF^n$ and its dual basis $\{e_1^{\vee},\cdots,e_n^{\vee}\}$ of $\RV^{\vee}$, then we have an isomorphism $\BV\cong\RF^n\times\RF^n$, and take
	\begin{align*}
	\BB_{\BV}:\RF^n\times\RF^n&\rightarrow\RF\times\RF\\
	\left(  \left(  \begin{pmatrix}
		x_1\\ \vdots \\ x_n
	\end{pmatrix},\begin{pmatrix}
		y_1 \\ \vdots \\ y_n
	\end{pmatrix}   \right),  \left(  \begin{pmatrix}
		\xi_1\\ \vdots \\ \xi_n
	\end{pmatrix},\begin{pmatrix}
		\zeta_1 \\ \vdots \\ \zeta_n
	\end{pmatrix}   \right)      \right)&\mapsto \left(  \sum_{i=1}^nx_i\zeta_{i},\sum_{i=1}^n y_i\xi_{i}  \right).
	\end{align*}
	Similarly take $\BW=\RW\times\RW^{\vee}$, where $\RW$ is a $2$-dimensional vector space over $\RF$. Fix an ordered basis $\{f_1,f_2\}$ of $\RW$ and its dual basis $\{f_1^{\vee},f_2^{\vee}\}$ of $\RW^{\vee}$, then we have an isomorphism $\BW\cong\RF^2\times\RF^2$, and take
	\begin{align*}
	\BW:\RF^2\times\RF^2&\rightarrow\RF\times\RF \\\left(  \left(  \begin{pmatrix}
		x_1\\ x_2
	\end{pmatrix},\begin{pmatrix}
		y_1 \\ y_2
	\end{pmatrix}   \right),  \left(  \begin{pmatrix}
		\xi_1\\\xi_2
	\end{pmatrix},\begin{pmatrix}
		\zeta_1 \\ \zeta_2
	\end{pmatrix}   \right)      \right)&\mapsto \left(  \sum_{i=1}^2x_i\zeta_{i},-\sum_{i=1}^2 y_i\xi_{i}  \right).
	\end{align*}
	Write $e:=(f_1,f_2^{\vee})$ and $f:=(-f_2,f_1^{\vee})$. 	In this case, $\RU(\BV,\BB_{\BV})=\GL(\RV)=\GL_n$ and $\RU(\BW,\BB_{\BW})=\GL(\RW)=\GL_2$. Recall that $v_0=e_n+e_n^{\vee}$, and we have 
\[
\RX=\GL_n/\GL_{n-1}\cong\{v\in\BV\mid \tr\;\BB_{\BV}(v,v)=2  \}\subset\BV:g\mapsto gv_0,
\]
where we identify $\RU(v_0^{\perp})=\GL_{n-1}$.
Let $\overline{\RX}:=\RX/\BG_m=\PGL_n/\GL_{n-1}$, which classifies the decompositions $\RV=\RV_{1}\oplus\RV_{n-1}$, the direct sum of $\RV$ as $1$-dimensional and $n-1$-dimensional subspaces. Let $\Phi_0$ be the characteristic function of $\BY(\Fo_{\RF})$. 
\subsection{Matching of unit elements}

\begin{lem}\label{lem_vola}
	The push-forward of $\CI(\Res(\Phi_0))(x)\ud x$ in $\BA^1$ is 
	\[
	\frac{1}{\zeta_{\RF}(n)}\BL_{\RX}^{\circ}=\frac{q^{\dim\overline{\RX}}}{\#\overline{\RX}(\kappa_{\RF})}\BL_{\RX}^{\circ}.
	\]
\end{lem}

	\begin{proof}
	The first part follows from the same argument as in Lemma \ref{lem_volb}. It is well known that
	\[
	\frac{q^{\dim\overline{\RX}}}{\#\overline{\RX}(\kappa_{\RF})}=\frac{1}{\zeta_{\RF}(n)}.
	\]
    \end{proof}

\begin{lem}\label{lem_a_pgl}
	\[
    \int_{\BG_m}\Omega(\Phi_0)(gz)\ud z= \frac{1}{\zeta_{\RF}(1)\zeta_{\RF}(n-1)}\Phi_{L\left(\std,\frac{n-1}{2}\right)\left(\std,\frac{n-1}{2}\right)}(g).
	\]
\end{lem}

\begin{proof}
	We only need to check that both sides match on the maximal torus. Note that
	\begin{align*}
		\int_{\BG_m}\Omega(\Phi_0)\left(  \begin{pmatrix}
			\xi & \\ & 1
		\end{pmatrix}\begin{pmatrix}
			z & \\ & 1
		\end{pmatrix} \right) \ud z&=\int_{\BG_m}|\xi|^{\frac{n}{2}}1_{\Fo_{\RF}}(\zeta z)1_{\Fo_{\RF}}(z^{-1})\ud z\\
		&=(1-q^{-1})|\xi|^{\frac{n}{2}}(1-\log_q|\xi|) \\
		&=\frac{1}{\zeta_{\RF}(1)\zeta_{\RF}(n-1)}\Phi_{L\left(\std,\frac{n-1}{2}\right)\left(\std,\frac{n-1}{2}\right)}\left(\begin{pmatrix}
			\xi & \\ & 1 
		\end{pmatrix} \right).
	\end{align*}
\end{proof}

Therefore,
\begin{thm}\label{thm_a}
	Let $\BL_{\RX}^{\circ}\in\CS(\GL_{n-1}\backslash\PGL_n/\GL_{n-1})$ be the push-forward of the probability measure which is equal to the characteristic function of $\overline{\RX}(\Fo_{\RF})\times\overline{\RX}(\Fo_{\RF})$ times an invariant measure, then we have
	\[
	\CT^{-1}\left(\BL_{\RX}^{\circ}\right)=\frac{q^{\dim\RX}}{\#\overline{\RX}(\kappa_{\RF})\zeta_{\RF}(2)\zeta_{\RF}(s_{\RX})}f_{L_{\RX}}=\frac{\zeta_{\RF}(n)}{\zeta_{\RF}(1)\zeta_{\RF}(2)\zeta_{\RF}(n-1)}f_{L\left(\std,\frac{n-1}{2}\right)\left(\std,\frac{n-1}{2}\right)}.
	\]
\end{thm}

\subsection{Matching of Hecke elements}

Let $\RB$ be the Borel subgroup of $\GL_n$ consisting of all upper triangular matrices and $\RA$ the maximal torus consisting of diagonal matrices under the ordered basis $\{e_1,\cdots,e_n\}$. Let $\overline{\RB}$ and $\overline{\RA}$ be their images in $\PGL_n$. 
\begin{lem}
Let \[
	\overline{\RX}^{\circ}:=\{\RV=\RV_{1}\oplus\RV_{n-1}\mid \RV_1 \nsubseteq \langle e_1,\cdots,e_{n-1}\rangle, e_1 \notin \RV_{n-1}    \}.
	\] 
	Then $\overline{\RX}^{\circ}$ is an open $\overline{\RB}$-orbit in $\overline{\RX}$.
\end{lem}

\begin{proof}
	Note that if we identify $\RV_{n-1}$ with its normal direction $\RF \mathbf{n}$ with $\mathbf{n}\in\RV^{\vee}$, then the action of $\GL(\RV)$ on $\mathbf{n}$ is given by the transpose. Hence, the Lemma follows easily. At this time the decomposition $\RV=\RF e_n\oplus \{v\in\RV\mid \langle v,e_1^{\vee} \rangle = 0  \}$ lies in $\overline{\RX}^{\circ}$.
\end{proof}

\begin{cor}
	Let $\RP(\overline{\RX})$ be the parabolic subgroup of $\PGL_n$ stabilizing $\overline{\RX}^{\circ}$, then its Levi
	\[
	\RL(\overline{\RX})=\left\{ \begin{pmatrix}
		t_1 & & \\ & g & \\ & & t_2
	\end{pmatrix} \mid t_1,t_2\in\BG_m,g\in\GL_{n-2}   \right\}/\BG_m.
	\]
	Therefore,
	\begin{align*}
	\delta_{(\overline{\RX})}^{\frac{1}{2}}\left(   \begin{pmatrix}
		a_1 & & \\ & \rotatebox{45}{$\vdots$} & \\ & & a_n
	\end{pmatrix}   \right)=|a_2|^{\frac{n-3}{2}}|a_3|^{\frac{n-5}{2}}\cdots |a_{n-1}|^{\frac{3-n}{2}}.
\end{align*}
Then $\BC[\overline{\RA}^{\vee}]^{\BW_{\PGL_n}}\rightarrow\BC[\RA_{\overline{\RX}}^{\vee}]^{\BW_{\overline{\RX}}}$ is given by
\begin{align}\label{hkduala}
	\nonumber  \left(\BC[x_1^{\pm},\cdots,x_n^{\pm}]/(x_1\cdots x_n-1)\right)^{\FS_n}\rightarrow\BC[x^{\pm}]^{\BZ/2\BZ}\\ 
	x_1\mapsto x, x_n \mapsto x^{-1}, x_j\mapsto q^{\frac{2j-n-1}{2}}\;(2\leq j\leq n-1),
\end{align}
where $\FS_n$ is the permutation group of of $(1,2,3,\cdots,n)$.
\end{cor}

Consider the nilpotent cone $\Sigma:=\{ (v,v^{\vee})\in\BV\mid \langle v,v^{\vee}\rangle =0  \}$ of $\BV$.

\begin{lem}
	\[
		\Sigma^{\circ}:=\left\{   \left( g\begin{pmatrix}
			1 \\ 0 \\ \vdots \\ 0
		\end{pmatrix}  ,  {^tg^{-1} }\begin{pmatrix}
			0 \\ \vdots \\ 0 \\ 1
		\end{pmatrix} \right)\mid g\in\GL_n \right\}
		\]
		is an open $\GL_n$-orbit of $\Sigma$.
\end{lem}

\begin{proof}
	It is clear that $\dim\Sigma=2n-1$. On the other hand, since the stabilizer of $(e_1,e_n^{\vee})$ is
	\[
	\left\{\begin{pmatrix}
		1 & * & \cdots & * \\ 0 & * & \cdots & * \\ \vdots & \vdots &\rotatebox{45}{$\vdots$} & \vdots \\0 & 0 & \cdots & 1
	\end{pmatrix}\right\},
	\]
	hence $\dim\Sigma^{\circ}=n^2-(n^2-(2n-1))=2n-1=\dim\Sigma$.
\end{proof}

Let $\omega_{\psi}$ be the Weil representation of $\GL_n\times\GL_2$ on $\CF(\BY)$.

\begin{lem}\label{lem_nilipvan_a}
	Let $\Phi\in\CF(\BY)^{\GL_2(\Fo_{\RF})}$, and $\omega_{\psi}(g)\Phi|_{\Sigma}=0$ for all $g\in\GL_2$, then $\Phi=0$.
\end{lem}

\begin{proof}
	Similarly to \ref{lem_nilipvan_d}, it suffices to show that the dual of the Jacquet module
\[
\CF(\BY)_{\RN}:=\CF(\BY)/\left\langle \Phi-\omega_{\psi}(n)\Phi\mid \Phi\in\CF(\BY),n\in \left\{\begin{pmatrix}
	1 & * \\ & 1
\end{pmatrix} \right\}  \right\rangle
\]
is isomorphic to the distributions on $\BY$ supported on $\Sigma$. Let $\varphi$ be a distribution on $\CF(\BY)$ that vanishes on every $\Phi-\omega_{\psi}\left(\begin{pmatrix}
	1 & n \\ & 1
\end{pmatrix}   \right)\Phi$ for $\Phi\in\CF(\BY)$ and $n\in\RF$. Since for $y\in\BY$, we have
\[
\omega_{\psi}\left(\begin{pmatrix}
	1 & n \\ & 1
\end{pmatrix}   \right)\Phi(y)=\psi\left(  \frac{n \cdot \tau(\BB(y,y))}{2}   \right) \Phi(y), 
\]
then
\[
\varphi(y)-\psi\left(  \frac{n \cdot \tau(\BB(y,y))}{2}   \right)\varphi(y)=0
\]
as distributions for any $n\in\RF$. Then $\varphi$ is supported $\{y\in\BY \mid\tau(\BB(y,y))=0 \}=\Sigma$.
\end{proof}

Now let $\sigma_1,\sigma_2\in\BC$ be such that $\Re(\sigma_1),\Re(\sigma_2)>1$, and consider the following intertwining operators
\[
\RZ_{\sigma_1,\sigma_2}:\CF(\BY)\rightarrow \CC^{\infty}(\GL_n\times\GL_2) :\Phi\mapsto \int_{\RF\times\RF} \omega_{\psi}(g_1,g_2)^{-1}\Phi\left( \begin{pmatrix}
	a \\ 0 \\ \vdots \\ 0
\end{pmatrix} ,\begin{pmatrix}
	0 \\ \vdots \\ 0 \\ b
\end{pmatrix} \right)|a|^{\sigma_1} |b|^{\sigma_2} \frac{\ud a}{|a|}\frac{\ud b}{|b|},
\]
which is absolutely convergent.

\begin{prp}\label{prp_cpthk_a}
	When $\Re(\sigma_1),\Re(\sigma_2)>1$,
	\begin{itemize}
		\item [(1)] Let $\Fs=(s_1,s_2,\cdots,s_n)\in X^*(\RA)\otimes_{\BZ}\BC\cong \BC^n$, write 
\[
\chi(\Fs): \RA\rightarrow\BC:  \begin{pmatrix}
	t_1 & &\\ & \rotatebox{45}{$\vdots$} &  \\ & & t_n 
\end{pmatrix} \mapsto \prod_{i=1}^n|t_i|^{s_i},
\]
and $\RI(\chi(\Fs))$ the correponding normalized induced representation of $\GL_n$. Then $\RZ_{\sigma_1,\sigma_2}$ factors through $\RI(\chi(\frac{n-1}{2}-\sigma_1,\frac{n-3}{2},\cdots, \frac{3-n}{2},   \frac{1-n}{2}+\sigma_2))\otimes\RI(\chi(\sigma_1-\frac{n-1}{2},-\sigma_2+\frac{n-1}{2}))$. 
	\item [(2)] Let $\lambda_{\RX}:\CH(\GL_{n},\GL_{n}(\Fo_{\RF}))\rightarrow\CH(\GL_2,\GL_2(\Fo_{\RF}))$ be the morphism between Hecke algebra whose Stakee transform is given by
	\begin{align*}
	\BC[x_1^{\pm},\cdots,x_n^{\pm}]^{\FS_n}\rightarrow\BC[x_1^{\pm},x_2^{\pm}]^{\FS_2}:x_1\mapsto x_1, x_n\mapsto x_2, x_j\mapsto q^{\frac{2j-n-1}{2}}\;(2\leq j\leq n-1).
\end{align*}
\end{itemize}
Then we have $\omega_{\psi}(f^{\vee})\Phi_0=\omega_{\psi}(\lambda_{\RX}(f))\Phi_0$.
	\end{prp}

\begin{proof}
	It follows from the same argument as in Proposition \ref{prp_cpthk_d}.
\end{proof}

\begin{thm}\label{thm_flhka}
	Conjecture \ref{cnj_flhk} is true for $\RX$ of type $\RA_{n-1}$.
\end{thm}

\begin{proof}
	Consider the natural homomorphism $\BC[\RA^{\vee}]^{\BW_{\GL_n}}\rightarrow\BC[\overline{\RA}^{\vee}]^{\BW_{\PGL_n}}$, and the corresponding homomorphism between Hecke algebras $\CH(\GL_n,\GL_n(\Fo_{\RF}))\rightarrow\CH(\PGL_n,\PGL_n(\Fo_{\RF}))$. Then we have a commutative diagram
		\begin{align*}
	\xymatrix{
\CH(\GL_n,\GL_n(\Fo_{\RF}))  \ar@{->}[d]^{\alpha} \ar@{->}[r]^{\lambda_{\RX}}  &   \CH(\GL_2,\GL_2(\Fo_{\RF})) \ar@{->}[d]^{\beta} \\
\CH(\PGL_n,\PGL_n(\Fo_{\RF})) \ar@{->}[r]^{\lambda_{\overline{\RX}}}  &\CH(\PGL_2,\PGL_2(\Fo_{\RF})),
}
\end{align*}
where the bottom arrow is given by (\ref{hkduala}). According to Lemma \ref{lem_git}, we have a commutative diagram
\[
\begin{tikzcd}
\GL_n/\GL_{n-1} \arrow[r, ""] \arrow[d, "p"]
&\BA^2 \arrow[d, "q"] \\
\PGL_n/\GL_{n-1} \arrow[r, "\pi"]
&\BA^1
\end{tikzcd},
\]
where $q$ is given by $(x,y)\mapsto xy$,
and $\pi$ can be identified with the canonical quotient map \[
\PGL_n/\GL_{n-1}\rightarrow\GL_{n-1}\rotatebox{60}{$\sslash$}\PGL_n/\GL_{n-1}.\]

For simplicity of notation, let us identify functions with measures using the Haar measures we specified, and let $(\cdot)_*$ denote the push-forward of measures. Let $h\in\CH(\GL_n,\GL_n(\Fo_{\RF}))$. Then we have
\[
\CT^{-1}(\alpha(h)\star \BL_{\RX}^{\circ})=\CT^{-1}\left(\zeta_{\RF}(n) \cdot \pi_*\circ p_*      ( h\star  \res(\Phi_0)   )\right).
\]
According to \cite[Section 7]{Sat63}, \[\pi_*\circ p_*(  h\star \res(\Phi_0)     )   =\pi_*\circ p_*(  \omega_{\psi}(h^{\vee})   (\res( \Phi_0)  ) ),\] which is equal to 
\[
q_*\circ\CI( \omega_{\psi}(h^{\vee})  (\res(\Phi_0)))=q_*\circ\CI( \omega_{\psi}(\lambda_{\RX}(h)) \res(\Phi_0))
\]
according to Proposition \ref{prp_cpthk_a}.

Due to Lemma \ref{lem_a_pgl}, we have
\[
\int_{\RZ}\omega_{\psi}(gz)\Phi_0(v_0)\ud z=	\frac{1}{\zeta_{\RF}(1)\zeta_{\RF}(n-1)}\Phi_{L\left(\std,\frac{n-1}{2}\right)\left(\std,\frac{n-1}{2}\right)}(g),\;\forall g\in\PGL_2,
\]
then for $g\in\PGL_2$,
\begin{align*}
	\int_{\RZ}\omega_{\psi}(gz)\omega_{\psi}(\lambda_{\RX}(h))\Phi_0(v_0)\ud z&=    \left(\int_{\RZ}\lambda_{\RX}(h)\star\omega_{\psi}(\cdot z)\Phi_0(v_0)\ud z)\right)(g)  \\
	&=\beta(\lambda_{\RX}(h))  \star  \frac{1}{\zeta_{\RF}(1)\zeta_{\RF}(n-1)}\Phi_{L\left(\std,\frac{n-1}{2}\right)\left(\std,\frac{n-1}{2}\right)} (g)
\end{align*}
due to \cite[Section 7]{Sat63}. Finally because $\alpha$ is surjective and $\beta(\lambda_{\RX}(h))=\lambda_{\overline{\RX}}(\alpha(h))$,
we obtain that 
\[
\CT^{-1}( h  \star  \BL_{\RX}^{\circ} )=\lambda_{\RX}(h)\star \frac{\zeta_{\RF}(n)}{\zeta_{\RF}(1)\zeta_{\RF}(2)\zeta_{\RF}(n-1)}f_{   L\left(\std,\frac{n-1}{2}\right)\left(\std,\frac{n-1}{2}\right)  }.
\]

\end{proof}

\section{Fundamental lemma of type $C_n$}\label{sec_c}

In this section, let $\BD=\Mat_2(\RF)$, and fix the isomorphism $\RF^2\otimes_{\RF}\RF^2\cong\BD$:
	\[
	\begin{pmatrix}
		a \\ b
	\end{pmatrix}\otimes  (c,d)\mapsto \begin{pmatrix}
		ac & ad \\ bc & bd
	\end{pmatrix}.
	\]
    Take $\BV=\BD^n$, and
  	\[
	\BB_{\BV}:\BD^n\times\BD^n\rightarrow\BD: ((x_1,\cdots,x_n),(y_1,\cdots,y_n))\mapsto \sum_{i=1}^n \overline{x_i}y_i.
	\]
	Under the above identification $\BD\cong\RF^2\otimes_{\RF}\RF^2$, and $\BD^n\cong\RF^{2n}\otimes_{\RF}\RF^2$, the above form is given by
	\begin{align*}
	\left( \begin{pmatrix}
		v_1 \\ \vdots \\ v_{2n}
	\end{pmatrix} \otimes (x_1,x_2) \right. &, \left.\begin{pmatrix}
		w_1 \\ \vdots \\ w_{2n}
	\end{pmatrix} \otimes (y_1,y_2) \right)    \\
	&\mapsto (v_1,\cdots,v_{2n})\cdot  \begin{pmatrix}
		J &  & \\ & \ddots  & \\ & & J
	\end{pmatrix}   \cdot\begin{pmatrix}
		w_1 \\ \vdots \\w_{2n} 
	\end{pmatrix} \cdot J^{-1} \begin{pmatrix}
		x_1 \\ x_2
	\end{pmatrix}  (y_1,y_2),
	\end{align*}
	where $J=\begin{pmatrix}
		0 & 1 \\ -1 & 0
	\end{pmatrix}$. 
	Similarly take $\BW=\BD^2$ and 
	\[
	\BB_{\BW}:\BD^2\times\BD^2\rightarrow\BD: ((x_1,x_2),(y_1,y_2))\mapsto \overline{x_1}y_2-\overline{x_2}y_1.
	\]
	Then under the identification $\BD^2\cong\RF^{4}\otimes_{\RF}\RF^2$ as above, this form is given by
	\begin{align*}
	\left( \begin{pmatrix}
		v_1 \\ \vdots \\ v_{4}
	\end{pmatrix} \otimes (x_1,x_2) \right. &, \left.\begin{pmatrix}
		w_1 \\ \vdots \\ w_{4}
	\end{pmatrix} \otimes (y_1,y_2) \right)    \\
	&\mapsto (v_1,\cdots,v_{4}) \begin{pmatrix}
		  0  & J \\ -J & 0 
	\end{pmatrix}  \cdot\begin{pmatrix}
		w_1 \\ \vdots \\w_{4} 
	\end{pmatrix} \cdot J \begin{pmatrix}
		x_1 \\ x_2
	\end{pmatrix}  (y_1,y_2).
	\end{align*}
	Let $e:=(\RI_2,0)$ and $f:=(0,\RI_2)$. In this case, we have $\RU(\BV,\BB_{\BV})=\Sp_{2n}$ and $\RU(\BW,\BB_{\BW})=\RO_4$. To simplify the notation, let $\{e_1,f_1,\cdots,e_n,f_n\}$ denote the standard basis of $\RF^{2n}$. Specifically, for $1\leq i\leq n$, let $e_i$ and $f_i$ be the column vectors with a $1$ in the $(2i-1)$-th and $2i$-th positions, respectively, and zero elsewhere.
		
	We have $\RG=\RU(\BV,\BB_{\BV})=\Sp_{2n}$, and $\RU(\BW,\BB_{\BW})=\RO_4$, and  
\[
\RX=\Sp_{2n}/\Sp_{2n-2}\cong\{ v\in\BV\mid \tr\; \BB_{\BV}(v,v)=2   \}\subset\BV:g\mapsto g v_0,
\]
where $v_0=\begin{pmatrix}
	0 \\ \vdots  \\ 0 \\ \RI_2
\end{pmatrix}\in\BD^n$. Let $\overline{\RX}=\Sp_{2n}/\Sp_{2n-2}\times\Sp_2$, which classifies the $2$-dimensional non-degerate subspaces $\RV_2$ of $\RF^{2n}$. Let $\Phi_0$ be the characteristic function of $\BY(\Fo_{\RF})$, 

\subsection{Matching of unit elements}

\begin{lem}\label{lem_volc}
	The push-forward of $\CI(\Res(\Phi_0))(x)\ud x$ in $\CS(\overline{\FX})$ is 
	\[
	\frac{1}{\zeta_{\RF}(2n)}\BL_{\RX}^{\circ}=\frac{\#\RX(\kappa_{\RF})}{q^{\dim_{\RX}}}\BL_{\RX}^{\circ}.
	\]
\end{lem}

\begin{proof}
	Replace $n$ by $2n$ in Lemma \ref{lem_volb}, and apply \cite[Theorem 8.3.1]{Igu00} again, we see
	\begin{align*}
		\int_{\RX=\BV_1}\Phi_0(v)\omega_1(v)=1-q^{-2n}.
	\end{align*}
	On the other hand, according to \cite[Chapter 8]{Tay92}, we have
	\begin{align*}
		\#\Sp_{2n}(\kappa_{\RF})=q^{n^2}\prod_{i=1}^n(q^{2i}-1),
	\end{align*}
	hence
	\begin{align*}
		\tau(\RX)=\frac{\#\RX(\kappa_{\RF})}{q^{\dim_{\RX}}}=\frac{q^{n^2}\prod_{i=1}^n(q^{2i}-1)}{q^{(n-1)^2}\prod_{i=1}^{n-1}(q^{2i}-1)}\frac{1}{q^{n(2n+1)-(n-1)(2n-1)}}=1-q^{-2n}.
	\end{align*}
\end{proof}

\begin{lem}\label{lem_match_c}
	\[
	\overline{\Omega}(\Phi_0)=\frac{1}{\zeta_{\RF}(2)\zeta_{\RF}(2n-2)}\Phi_{L\left(\std,\frac{n-1}{2}\right)\left(\std,\frac{n-3}{2}\right)}.
	\]
\end{lem}

\begin{proof}
	
Using the Iwasawa decompositon of $\SL_2$, write $g=\begin{pmatrix}
	1 & x \\ 0  & 1
\end{pmatrix}\begin{pmatrix}
	m & 0 \\ 0 & m^{-1} 
\end{pmatrix}k$ with $k\in\SL_2(\Fo_{\RF})$. But be careful that, according to our choices of measures, we have
\begin{align*}
	\ud g=\frac{1-q^{-2}}{1-q^{-1}}\ud x|m|^{-2}\frac{\ud m}{|m|}\ud k,
\end{align*}
where $\ud k$ is the normalized measure such that 
\begin{align*}
	\int_{\RK}1\ud k=1.
\end{align*}
Then we have
\begin{align*}
	\overline{\Omega}(\Phi_0)\left(  \begin{pmatrix}
		\xi & \\ & 1
	\end{pmatrix} \right)&= \frac{1-q^{-2}}{1-q^{-1}}  \int_{\SL_2}\Omega(\Phi_0)\left(\begin{pmatrix}
		\xi & \\  & 1 
	\end{pmatrix}  \begin{pmatrix}
		1 & x \\ & 1
	\end{pmatrix}     \begin{pmatrix}
		m & \\ & m^{-1}
	\end{pmatrix} k \right)\ud x |m|^{-2}\frac{\ud m}{|m|}\ud k\\
	&=\frac{1-q^{-2}}{1-q^{-1}}\int_{1\leq |m|\leq |\xi|^{-1}}\int_{|x|\leq |m\xi^{-1}|}\ud x|m|^{-2}\frac{\ud m}{|m|}\\
	&=\frac{1-q^{-2}}{1-q^{-1}}|\xi|^{n-1}(1-q^{-1})\frac{1-q^{-1}|\xi|}{1-q^{-1}}\\
	&=\frac{1}{\zeta_{\RF}(2n-2)\zeta_{\RF}(2)}\Phi_{L\left(\std,\frac{n-1}{2}\right)\left(\std,\frac{n-3}{2}\right)}.
\end{align*}
\end{proof}

Combine Lemma \ref{lem_volc} and Lemma \ref{lem_match_c} together, we get
\begin{thm}\label{thm_c}
	Let $\BL_{\RX}^{\circ}\in\CS(\Sp_{2n-2}\times\SL_2\backslash\Sp_{2n}/\Sp_{2n-2}\times\Sp_2)$ be the push-forward of the probability measure which is equal to the characteristic function of $\overline{\RX}(\Fo_{\RF})\times\overline{\RX}(\Fo_{\RF})$ times an invariant measure, then we have
	\[
	\CT^{-1}\left(\BL_{\RX}^{\circ}\right)=\frac{q^{\dim\overline{\RX}}}{\#\overline{\RX}(\kappa_{\RF})\zeta_{\RF}(2)\zeta_{\RF}(s_{\RX})}f_{L_{\RX}}=\frac{\zeta_{\RF}(2n)}{\zeta_{\RF}(2)^2\zeta_{\RF}(2n-2)}f_{L\left(\std,\frac{n-1}{2}\right)\left(\std,\frac{n-3}{2}\right)}.
	\]
\end{thm}

\subsection{Matching of Hecke elements}

Let $\RB$ be the Borel subgroup of $\Sp_{2n}$ fixing the isotropic flag
\[
\langle e_1\rangle \subset \langle e_1,e_2\rangle \subset \cdots \subset \langle e_1,\cdots,e_n\rangle,
\]
and $\RA$ the maximal torus of $\RB$ fixing each line $\langle e_i \rangle $ for $1\leq i\leq n$.

\begin{lem}
	Let $\overline{\RX}^{\circ}$ be the subset of $\overline{\RX}$ consisting of $2$-dimensional non-degenerate subspaces $\RV_2\subset\RF^{2n}$ such that $\RV_2\cap \langle e_1,e_2\rangle ^{\perp}=\{0\}$. Then $\overline{\RX}^{\circ}$ is an open $\RB$-orbit in $\overline{\RX}$. \end{lem}

\begin{proof}
	Let $\RV_2\in\overline{\RX}^{\circ}$ and is generated by \[\begin{pmatrix}
		x_1 \\ \vdots \\ x_{2n}\end{pmatrix} ,\;\begin{pmatrix} y_1 \\ \vdots \\ y_n 
	\end{pmatrix}\]
	with
	\[
	x_1y_n+\cdots x_ny_{n+1}-x_{n+1}y_n-\cdots -x_{2n}y_1=1
	\] under the ordered basis $\{e_1,\cdots,e_n,f_n,\cdots,f_1 \}$. As $\RV_2\cap\{e_1,e_2\}^{\perp}=\{0\}$, $\begin{pmatrix}
		x_{2n-1} & y_{2n-1} \\ x_{2n} & y_{2n}
	\end{pmatrix}$ is non-singular, and we may assume
	\[
	\begin{pmatrix}
		x_{2n-1} & y_{2n-1} \\ x_{2n} & y_{2n}
	\end{pmatrix}=\RI_2.
	\]
    Also from here one can see that $\overline{\RX}^{\circ}$ is open in $\overline{\RX}$.

First observe that
		\[
	\begin{pmatrix}
		s_n \left(\tensor*[^t]{  \begin{pmatrix} 1 & &  & -x_{n+1} & -y_{n+1} \\  & 1 & & -x_{n+2} & -y_{n+2} \\ & & \rotatebox{45}{$\vdots$}  & \vdots & \vdots  \\ & & & 1 & 0 \\ & & &  0& 1   \end{pmatrix}  }{^{-1}}  \right)
 s_n & 0_n \\ 0_n &   \begin{pmatrix} 1 & &  & -x_{n+1} & -y_{n+1} \\  & 1 & & -x_{n+2} & -y_{n+2} \\ & & \rotatebox{45}{$\vdots$}  & \vdots & \vdots  \\ & & & 1 & 0 \\ & & &  0& 1   \end{pmatrix} 	\end{pmatrix}\begin{pmatrix}
 	x_1 \\ x_2 \\ \vdots \\ x_n\\ x_{n+1} \\ x_{n+2} \\ \vdots \\ 1 \\ 0 
 \end{pmatrix}=\begin{pmatrix}
 	x_1^{\prime}  \\ x_2^{\prime} \\ \vdots \\ x_n^{\prime} \\ 0 \\ 0 \\ \vdots \\ 1 \\ 0
 \end{pmatrix},
	\]
	and 
\[
	\begin{pmatrix}
		s_n \left(\tensor*[^t]{  \begin{pmatrix} 1 & &  & -x_{n+1} & -y_{n+1} \\  & 1 & & -x_{n+2} & -y_{n+2} \\ & & \rotatebox{45}{$\vdots$}  & \vdots & \vdots  \\ & & & 1 & 0 \\ & & &  0& 1   \end{pmatrix}  }{^{-1}}  \right)
 s_n & 0_n \\ 0_n &   \begin{pmatrix} 1 & &  & -x_{n+1} & -y_{n+1} \\  & 1 & & -x_{n+2} & -y_{n+2} \\ & & \rotatebox{45}{$\vdots$}  & \vdots & \vdots  \\ & & & 1 & 0 \\ & & &  0& 1   \end{pmatrix} 	\end{pmatrix}\begin{pmatrix}
 	y_1 \\ y_2 \\ \vdots \\ y_n\\ y_{n+1} \\ y_{n+2} \\ \vdots \\ 0 \\ 1 
 \end{pmatrix}=\begin{pmatrix}
 	y_1^{\prime}  \\ y_2^{\prime} \\ \vdots \\ y_n^{\prime} \\ 0 \\ 0 \\ \vdots \\ 0 \\ 1 
 \end{pmatrix}
 \]
 for some $x_1^{\prime},\cdots,x_n^{\prime},y_1^{\prime},\cdots,y_n^{\prime}\in\RF$.
 Then we can take 
 \[
	\begin{pmatrix}
		\RI_n   &     \begin{pmatrix} -y_n^{\prime} & -y_{n-1}^{\prime} & \cdots & -y_2^{\prime} & -y_1^{\prime }\\-x_n^{\prime} & -x_{n-1}^{\prime} & \cdots & -x_2^{\prime} & -y_2^{\prime} \\ 0 & 0 & \cdots & -x_3^{\prime} & -y_3^{\prime}  \\ \vdots &\vdots & \rotatebox{45}{$\vdots$} & \vdots&  \vdots \\ 0&0 &\cdots  & -x_{n-1}^{\prime} & -y_{n-1}^{\prime} \\ 0 & 0  & 0 & -x_n^{\prime} & -y_n^{\prime}\end{pmatrix}      \\ 0_n &  \RI_n \end{pmatrix}\begin{pmatrix}
 	x_1^{\prime}  \\ x_2^{\prime} \\ \vdots \\ x_n^{\prime} \\ 0  \\ 0 \\ \vdots \\  1\\ 0 
 \end{pmatrix}=\begin{pmatrix}
 	x_1^{\prime}-y_2^{\prime} \\ 0 \\ \vdots \\ 0\\ 0 \\ 0 \\ \vdots \\ 1 \\ 0
 \end{pmatrix},
\]
and
 \[
	\begin{pmatrix}
		\RI_n   &     \begin{pmatrix} -y_n^{\prime} & -y_{n-1}^{\prime} & \cdots & -y_2^{\prime} & -y_1^{\prime }\\-x_n^{\prime} & -x_{n-1}^{\prime} & \cdots & -x_2^{\prime} & -y_2^{\prime} \\ 0 & 0 & \cdots & -x_3^{\prime} & -y_3^{\prime}  \\\vdots  &\vdots &  \rotatebox{45}{$\vdots$} & \vdots&  \vdots \\ 0&0 &\cdots  & -x_{n-1}^{\prime} & -y_{n-1}^{\prime} \\ 0 & 0  & 0 & -x_n^{\prime} & -y_n^{\prime}\end{pmatrix}      \\ 0_n &  \RI_n \end{pmatrix}\begin{pmatrix}
 	y_1^{\prime}  \\ y_2^{\prime} \\ \vdots \\ y_n^{\prime} \\ 0  \\ 0 \\ \vdots \\  0 \\ 1 
 \end{pmatrix}=\begin{pmatrix}
 	0 \\ 0 \\ \vdots \\ 0\\ 0 \\ 0 \\ \vdots \\ 0 \\ 1
 \end{pmatrix}.
\]
Moreover, according to the choice of the vectors at the very beginning, actually $x_1^{\prime}-y_2^{\prime}=1$. Therefore $\overline{\RX}^{\circ}$ is an $\RB$-orbit with $\langle e_1+f_2,f_1\rangle \in\overline{\RX}^{\circ}$.
\end{proof}

\begin{cor}
	$\RP(\overline{\RX})$ is the parabolic subgroup fixing the isotropic flag
\[
\langle e_1,e_2\rangle .
\] 
Then its Levi 
\[
\RL(\overline{\RX})=\left\{\begin{pmatrix}
 g & &   \\ &  h & \\ & &  s_2 {^tg^{-1}} s_2   
\end{pmatrix} \mid g \in\GL_2 ,\;h \in\Sp_{2n-4} \right\}\cong\GL_2\times\Sp_{2n-2}
\]
Therefore,
\[
\delta_{(\overline{\RX})}^{\frac{1}{2}} \left(  \begin{pmatrix}
	t_1 & & & & & \\ & \rotatebox{45}{$\vdots$} & & & & \\ & & t_n & & & \\ & & & t_n^{-1} & & \\ & & & & \rotatebox{45}{$\vdots$} & \\ & & & & & t_1^{-1} 
\end{pmatrix}  \right)=|t_1|^{\frac{1}{2}}|t_2|^{-\frac{1}{2}}|t_3|^{n-2}\cdots |t_n|^1.
\]
Moreover, the quotient map $\RA\rightarrow\RA_{\overline{\RX}}\cong\BG_m$ is given by 
\[
\begin{pmatrix}
	t_1 & & & & & \\ & \rotatebox{45}{$\vdots$} & & & & \\ & & t_n & & & \\ & & & t_n^{-1} & & \\ & & & & \rotatebox{45}{$\vdots$} & \\ & & & & & t_1^{-1} 
\end{pmatrix} \mapsto t_1t_2,
\]
then after the Satake transform, the normalized map $\BC[\RA^{\vee}]^{\BW_{\RG}}\rightarrow\BC[\RA_{\overline{\RX}}^{\vee}]^{\BW_{\overline{\RX}}}$ is given by
\begin{align}\label{hkdualc}
	\BC[x_1^{\pm},\cdots,x_n^{\pm}]^{\BW_{\Sp_{2n}}}\rightarrow\BC[x^{\pm}]^{\BZ/2\BZ}:x_1,x_2\mapsto x ,\;x_j\mapsto q^{j-n-1}\;\mathrm{for}\;3\leq j\leq n.
\end{align}

\end{cor}

Write $\Sigma:=\{v\in\BV\mid \BB_{\BV}(v,v)=0\}$ for the nilpotent cone.

\begin{lem}
	\begin{align*}
	\Sigma=\left\{   \begin{pmatrix}
		v_1 \\ \vdots \\ v_{2n}
	\end{pmatrix} \otimes (1,0)+\begin{pmatrix}
		w_1 \\ \vdots \\ w_{2n}
	\end{pmatrix} \otimes (0,1)  \mid(v_1,\cdots,v_{2n})\cdot  \begin{pmatrix}
		J &  & \\ & \ddots  & \\ & & J
	\end{pmatrix}   \cdot\begin{pmatrix}
		w_1 \\ \vdots \\w_{2n} 
	\end{pmatrix} =0    \right\},
\end{align*}
and $\Sigma$ admits an open $\Sp_{2n}(\RF)\times\SL_2(\RF)$-orbit $\Sigma^{\circ}$.
\end{lem}

\begin{proof}
	The first statement is clear. As for the second assertion, note that for any orthogonal non-zero $v,w\in\RF^{2n}$, as long as they are not parallel to each other, there is some $g\in\Sp_{2n}$ such that $ge_1=v$ and $ge_2=w$.
\end{proof}

 \begin{lem}
	Let $\Phi\in\CF(\BY)^{ \PGL_2(\Fo_{\RF})}$, and 
	\[  \int_{\SL_2} \omega_{\psi}(g)\omega_{\psi}(x)\Phi(v)\ud x=0,\;\forall v\in\Sigma^{\circ} ,\] 
	for all $g\in\PGL_2(\RF)$, then $ \displaystyle{ \int_{\SL_2}\omega_{\psi}(x)\Phi(v)=0}$ for $v\in\BV^{\circ}$, where $\BV^{\circ}$ is the set of $ v\otimes(1,0)+w\otimes(0,1)$ such that $v$ and $w$ are not parallel.
	\end{lem}
\begin{proof}

Consider the $\PGL_2$-equivariant map
\[
\CF(\BY)\rightarrow\CC^{\infty}(\BV^{\circ}):\Phi\mapsto \left(v\mapsto  \int_{\SL_2} \omega_{\psi}(g)\Phi(v)\ud g    \right).
\]
Let $\FV$ be the image of $\CF(\BY)$. Note that for any $\begin{pmatrix}
	1 & n \\ & 1
\end{pmatrix}\in\PGL_2$, similarly for any $\Phi\in\CC^{\infty}(\BV^{\circ})$,
\begin{align*}
\omega_{\psi}\left(  \begin{pmatrix}
	1 & n \\ & 1
\end{pmatrix} \right)\Phi(v)-\Phi(v)=\psi\left(   \frac{ n\cdot \tau(\BB_{\BV}(v,v))   }{2} \right)\Phi(v),
\end{align*}
the dual of the Jacquet module of $\CC^{\infty}(\BV^{\circ})$ is the space of compact support distributions supported on $\BV^{\sigma}\cap\Sigma=\Sigma^{\circ}$. As the Jacquet functor is exact, restriction to $\Sigma^{\circ}$ is the Jacquet functor of $\FV$. Therefore, according to the condition, the $\PGL_2$-module generated by $\Phi$ inside $\FV$ is cuspidal, hence the only $\PGL_2(\Fo_{\RF})$-invariant element is $0$.

\end{proof}

Now let $\sigma\in\BC$ be such that $\Re(\sigma)>1$, and consider the intertwining operators
\begin{align*}
\RZ_{\sigma}:\CF(\BY)&\rightarrow  \CC^{\infty}(\Sp_{2n}\times\PGL_2)      \\\Phi &\mapsto \left(  (g_1,g_2) \mapsto  \int_{\SL_2\times \RF} \omega_{\psi}(g_1,g_2)^{-1}\Phi ( ( ae_1 \otimes (1,0) +e_2 \otimes (0,1)) x) \ud x |a|^{\sigma} \frac{\ud a}{|a|}\right),
\end{align*}
then

\begin{prp}
	When $\Re(\sigma)>1$,
	\begin{itemize}
		\item [(1)] Let $\Fs=(s_1,s_2,\cdots,s_n)\in X^*(\RA)\otimes_{\BZ}\BC=\BC^n$, write 
\[
\chi(\Fs): \RA\rightarrow\BC:  \begin{pmatrix}
	t_1 & & & & & \\ & \rotatebox{45}{$\vdots$} & & & & \\ & & t_n & & & \\ & & & t_n^{-1} & & \\ & & & & \rotatebox{45}{$\vdots$} & \\ & & & & & t_1^{-1} 
\end{pmatrix} \mapsto \prod_{i=1}^n|t_i|^{s_i}
\]
and $\RI(\chi(\Fs))$ the correponding normalized induced representation of $\Sp_{2n}$. Then $\RZ_{\sigma}$ factors through $\RI(\chi(n-\sigma,n-1-\sigma,n-2,\cdots,1) )\otimes\RI(\chi(\sigma-n+\frac{1}{2}))$ as representations of $\Sp_{2n}\times\PGL_2$.
\item [(2)] Write $\RK$ for the maximal compact subgroup $\Sp_{2n}(\Fo_{\RF})$. Let \[\lambda_{\RX}:\CH(\Sp_{2n},\RK)\rightarrow\CH(\PGL_2,\PGL_2(\Fo_{\RF}))\] be the morphism of the Hecke algebra corresponding to (\ref{hkdualc}). Let $\Phi_0$ be the characteristic function of $\BY(\Fo_{\RF})$, then we have 
	\[
\int_{\SL_2}(\omega_{\psi}(h^{\vee})-\omega_{\psi}(\lambda_{\RX}(h)))\omega_{\psi}(x) \Phi_0(v)\ud v=0.
\]
for all $v\in\BV^{\circ}$ and $h\in\CH(\Sp_{2n},\RK)$. 
	\end{itemize}

\end{prp}

\begin{proof}
	$(1)$ follows immediately from the computation, and $(2)$ is analogous to the arguments presented in the previous sections.
\end{proof}

\begin{thm}
	Conjecture \ref{cnj_flhk} is true for $\RX$ of type $\RC_n$.
\end{thm}

\begin{proof}
	Consider the function $\overline{\Omega}(\Phi_0) $ on $\PGL_2$. As noted before,
	\[
	\int_{\RF} \overline{\Omega}(\Phi_0)\left(\Bw \begin{pmatrix}
		\xi & \\ & 1
	\end{pmatrix}  n\right)  \psi^{-1}(n)\ud n=\int_{\SL_2}\widehat{\BT}(\CI\circ\res(\Phi_0))(\xi,g)\ud g.
	\]
	Then the proof of the theorem follows by an argument analogous to that employed in the previous cases.
	\end{proof}

\section{Fundamental lemma of type $B_{n}$}\label{sec_b}

In this section, we present two distinct approaches to proving the fundamental lemma. The first is a direct computational method, where we demonstrate that
\[
\widetilde{\BT}(\BL_{\RX}^{\circ}) =\frac{q^{\dim\RX}}{\#\RX(\kappa_{\RF})\zeta_{\RF}(2)\zeta_{\RF}(s_{\RX})}\widetilde{\BT}\circ\CT\left(f_{L\left(\std,n-\frac{1}{2}\right)L\left(\std,\frac{1}{2}\right)}\right)
\]
without recourse to the metaplectic cover $\Mp_2$. We establish this identity for the unit element in the initial section; while the computation is involved, this method offers a potential path for a generalization to arbitrary Hecke elements.

The second approach utilizes a diagrammatic argument to establish the commutativity of the diagram below. This commutativity, in conjunction with the fundamental lemmas for $\BT$, $\overline{\RU}$, and $\BJ$, provides a structural proof of the fundamental lemma for the transfer operator $\CT$. This aligns the explicit computations required for our first approach with the computations in \cite{Jac87}.

\begin{equation}\label{diag_b}
\begin{tikzcd}
\CS(\SO_{2n}\backslash\SO_{2n+1}/\SO_{2n}) \arrow[rr, "\widetilde{\BT}"] &                                                                                                                                                           & {\CS_{\BV}(\RT)} \\
                                                             & {\CS^-_{L\left(\st,n-\frac{1}{2}  \right)L\left(\st, \frac{1}{2}\right)}(\RN,\psi\backslash \PGL_2/\RN,\psi)} \arrow[lu, "\CT"] \arrow[ru, "\widetilde{\BT}\circ\CT"] &                                              \\
                                                             & {\CS^-_{L\left(\st,n-\frac{1}{2}\right)}(\RA\backslash\PGL_2/\RN,\psi)} \arrow[u, "\overline{\RU}"] \arrow[ruu, "\BJ", dotted, bend right]                        &                                             
\end{tikzcd}
\end{equation}

We provide a brief overview of the operators appearing in the diagram above. 
\begin{itemize}
	\item $\widetilde{\BT}$ denotes the transfer operator arising from the Weil representation,
	\item $\CT$ is the transfer operator defined in \cite{Sak21},
	\item $\overline{\RU}$ corresponds to the unfolding process, but with non-standard test measures to ensure compatibility with the space of orbital integrals,
	\item $\BJ$ represents the comparison established by H. Jacquet in \cite{Jac87}, and we extend this comparison from standard test measures to the non-standard unit element required for our current setting.
\end{itemize}

Section \ref{ssec_bunit} is devoted to the direct computational proof of the fundamental lemma. The subsequent sections develop the structure proof by establishing the commutativity of the diagram relating the transfer operators, the Weil representations, and the Jacquet comparison.

In this section, let $\BD=\RF$, $\dim_{\RF}\BV$ is odd, suppose there is an ordered basis \[\{e_1,\cdots,e_n,w,f_n,\cdots,f_1\}\] such that 
	\[
	\BB_{\BV}(e_i,e_j)=\BB_{\BV}(f_i,f_j)=0,\;\mathrm{and}\;\BB_{\BV}(e_i,f_j)=\delta_{ij}\; \forall \;1\leq i,j\leq n,
	\]
	\[
	\BB_{\BV}(e_i,w)=\BB_{\BV}(f_i,w)=0,\;\forall\; 1\leq i\leq n,
	\]
	and
	\[
	\BB_{\BV}(w,w)=2.
	\]
	Under the isomorphism $\BV\cong\RF^{2n+1}$ under the above basis, we have
		\[
	\BB_{\BV}:\RF^{2n+1}\times\RF^{2n+1}\rightarrow\RF:\left(\begin{pmatrix}
		x_1 \\ \vdots \\ x_{2n+1}
	\end{pmatrix}  ,    \begin{pmatrix}
		y_1\\ \vdots \\ y_{2n+1}
	\end{pmatrix}  \right)\mapsto\sum_{i=1}^{2n+1}x_iy_{2n+2-i}+x_{n+1}y_{n+1}.
	\]
    We take $\BW$ to be a $2$-dimensional symplectic space over $\RF$, and fix the ordered basis $\{e,f\}$ of $\BW$ such that $\BB_{\BW}(e,f)=1$. Under the isomorphism $\BW\cong\RF^2$ under this fixed basis, we have
	\[
	\BB_{\BW}:\RF^2\times\RF^2\rightarrow\RF:\left( \begin{pmatrix}
		x_1 \\ x_2
	\end{pmatrix} ,   \begin{pmatrix}
		y_1 \\ y_2
	\end{pmatrix}  \right)\mapsto x_1y_2-x_2y_1.
	\]
	
	In this case, $\RU(\BV,\BB_{\BV})=\SO(\BV)$ and $\RU(\BW,\BB_{\BW})=\SL_2$. Moreover, by the composition
\[
\Mp_2\rightarrow\widetilde{\SL_2}\rightarrow\Mp(\BV\otimes_{\BD}\BW)^{\BC^1}_{\psi,\BX},
\]
we obtain a Weil representation of $\Mp_2$. 
Let $\Phi_0$ be the characteristic function of $\BY(\Fo_{\RF})$. 

\subsection{Matching of unit elements}\label{ssec_bunit}

To compare $f_{L\left(\std,n-\frac{1}{2}\right)L\left(\std,\frac{1}{2}\right)}$ with $\BL_{\RX}^{\circ}$, our strategy is to compute $\widetilde{\BT}\circ\CT\left(f_{L\left(\std,n-\frac{1}{2}\right)L\left(\std,\frac{1}{2}\right)}\right)$ and $\widetilde{\BT}(\BL_{\RX}^{\circ})$ separately. Let us start with computing $\widetilde{\BT}(\BL_{\RX}^{\circ})$.

\begin{lem}\label{lem_b><1}

There is some $k$ large enough, such that
\begin{itemize}
	\item [(1)] when $|\zeta|\geq 1$, 	\begin{align*}&\int_{\RN}^*\omega_{\psi}\left(\Bw   \begin{pmatrix}
			\zeta & 0 \\ 0 & \zeta^{-1}
		\end{pmatrix}u\right)\Phi_0(v_0)\psi^{-1}(u)\ud u\\
&=\gamma(\zeta,\psi)|\zeta|^{-\frac{2n+1}{2}}\left(    \int_{|u|\leq 1}\psi^{-1}(u)\int_{z\in\Fo_{\RF}}\psi(uz^2)\ud z  \ud u  + \int_{1<|u|\leq q^k}\psi^{-1}(u)|u|^{-n} \int_{\Fo_{\RF}}\psi(uz^2)\ud z   \ud u   \right) .
	\end{align*}
	\item [(2)] when $|\zeta|<1$, 
	\begin{align*}
	&\int_{\RN}^*\omega_{\psi}\left(\Bw   \begin{pmatrix}
			\zeta & 0 \\ 0 & \zeta^{-1}
		\end{pmatrix}u\right)\Phi_0(v_0)\psi^{-1}(u)\ud u\\
&=\gamma(\zeta,\psi)|\zeta|^{-\frac{2n+1}{2}}\int_{|\zeta|^{-1}\leq |u|\leq q^k}\psi^{-1}\left( \frac{1}{u \zeta^2}+u   \right)|u|^{-n}\int_{\Fo_{\RF}}\psi(u z^2)\ud z \ud u.
\end{align*}
	\end{itemize}
	\end{lem}

\begin{proof}
	\begin{align*}
	 & \int_{\Fp_{\RF}^{-k}} \omega_{\psi}\left( \Bw\begin{pmatrix}
	 	\zeta & \\  & \zeta^{-1}
	 \end{pmatrix} \begin{pmatrix}
			1 & u \\ 0 & 1
		\end{pmatrix} \right)\Phi_0(v_0)\psi^{-1}(u )\ud u\\
		&=\gamma(\zeta,\psi)\int_{\Fp_{\RF}^{-k}}\int_{\RF^{2n+1}} \omega_{\psi} \left(\begin{pmatrix}
	\zeta & 0 \\ 0 & \zeta^{-1}
\end{pmatrix}  \begin{pmatrix}
		1 & u \\ 0 & 1
	\end{pmatrix} \right)\Phi_0 ( (x_1,\cdots,x_n,z,y_1,\cdots,y_n) )\\
	&\cdot\psi^{-1}(x_n+y_n)     \ud x_1\cdots \ud y_n \ud z \cdot \psi^{-1}(u)\ud u \\
	&=\gamma(\zeta,\psi)\int_{\Fp_{\RF}^{-k}}\int_{\RF^{2n+1}}\left|\zeta \right|^{\frac{2n+1}{2}}\psi( u\zeta^2(x_1y_1+\cdots x_ny_n+z^2)   )\Phi_0((\zeta x_1,\cdots,\zeta x_n,\zeta z, \zeta y_1,\cdots,\zeta y_n))\\
	&\cdot\psi^{-1}(x_n+y_n)     \ud x_1\cdots \ud y_n \ud z \cdot \psi^{-1}(u)\ud u \\
	&=\gamma(\zeta,\psi)\int_{\Fp_{\RF}^{-k}}\int_{\RF^{2n+1}}\left| \zeta  \right|^{-\frac{2n+1}{2}}\psi( u (x_1y_1+\cdots x_ny_n+z^2)   )\Phi_0((x_1,\cdots,x_n,z,y_1,\cdots,y_n))\\
	&\cdot\psi^{-1}(\zeta ^{-1}x_n+\zeta^{-1} y_n)     \ud x_1\cdots \ud y_n \cdot \psi^{-1}(u)\ud u .\\
\end{align*}
Note that we have
\begin{align*}
	&\int_{\RF^n}\Phi_0((x_1,\cdots,x_n,z,y_1,\cdots,y_n))\psi( u (x_1y_1+\cdots x_ny_n)   )\psi^{-1}(\zeta^{-1}x_n) \ud x_1\cdots\ud x_n\\
	&=\Phi_0(uy_1,\cdots,uy_{n-1},uy_n-\zeta^{-1},z,y_1,\cdots,y_n),
\end{align*}
then integrate over the other $\RF^n$, we obtain
\begin{align*}
	&\int_{\RF^n}\Phi_0(u y_1,\cdots,uy_{n-1},uy_n-\zeta^{-1},z,y_1,\cdots,y_n)\psi^{-1}(\zeta^{-1}y_n)\ud y_1\cdots \ud y_n\\
	&=1_{\Fo_{\RF}}(z)\int_{\RF^{n-1}}1_{\Fo_{\RF}^{n-1}}(u y_1,\cdots,u y_{n-1})1_{\Fo_{\RF}^{n-1}}(y_1,\cdots,y_{n-1}) \ud y_1\cdots \ud y_{n-1}\\
	&\cdot\int_{\RF}1_{\Fo_{\RF}}(u y_n-\zeta^{-1})1_{\Fo_{\RF}}(y_n)\psi^{-1}(\zeta^{-1}y_n)\ud y_n.
	\end{align*}
It is easy to see that
\begin{align*}
	\int_{\RF}1_{\Fo_{\RF}}(uy)1_{\Fo_{\RF}}(y)\ud y=\left\{ \begin{array}{rcl}
1 & \mbox{for}
& |u|\leq 1; \\ |u |^{-1} & \mbox{for} & |u|>1.
\end{array}\right.
\end{align*}
Moreover, when $u\neq 0$,
\begin{align*}
	&\int_{\RF}1_{\Fo_{\RF}}(uy-\zeta^{-1})1_{\Fo_{\RF}}(y)\psi^{-1}(\zeta^{-1}y)\ud y\\
	&=\int_{\RF}1_{\Fo_{\RF}\cap (\zeta^{-1}u^{-1}+u^{-1}\Fo_{\RF})}(y) \psi^{-1}(t_2y)\ud y\\
	&=\BF_{\psi}\left(1_{\Fo_{\RF}\cap (\zeta ^{-1}u^{-1}+u^{-1}\Fo_{\RF})}\right)(\zeta^{-1}),
\end{align*}
where $1_{\Fo_{\RF}\cap (\zeta^{-1}u^{-1}+u^{-1}\Fo_{\RF})}$ is the characteristic function of $\Fo_{\RF}\cap (\zeta^{-1}u^{-1}+u^{-1}\Fo_{\RF})$. Note that
\begin{align*}
	\Fo_{\RF}\cap (\zeta^{-1}u^{-1}+u^{-1}\Fo_{\RF})=\left\{ \begin{array}{rcl}
\Fo_{\RF} & \mbox{for}
& |u|\leq 1,\;|\zeta|\geq 1;  \\
\emptyset  & \mbox{for} & |u|\leq 1,\;|\zeta|< 1;\\
 \zeta^{-1}u^{-1}+u^{-1}\Fo_{\RF} & \mbox{for} & |u|>1, \;|\zeta|\geq |u|^{-1};\\
 \emptyset & \mbox{for} & |u|>1,\;|\zeta|<|u|^{-1}.
\end{array}\right.
\end{align*}
Then we have its Fourier transform 
\begin{align*}
	\BF_{\psi}\left(1_{\Fo_{\RF}\cap (\zeta^{-1}u^{-1}+u^{-1}\Fo_{\RF})}\right)(\zeta^{-1})=\left\{ \begin{array}{rcl}
1 & \mbox{for}
& |u|\leq 1,\;|\zeta|\geq 1;  \\ 0  & \mbox{for} & |u|\leq 1,\;|\zeta|<1;\\
 \psi(-\zeta^{-2} u^{-1})|u|^{-1}1_{u\Fo_{\RF}}(\zeta^{-1}) & \mbox{for} & |u|>1, \;|\zeta|\geq |u|^{-1}; \\
0 & \mbox{for} & |u|>1,\;|\zeta|<|u|^{-1}.
\end{array}\right.
\end{align*}
 Therefore, when $|u|\leq 1$, we have the integrand function is
\begin{align*}
	\left\{ \begin{array}{rcl}
 \left|\zeta\right|^{-\frac{2n+1}{2}} \psi^{-1}(u )  & \mbox{for}
\;|\zeta|\geq 1; \\ 0  & \mbox{otherwise}.
\end{array}\right.
\end{align*}
And when $|u|>1$, the integrand function is 
\begin{align*}
	\left\{ \begin{array}{rcl}
|\zeta|^{-\frac{2n+1}{2}} \psi^{-1}\left(   \frac{1}{u \zeta^2} + u\right)|u|^{-n}  & \mbox{for}
\; u \geq \max\{q,|\zeta|^{-1}\} \\ 0  & \mbox{otherwise}. 
\end{array}\right.
\end{align*}
Then, when $k$ is large enough, the final integral becomes:
\begin{itemize}
	\item [(1)] when $|\zeta|\geq 1$, 	\[
\gamma(\zeta,\psi)|\zeta|^{-\frac{2n+1}{2}}\left(    \int_{|u|\leq 1}\psi^{-1}(u)\int_{z\in\Fo_{\RF}}\psi(uz^2)\ud z  \ud u  + \int_{1<|u|\leq q^k}\psi^{-1}(u)|u|^{-n} \int_{\Fo_{\RF}}\psi(uz^2)\ud z   \ud u   \right) .
	\]
	\item [(2)] when $|\zeta|<1$, 
	\begin{align*}
\gamma(\zeta,\psi)|\zeta|^{-\frac{2n+1}{2}}\int_{|\zeta|^{-1}\leq |u|\leq q^k}\psi^{-1}\left( \frac{1}{u \zeta^2}+u   \right)|u|^{-n}\int_{\Fo_{\RF}}\psi(u z^2)\ud z \ud u.
\end{align*}
	\end{itemize}

\end{proof}

				\begin{lem}\label{lem_gsum}
				\[
		\int_{\Fo_{\RF}}\psi(uz^2)\ud z=\left\{ \begin{array}{rcl}
1  & \mbox{for} & |u|\leq  1; \\
|u|^{-\frac{1}{2}} & \mbox{for }& |u|>1,\;\log_q|u|\equiv 0 \mod 2;  \\
q^{-\frac{1}{2}}|u|^{-\frac{1}{2}}\sum_{a\in\kappa_{\RF}}\psi\left(  \frac{\ac(u)\cdot a^2}{\varpi_{\RF}} \right)
& \mbox{for} & |u|>1,\; \log_q|u|\equiv 1\mod 2.
\end{array}\right.
		\]
\end{lem}
		
		\begin{proof}
			The case that $|u|\leq 1$ is obvious, and we only have to consider the case that $|u|>1$. In the case that $|u|=q^{2r}$ for some $r>0$, we have
	\begin{align*}
		\int_{\Fo_{\RF}}\psi(uz^2)\ud z&=\sum_{a\in\Fo_{\RF}/\Fp_{\RF}^r}\int_{\Fp_{\RF}^r}\psi(u(a+z)^2)\ud z\\
		&=\sum_{a\in\Fo_{\RF}/\Fp_{\RF}^r}\int_{\Fp_{\RF}^r}\psi(a^2u+2azu+z^2u)\ud z\\
		&=\sum_{a\in\Fo_{\RF}/\Fp_{\RF}^r} \psi(a^2u)\int_{\Fp_{\RF}^r}\psi(2azu)\ud z=q^{-r}=|u|^{-\frac{1}{2}}.
	\end{align*}
	If $|x|=q^{2r+1}$ for some $r\geq 0$, similarly we have
	\begin{align*}
		\int_{\Fo_{\RF}}\psi(u z^2)\ud z&=\sum_{a\in\Fo_{\RF}/\Fp_{\RF}^{r+1}}\int_{\Fp_{\RF}^{r+1}}\psi(u(a+z)^2)\ud z\\
		&=\sum_{a\in\Fo_{\RF}/\Fp_{\RF}^{r+1}}\int_{\Fp_{\RF}^{r+1}}\psi(a^2u +2azu+z^2u )\ud z\\
		&=\sum_{a\in\Fo_{\RF}/\Fp_{\RF}^{r+1}} \psi(a^2u)\int_{\Fp_{\RF}^{r+1}}\psi(2azu)\ud z\\
		&=\sum_{a\in\Fp_{\RF}^r/\Fp_{\RF}^{r+1} }\psi(a^2u)q^{-(r+1)}=q^{-\frac{1}{2}}|u|^{-\frac{1}{2}}\sum_{a\in\kappa_{\RF}}\psi\left(  \frac{ac(u)\cdot a^2}{\varpi_{\RF}} \right).
	\end{align*}

		\end{proof}
		
\begin{lem}\label{lem_basic_b_<1}
	When $|\zeta|<1$, we have
	\begin{align*}
		&\int_{\RN}^*\omega_{\psi}(\Bw)\omega_{\psi}\left(  \begin{pmatrix}
			\zeta & 0 \\ 0 & \zeta^{-1}
		\end{pmatrix}\right)\omega_{\psi}( u )\Phi_0(v_0)\psi^{-1}(u )\ud u =\gamma(\zeta,\psi)|\zeta|^{-\frac{1}{2}}\left(  \psi\left( \frac{2}{\zeta}\right)+\psi\left(\frac{-2}{\zeta}\right) \right).
	\end{align*}
\end{lem}
\begin{proof}
	Let us start with the integral
	\[
	|\zeta|^{-\frac{2n+1}{2}}\int_{|x|=q^i}\psi^{-1}\left( \frac{1}{u\zeta^2}+u   \right)|u |^{-n}\int_{\Fo_{\RF}}\psi(u z^2)\ud z \ud u
\]
	for some $i$ such that $|\zeta|^{-1}\leq |u|\leq q^k$. In the case that $i$ is even, this integral becomes 
		\[
		|\zeta|^{-\frac{2n+1}{2}}\int_{|u|=q^i}\psi^{-1}\left( \frac{1}{u\zeta^2}+u   \right)|u|^{-n-\frac{1}{2}}\ud u,
\]
	which is non-zero only when $|u|=|\zeta|^{-1}$. In this case, we may make a change of variable $u$ to $\zeta^{-1}u$, then it becomes
	\begin{align*}
		|\zeta|^{-1}\int_{\Fo_{\RF}^{\times}}\psi^{-1}\left(   \frac{u+u^{-1}}{\zeta}    \right)\ud u.
	\end{align*}
	Write $|\zeta|=q^{-2r}$ with $r>0$, then the above is 
	\begin{align*}
	|\zeta|^{-1}\sum_{a\in\Fo_{\RF}^{\times}/(1+\Fp_{\RF}^r)}\int_{\Fp_{\RF}^r}\psi^{-1}\left(    \frac{ a(1+z)+a^{-1}(1-z)       }{\zeta}       \right)\ud z&=|\zeta|^{-1}\cdot |\zeta|^{\frac{1}{2}}\left(  \psi\left( \frac{2}{\zeta}\right)+\psi\left(\frac{-2}{\zeta}\right) \right)\\
	&=|\zeta|^{-\frac{1}{2}}\left(  \psi\left( \frac{2}{\zeta}\right)+\psi\left(\frac{-2}{\zeta}\right) \right).
	\end{align*}
	
	When $|u|=q^i$ such that $i > 1$ is odd,
	\begin{align*}
	&|\zeta|^{-\frac{2n+1}{2}}\int_{|u|=q^i}\psi^{-1}\left( \frac{1}{u\zeta^2}+u   \right)|u |^{-n}\int_{\Fo_{\RF}}\psi(u z^2)\ud z \ud u\\
	&	=|\zeta|^{-\frac{2n+1}{2}}\int_{|u|=q^{i}}|u|^{-n-\frac{1}{2}}q^{-\frac{1}{2}}\psi^{-1}\left(  \frac{1}{u\zeta^2}+u  \right)\sum_{a\in\kappa_{\RF}}\psi\left(  \frac{\ac(x)\cdot a^2}{\varpi_{\RF}} \right)\ud u.
	\end{align*}
	Replace $u$ by $\zeta^{-1}u$, 
	\begin{align*}
	&|\zeta|^{-\frac{2n+1}{2}}\int_{|u|=q^{i}}|u|^{-n-\frac{1}{2}}q^{-\frac{1}{2}}\psi^{-1}\left(  \frac{1}{u\zeta^2}+u  \right)\sum_{a\in\kappa_{\RF}}\psi\left(  \frac{\ac(x)\cdot a^2}{\varpi_{\RF}} \right)\ud u\\
		&=q^{-\frac{1}{2}}|\zeta|^{-1}\int_{|u|=|\zeta|q^i\geq 1}\psi^{-1}\left(  \frac{u+u^{-1}}{\zeta}  \right)\sum_{a\in\kappa_{\RF}}\psi\left(  \frac{ac(u) a^2}{\ac(\zeta)\varpi_{\RF}} \right)\ud u.
	\end{align*}
	Write $|\zeta|q^i=q^{r}$ for some $r\geq 0$ and $u=\varpi_{\RF}^{-r}v$ with $v\in\Fo_{\RF}^{\times}$, then
	\begin{align*}
	&q^{-\frac{1}{2}}|\zeta|^{-1}\int_{|u|=|\zeta|q^i\geq 1}\psi^{-1}\left(  \frac{u+u^{-1}}{\zeta}  \right)\sum_{a\in\kappa_{\RF}}\psi\left(  \frac{ac(u) a^2}{\ac(\zeta)\varpi_{\RF}} \right)\ud u\\
		&=q^{i-\frac{1}{2}}\int_{v\in\Fo_{\RF^{\times}}}\psi^{-1}\left(  \frac{\varpi_{\RF}^{-r}v+\varpi_{\RF}^{r}v^{-1}}{\zeta}  \right)\sum_{a\in\kappa_{\RF}}\psi\left(  \frac{v a^2}{\ac(\zeta)\varpi_{\RF}} \right)\ud v\\
		&=q^{i-\frac{1}{2}}\sum_{a\in\kappa_{\RF}}\int_{v\in\Fo_{\RF}^{\times}}\psi^{-1}\left(  \frac{\varpi_{\RF}^{-r}v+\varpi_{\RF}^{r}v^{-1}}{\zeta} -\frac{ a^2v }{\ac(\zeta)\varpi_{\RF}} \right)\ud v\\
		&=q^{i-\frac{1}{2}}\sum_{a\in\kappa_{\RF}}\int_{v\in\Fo_{\RF}^{\times}}\psi^{-1}\left(  \frac{1}{ac(\zeta)}\left(    \left( \frac{1}{\varpi_{\RF}^i}-\frac{a^2}{\varpi_{\RF}} \right)v+\frac{\varpi_{\RF}^{2r}}{\varpi_{\RF}^i}v^{-1}       \right)      \right)\ud v\\
		&=q^{i-\frac{1}{2}}\sum_{a\in\kappa_{\RF}}\int_{v\in\Fo_{\RF}^{\times}}\psi^{-1}\left(  \frac{1}{ac(\zeta)}\left(  \frac{1-\varpi_{\RF}^{i-1}a^2}{\varpi_{\RF}^i}v+\frac{\varpi_{\RF}^{2r}}{\varpi_{\RF}^i}v^{-1}       \right)      \right)\ud v
	\end{align*}
	Write $i=2m+1$ with $m\geq 1$, $v=b(1+w)$ with $b\in\Fo_{\RF}^{\times}/(1+\Fp_{\RF}^{m+1})$ and $w\in\Fp_{\RF}^{m+1}$. For the integral being non-zero, we need
	\[
	(1-\varpi_{\RF}^{2m}a^2)b-\varpi_{\RF}^{2r}b^{-1}\in\Fp_{\RF}^m\Leftrightarrow b-\varpi_{\RF}^{2r}b^{-1}\in\Fp_{\RF}^m.
	\]
	The only possibility is $r=0$, and we must have $b\equiv b^{-1}\mod \Fp_{\RF}^m$, in which case
		\begin{align*}
		&q^{i-\frac{1}{2}}\sum_{a\in\kappa_{\RF}}\int_{v\in\Fo_{\RF}^{\times}}\psi^{-1}\left(  \frac{1}{ac(\zeta)}\left(  \frac{1-\varpi_{\RF}^{i-1}a^2}{\varpi_{\RF}^i}v+\frac{\varpi_{\RF}^{2r}}{\varpi_{\RF}^i}v^{-1}       \right)      \right)\ud v\\
		&=q^{i-\frac{1}{2}} q^{-m-1} \sum_{a\in\kappa_{\RF}} \left\{\sum_{1+\Fp_{\RF}^m/1+\Fp_{\RF}^{m+1}} +\sum_{-1+\Fp_{\RF}^m/1+\Fp_{\RF}^{m+1}} \right\}\psi^{-1}\left(   \frac{1}{\ac(\zeta)}\left( \frac{b+b^{-1}}{\varpi_{\RF}^{2m+1}}-\frac{a^2b}{\varpi_{\RF}}   \right)      \right).
	\end{align*}
	As $m\geq 1$, the summation over $a$ is independent of $b$, hence it is 
	\begin{align*}
	\begin{split}
		q^{-1}|\zeta|^{-\frac{1}{2}}&\sum_{a\in\kappa_{\RF}}\psi\left(\frac{a^2}{\ac(\zeta)\varpi_{\RF}}\right)  \sum_{1+\Fp_{\RF}^m/1+\Fp_{\RF}^{m+1}} \psi^{-1} \left( \frac{b+b^{-1}}{\zeta} \right)  + \\
		&q^{-1}|\zeta|^{-\frac{1}{2}}\sum_{a\in\kappa_{\RF}}\psi\left(-\frac{a^2}{\ac(\zeta)\varpi_{\RF}}\right)  \sum_{-1+\Fp_{\RF}^m/1+\Fp_{\RF}^{m+1}} \psi^{-1} \left( \frac{b+b^{-1}}{\zeta} \right).
		\end{split}
	\end{align*}
Write $b=1+u$ with $u\in\Fp_{\RF}^m/\Fp_{\RF}^{m+1}$, then we have
\begin{align*}
	\sum_{b\in1+\Fp_{\RF}^m/1+\Fp_{\RF}^{m+1}}\psi^{-1}\left( \frac{b+b^{-1}}{\zeta}  \right)&=\psi^{-1}\left(\frac{2}{\zeta} \right)\sum_{u \in\Fp_{\RF}^m /\Fp_{\RF}^{m+1}}\psi^{-1}\left( \frac{u^2}{\zeta} \right)\\
	&=\psi^{-1}\left(\frac{2}{\zeta} \right)\sum_{u \in\kappa_{\RF}}\psi^{-1}\left( \frac{u^2}{\ac(\zeta)\varpi_{\RF}} \right).\\
\end{align*}
It is clear that 
\begin{align*}\sum_{u \in\kappa_{\RF}}\psi^{-1}\left( \frac{u^2}{\ac(\zeta)\varpi_{\RF}} \right)
	\sum_{a\in\kappa_{\RF}}\psi\left(\frac{a^2}{\ac(\zeta)\varpi_{\RF}}\right) =q,
\end{align*}
hence, the original integral is still
\[
|\zeta|^{-\frac{1}{2}}\left(  \psi\left( \frac{2}{\zeta}\right)+\psi\left(\frac{-2}{\zeta}\right) \right).
\]

The last case is when $|\zeta|=q^{-1}$. When $i\geq 3$, we have $r=i-1$. At this time, we have
\[
\left|\frac{\varpi_{\RF}^{2r}}{\varpi_{\RF}^i}v^{-1}\right|\leq 1,
\]
and
\[
 \left|\frac{1-\varpi_{\RF}^{i-1}a^2}{\varpi_{\RF}^i}\right|>q,
\]
hence
\[
q^{i-\frac{1}{2}}\sum_{a\in\kappa_{\RF}}\int_{v\in\Fo_{\RF}^{\times}}\psi^{-1}\left(  \frac{1}{ac(\zeta)}\left(  \frac{1-\varpi_{\RF}^{i-1}a^2}{\varpi_{\RF}^i}v+\frac{\varpi_{\RF}^{2r}}{\varpi_{\RF}^i}v^{-1}       \right)      \right)\ud v=0.
\]
So the only case left is when $i=1$. Make a change of variable that $x=\zeta^{-1}u$, we have
\begin{align*}
	&|\zeta|^{-\frac{2n+1}{2}}\int_{|x|=|\zeta|^{-1}}\psi^{-1}\left( \frac{1}{x\zeta^2}+x   \right)|x|^{-n}\int_{\Fo_{\RF}}\psi(xz^2)\ud z\\
	&=q^{-\frac{1}{2}}|\zeta|^{-1}\int_{x\in\Fo_{\RF^{\times}}}\psi^{-1}\left( \frac{u+u^{-1}}{\zeta}   \right)\sum_{a\in\kappa_{\RF}}\psi\left(  \frac{ua^2}{\ac(\zeta)\varpi_{\RF}}  \right)\ud u.
	\end{align*}
Write $u=b(1+v)$ with $b\in\Fo_{\RF}^{\times}/1+\Fp_{\RF}$ and $v\in\Fp_{\RF}$, then we have
\begin{align*}
	\psi^{-1}\left( \frac{u+u^{-1}}{\zeta}   \right)\psi\left(  \frac{ua^2}{\ac(\zeta)\varpi_{\RF}}  \right)&=\psi^{-1}\left(  \frac{b(1+v)+b^{-1}(1-v)-b(1+v)a^2}{\zeta}         \right)\\
	&=\psi^{-1}\left(   \frac{b+b^{-1}-ba^2}{\zeta}   \right)\cdot \psi^{-1}\left(\frac{b-b^{-1}-ba^2}{\zeta} \cdot v   \right)\\
	&=\psi^{-1}\left(   \frac{b+b^{-1}-ba^2}{\zeta}   \right).
\end{align*}
Therefore,
\begin{align*}
&q^{-\frac{1}{2}}|\zeta|^{-1}\int_{x\in\Fo_{\RF^{\times}}}\psi^{-1}\left( \frac{u+u^{-1}}{\zeta}   \right)\sum_{a\in\kappa_{\RF}}\psi\left(  \frac{ua^2}{\ac(\zeta)\varpi_{\RF}}  \right)\ud u\\
	&=q^{-\frac{1}{2}}|\zeta|^{-1}q^{-1} \sum_{b\in\Fo_{\RF}^{\times}/1+\Fp_{\RF}}\sum_{a\in\kappa_{\RF}}\psi^{-1}\left(   \frac{b+b^{-1}-ba^2}{\zeta}   \right)
\end{align*}
Then we claim first for $ \alpha=\pm2$,
\[
\#\{ (a\in\kappa_{\RF},b\in\kappa_{\RF}^{\times}) \mid b+b^{-1}-ba^2=\alpha \}=2q-3.
\]
Indeed, for $\alpha=2$, this is equivalent to count the pairs $(a,b)$ such that $\displaystyle{a^2=(1-b^{-1})^2}$, so there are totally $2(q-1-1)+1=2q-3$ pairs. Similarly for$\alpha\neq 2$.

For $\alpha\neq \pm 2$, we claim
\[
\#\{ (a\in\kappa_{\RF},b\in\kappa_{\RF}^{\times}) \mid b+b^{-1}-ba^2=\alpha \}=q-3.
\]
In fact, when $\alpha^2-4$ is a square, it is known that there are $\displaystyle{\frac{q+1}{2}}$ elements in $\kappa_{\RF}$ such that $z^2-\alpha z+1$ is a square, hence ther number of pairs $(a,b)$ is 
\[
2\left(\frac{q+1}{2}-1-2\right)+2=q-3.
\]
When $\alpha^2-4$ is not a square, it is also known that there are $\displaystyle{\frac{q-1}{2}}$ elements in $\kappa_{\RF}$ such that $z^2-\alpha z+1$ is a square, hence, the total number is
\[
2\left(   \frac{q-1}{2}-1 \right)=q-3.
\]
Therefore, 
\begin{align*}
q^{-\frac{1}{2}}|\zeta|^{-1}q^{-1} \sum_{b\in\Fo_{\RF}^{\times}/1+\Fp_{\RF}}\sum_{a\in\kappa_{\RF}}\psi^{-1}\left(   \frac{b+b^{-1}-ba^2}{\zeta}   \right)&\\
&=|\zeta|^{-\frac{1}{2}}\left(  \psi\left( \frac{2}{\zeta}\right)+\psi\left(\frac{-2}{\zeta}\right) \right).
\end{align*}
\end{proof}

\begin{lem}\label{lem_basic_b_>1}
	When $|\zeta|\geq 1$, we have 
	\[\int_{\RN}^*\omega_{\psi}(\Bw)\omega_{\psi}\left(  \begin{pmatrix}
			\zeta & 0 \\ 0 & \zeta^{-1}
		\end{pmatrix}\right)\omega_{\psi}( u )\Phi_0(v_0)\psi^{-1}(u )\ud u =\gamma(\zeta,\psi)|\zeta|^{-n-\frac{1}{2}}(1+q^{-n}).
\]
	
\end{lem}

\begin{proof}
	When $|\zeta|\geq 1$, according to the above Lemma \ref{lem_b><1}, there is some $k$ large enough such that we have
	\begin{align*}
		&\int_{\RN}^*\omega_{\psi}(\Bw)\omega_{\psi}\left(  \begin{pmatrix}
			\zeta & 0 \\ 0 & \zeta^{-1}
		\end{pmatrix}\right)\omega_{\psi}( u )\Phi_0(v_0)\psi^{-1}(u )\ud u \\
		& =\gamma(\zeta,\psi)|\zeta|^{-\frac{2n+1}{2}}\left(    \int_{|u|\leq 1}\psi^{-1}(u)\int_{z\in\Fo_{\RF}}\psi(uz^2)\ud z  \ud u  + \int_{1<|u|\leq q^k}\psi^{-1}(u)|u|^{-n} \int_{\Fo_{\RF}}\psi(uz^2)\ud z   \ud u   \right) .\\
	\end{align*}
	It is clear that
	\begin{align*}
		  \int_{|u|\leq 1}\psi^{-1}(u)\int_{z\in\Fo_{\RF}}\psi(uz^2)\ud z  \ud u =1.
	\end{align*}
	Now, let us consider the integral
	\begin{align*}
		 \int_{|u|=q^i}\psi^{-1}(u)|u|^{-n} \int_{\Fo_{\RF}}\psi(uz^2)\ud z   \ud u ,
	\end{align*}
	where $0<i\leq k$. If $i$ is even, we know
	\begin{align*}
		 \int_{|u|=q^i}\psi^{-1}(u)|u|^{-n} \int_{\Fo_{\RF}}\psi(uz^2)\ud z   \ud u=0
	\end{align*}
	due to Lemma \ref{lem_gsum}. When $i$ is odd, according to Lemma \ref{lem_gsum} again, we have
	\begin{align*}
		&\int_{|u|=q^i}\psi^{-1}(u)|u|^{-n} \int_{\Fo_{\RF}}\psi(uz^2)\ud z   \ud u\\
		&=\int_{|u|=q^i}\psi^{-1}(u)|u|^{-n} q^{-\frac{1}{2}}|u|^{-\frac{1}{2}}\sum_{a\in\kappa_{\RF}}\psi\left(  \frac{\ac(u)\cdot a^2}{\varpi_{\RF}} \right)\ud u\\
		&=q^{-i\left(n+\frac{1}{2}\right)-\frac{1}{2}}\sum_{a\in\kappa_{\RF}}\int_{|u|=q^i}\psi^{-1}\left(  u\left(  1-a^2\varpi_{\RF}^{i-1} \right)      \right)\ud u.
	\end{align*}
	If $i\geq 3$, $|1-a^2\varpi_{\RF}^{i-1}|=1$, then
	\begin{align*}
		\int_{|u|=q^i}\psi^{-1}\left(  u\left(  1-a^2\varpi_{\RF}^{i-1} \right)      \right)\ud u=0
	\end{align*}
	for all $a\in\kappa_{\RF}$. If $i=1$, then
	\begin{align*}
		&\int_{|u|=q^i}\psi^{-1}(u)|u|^{-n} \int_{\Fo_{\RF}}\psi(uz^2)\ud z   \ud u\\
		&=q^{-n-1}\sum_{a\in\kappa_{\RF}}\int_{|u|=q} \psi^{-1}(u(1-a^2))\ud u \\
		&=q^{-n-1}\left\{\sum_{a=\pm 1}+\sum_{a\in\kappa_{\RF}\setminus\{\pm 1\}}  \right\} \int_{|u|=q}\psi^{-1}(u(1-a^2))\ud u\\
		&=q^{-n-1}(  2(q-1)+ (q-2)(-1)    )\\
		&=q^{-n}.
	\end{align*}
	Therefore,
	\begin{align*}
		& |\zeta|^{-\frac{2n+1}{2}}\left(    \int_{|u|\leq 1}\psi^{-1}(u)\int_{z\in\Fo_{\RF}}\psi(uz^2)\ud z  \ud u  + \int_{1<|u|\leq q^k}\psi^{-1}(u)|u|^{-n} \int_{\Fo_{\RF}}\psi(uz^2)\ud z   \ud u   \right) \\
		 &= |\zeta|^{-n-\frac{1}{2}}(1+q^{-n}).
	\end{align*}
\end{proof}

Combine Lemma \ref{lem_basic_b_<1} and Lemma \ref{lem_basic_b_>1}, we have
\begin{prp}\label{prp_b_weil}
	\begin{align*}
	\widetilde{\BT}(\CI(\res(\Phi_0)))\left(   \begin{pmatrix}
		\zeta & 0 \\ 0 & \zeta^{-1}
	\end{pmatrix} \right)=\gamma(\zeta,\psi)\left\{ \begin{array}{rcl}
 |\zeta|^{-\frac{1}{2}} \left( \psi^{-1}\left(\frac{2}{\zeta}\right)+\psi\left(\frac{2}{\zeta} \right)  \right)   & \mbox{for} & |\zeta|\leq  q^{-1}; \\
|\zeta|^{-n-\frac{1}{2}}(1+q^{-n}) & \mbox{for} & |\zeta|\geq 1.
\end{array}\right.
	\end{align*}

\end{prp}

Next let us compute $\widetilde{\BT}\circ\CT\left(f_{L\left(\std,n-\frac{1}{2}\right)L\left(\std,\frac{1}{2}\right)}\right)$. It is clear that
\begin{align*}
	 &\iota\circ|\cdot|\left( f_{L\left(\std,n-\frac{1}{2} \right)L\left(\std,  \frac{1}{2} \right)}\left( \begin{pmatrix}
			\cdot & 0 \\ 0 & 1
		\end{pmatrix}  \right)\right)(\xi)=(1-q^{-2})^{-1}(1-q^{n-1})^{-1}\left(1-q^{-n}\right)^{-1}\\
		&\cdot \left\{ \begin{array}{rcl}
|\xi|^{n-1} (1-q^{-2n)})-q^{n-1}(1-q^{-2})   & \mbox{for} & |\xi|\leq  1; \\
-q^{-n-1}+q^{n-3}& \mbox{for} & |\xi|=q; \\
(1-q^{n-1})  \int_{|x|^2=|\xi|^{-1}}  \psi^{-1}\left( x\xi + x^{-1} \right)\ud x & \mbox{for} & |\xi|>q .
\end{array}\right.
\end{align*}

To calculate its Fourier transform, observe that the Kloosterman integrals are related to the Fourier transform of $\displaystyle{ 1_{\Fo_{\RF}}(x)\psi^{-1}\left( \frac{1}{x} \right) }$:
\begin{lem}\label{lem_four_kloo}
	\begin{align*}
		\BF_{\psi}\left(  1_{\Fo_{\RF}}(x)\psi^{-1}\left( \frac{1}{x} \right) \right)(y)=\left\{ \begin{array}{rcl}
1-q^{-1}-q^{-2}  & \mbox{for} & |y|\leq  1; \\
-q^{-1}-q^{-2} & \mbox{for }& |y|=q;  \\
\int_{|x|^2=|y|^{-1}}\psi^{-1}\left( \frac{1}{y}+xy  \right) \ud x   & \mbox{for} & |y|>q.
\end{array}\right.
	\end{align*}
\end{lem}

\begin{proof}
	Since $\displaystyle{1_{\Fo_{\RF}}(x)\psi^{-1}\left( \frac{1}{x} \right) }\in\CL^1(\RF)\cap\CL^2(\RF)$, we have
	\begin{align*}
		\BF_{\psi}\left(  1_{\Fo_{\RF}}(x)\psi^{-1}\left( \frac{1}{x} \right) \right)(y)&=\int_{\RF} 1_{\Fo_{\RF}}(x)\psi^{-1}\left( \frac{1}{x} \right) \psi^{-1}(xy)\ud x \\
		&=\int_{\Fo_{\RF}}\psi^{-1}\left( \frac{1}{x}+xy   \right)\ud x.
	\end{align*}
	When $|y|\leq 1$, we have $\psi^{-1}(xy)=1$ for all $x\in\Fo_{\RF}$, hence
	\begin{align*}
		\int_{\Fo_{\RF}}\psi^{-1}\left( \frac{1}{x}+xy   \right)\ud x=\int_{\Fo_{\RF}}\psi^{-1}\left(  \frac{1}{x} \right)\ud x&=\int_{|x|\geq 1}|x|^{-2}\psi^{-1}(x)\ud x\\
		&=\int_{|x|=1}\ud x+q^{-2}\int_{|x|=q}\psi^{-1}(x)\ud x\\
		&=(1-q^{-1})+q^{-2}\cdot (-1).
	\end{align*}
	When $|y|=q$, then $\psi^{-1}(xy)=1$ for all $|x|\leq q^{-1}$, hence
	\begin{align*}
		\int_{\Fo_{\RF}}\psi^{-1}\left( \frac{1}{x}+xy   \right)\ud x&=\int_{|x|=1}\psi^{-1}(xy)\ud x+\int_{|x|\leq q^{-1}}\psi^{-1}\left(\frac{1}{x}\right)\ud x\\
		&=-q^{-1}+\int_{|x|\geq q}\psi^{-1}(x)|x|^{-2}\ud x=-q^{-1}-q^{-2}.
	\end{align*}
	When $|y|\geq q$, it follows from, for example, the last paragraph in \cite[Section 5]{Sak13}.
\end{proof}

\begin{lem}\label{lem_b_basic_laststep}
	\begin{align*}
		&\BF_{\psi^{-1}}\circ\iota\circ|\cdot|\left( f_{L\left(\std,s_1\right)L\left(\std,s_2\right)}\left( \begin{pmatrix}
			\cdot & 0 \\ 0 & 1
		\end{pmatrix}  \right)\right)(x)\\
		&=(1-q^{-2})^{-1}(1-q^{-n})^{-1} \left\{ \begin{array}{rcl}
\psi^{-1}\left(\frac{1}{x}\right)+1 & \mbox{for} & |x|\leq  q^{-1}; \\
|x|^{-n}(1+q^{-n}) & \mbox{for} & |x|\geq 1.
\end{array}\right.	\end{align*}
	
\end{lem}

\begin{proof}
	From Lemma \ref{lem_four_kloo}, 
	\begin{align*}
		&(1-q^{-2})(1-q^{n-1})(1-q^{-n})\cdot \iota\circ|\cdot|\left( f_{L\left(\std,s_1\right)L\left(\std,s_2\right)}\left( \begin{pmatrix}
			\cdot & 0 \\ 0 & 1
		\end{pmatrix}  \right)\right)(\xi)\\
		&-(1-q^{n-1})\BF_{\psi}\left(  1_{\Fo_{\RF}}(\cdot)\psi^{-1}\left( \frac{1}{\cdot} \right) \right)(\xi) \\
		&= \left\{ \begin{array}{rcl}
|\xi|^{n-1} (1-q^{-2n)})+q^{-2}+q^{-1}-1-q^{n-2} & \mbox{for} & |\xi|\leq  1; \\
  -q^{-n-1}+q^{-1}+q^{-2}-q^{n-2}  & \mbox{for} & |\xi|=q; \\
0 & \mbox{for} & |\xi|>q .
\end{array}\right.
	\end{align*}
	Now, let us first compute the Fourier transform of 
	\begin{align*}
		f_1(\xi):=\left(|\xi|^{n-1} (1-q^{-2n)})+q^{-2}+q^{-1}-1-q^{n-2}\right)1_{\Fo_{\RF}}(\xi).
	\end{align*}
	By definition, when $|x|\leq 1$,
	\begin{align*}
		&\BF_{\psi^{-1}}(f_1)(x)=\int_{\RF}f_1(\xi)\psi(\xi x)\ud \xi \\
		&= \int_{|\xi|\leq1 }\left(|\xi|^{n-1} (1-q^{-2n)})+q^{-2}+q^{-1}-1-q^{n-2}\right)\ud \xi\\
		&=\sum_{|\xi|=q^{-i},i=0}^{\infty}q^{-i(n-1)}q^{-i}(1-q^{-1})(1-q^{-2n})+q^{-2}+q^{-1}-1-q^{n-2} \\
		&=(1-q^{-1})(1+q^{-n})+q^{-2}+q^{-1}-1-q^{n-2}\\
		&=-q^{-n-1}+q^{-n}+q^{-2}-q^{n-2}.
	\end{align*}
	
	 When $|x|>1$,
	\begin{align*}
		&\BF_{\psi^{-1}}(f_1)(x)=\int_{\RF}f_1(\xi)\psi(\xi x)\ud \xi \\
		&=\left\{ \int_{|\xi|\leq |x|^{-1}}+\int_{|\xi|=|x|^{-1}q}\right\}\left(|\xi|^{n-1} (1-q^{-2n)})+q^{-2}+q^{-1}-1-q^{n-2}\right)\psi(\xi x)\ud \xi
	\end{align*}
	Write $|x|=q^k$ with $k>0$, then
	\begin{align*}
		&\int_{|\xi|\leq |x|^{-1}}\left(|\xi|^{n-1} (1-q^{-2n)})+q^{-2}+q^{-1}-1-q^{n-2}\right)\psi(\xi x)\ud \xi\\
		&=(1-q^{-2n})\sum_{|\xi|=q^{-i},i=k}^{\infty}q^{-i(n-1)}q^{-i}(1-q^{-1})+|x|^{-1}(q^{-2}+q^{-1}-1-q^{n-2})\\
		&=|x|^{-n}(1+q^{-n})(1-q^{-1})+|x|^{-1}(q^{-2}+q^{-1}-1-q^{n-2}).
	\end{align*}
	And
	\begin{align*}
		&\int_{|\xi|=|x|^{-1}q}\left(|\xi|^{n-1} (1-q^{-2n)})+q^{-2}+q^{-1}-1-q^{n-2}\right)\psi(\xi x)\ud \xi\\
		&=-|x|^{-1}\left(    (|x|^{-1}q)^{n-1}(1-q^{-2n})+q^{-2}+q^{-1}-1-q^{n-2}      \right),
	\end{align*}
	hence
	\begin{align*}
		\BF_{\psi^{-1}}(f_1)(x)&=|x|^{-n}(q^{-n}-q^{-1}+1-q^{n-1}) \\
		&=|x|^{-n}(1+q^{-n})(1-q^{n-1}).
	\end{align*}
	To summarize, 
	\begin{align*}
		\BF_{\psi^{-1}}(f_1)(x)=\left\{ \begin{array}{rcl}
-q^{-n-1}+q^{-n}+q^{-2}-q^{n-2} &  \mbox{for} & |x|\leq  1; \\
|x|^{-n}(1+q^{-n})(1-q^{n-1}) & \mbox{for} & |x |>1 .
\end{array}\right.
	\end{align*}
	
Next set
\begin{align*}
	f_2(\xi):= \left\{ \begin{array}{rcl}
-q^{-n-1}+q^{-1}+q^{-2}-q^{n-2} 
 & \mbox{for} &|\xi|=  q; \\
0& \mbox{for} & |\xi|\neq q .
\end{array}\right.
\end{align*}
When $|x|\leq q^{-1}$, we have
\begin{align*}
	\BF_{\psi^{-1}}(f_2)(x)&=\int_{|\xi|=q}(-q^{-n-1}+q^{-1}+q^{-2}-q^{n-2} )\psi(x\xi )\ud \xi \\
	&=(q-1)(-q^{-n-1}+q^{-1}+q^{-2}-q^{n-2}) \\
	&=q^{-n-1}-q^{-n}-q^{-2}+1+q^{n-2}-q^{n-1}.
\end{align*}
When $|x|\geq 1$, $\BF_{\psi^{-1}}(f_2)(x)\neq 0$ only when $|x|=1$, in which case
\begin{align*}
	\BF_{\psi^{-1}}(f_2)(x)=q^{-n-1}-q^{-1}-q^{-2}+q^{n-2}.
\end{align*}
	Therefore,
	\begin{align*}
		\BF_{\psi^{-1}}(f_2)(x)=\left\{ \begin{array}{rcl}
q^{-n-1}-q^{-n}-q^{-2}+1+q^{n-2}-q^{n-1} &\mbox{for} & |x|\leq  q^{-1}; \\
q^{-n-1}-q^{-2}-q^{-1}+q^{n-2} & \mbox{for} & |x |=1 ;\\
0 & \mbox{for} & |x |>1.
\end{array}\right.
	\end{align*}
Combine them together, and we have
\begin{align*}
		\BF_{\psi^{-1}}(f_1+f_2)(x)=\left\{ \begin{array}{rcl}
1 -q^{n-1}& \mbox{for} & |x|\leq  q^{-1}; \\
q^{-n}(1-q^{n-1}) & \mbox{for} & |x |=1 ;\\
|x|^{-n}(1+q^{-n})(1-q^{n-1}) & \mbox{for} & |x|>1.
\end{array}\right.
	\end{align*}
	Therefore, 
	\begin{align*}
		&\BF_{\psi^{-1}}\circ\iota\circ|\cdot|\left( f_{L\left(\std,s_1\right)L\left(\std,s_2\right)}\left( \begin{pmatrix}
			\cdot & 0 \\ 0 & 1
		\end{pmatrix}  \right)\right)(x) \\
		& =(1-q^{-2})^{-1}(1-q^{-n})^{-1} \left\{ \begin{array}{rcl}
\psi^{-1}\left(\frac{1}{x}\right)+1 & \mbox{for} & |x|\leq  q^{-1}; \\
1+q^{-n}& \mbox{for} & |x|=1 ;\\
|x|^{-n}(1+q^{-n}) & \mbox{for} & |x|>1.
\end{array}\right. \\
	\end{align*}
\end{proof}

\begin{prp}\label{prp_b_tranbasic}
	\begin{align*}
	&\widetilde{\BT}\circ\CT \left( f_{L\left(\std,s_1\right)L\left(\std,s_2\right)}\left( \begin{pmatrix}
			\cdot & 0 \\ 0 & 1
		\end{pmatrix}  \right)\right)(x) \\
		&=(1-q^{-2})^{-1}(1-q^{-n})^{-1} \left\{ \begin{array}{rcl}
 |x|^{-\frac{1}{2}} \left( \psi^{-1}\left(\frac{2}{x}\right)+\psi\left(\frac{2}{x} \right)  \right)   & \mbox{for} & |x|\leq  q^{-1}; \\
|x|^{-n-\frac{1}{2}}(1+q^{-n}) & \mbox{for} & |x|\geq 1.
\end{array}\right.
	\end{align*}
\end{prp}

\begin{proof}
	This is according to Proposition \ref{prp_tr_weil_b} and Lemma \ref{lem_b_basic_laststep}.
\end{proof}

Finally, let $\BL_{\RX}^{\circ}$ be the basic measure on $\RX\times\RX\sslash\SO_{2n+1}\cong\BA^1$, then the image of $\Phi_0$ to $\CS(\FX)$ is 
\begin{align*}
	 (1+q^{-n})\BL_0=\frac{\zeta_{\RF}(n)}{\zeta_{\RF}(2n)}\BL_{\RX}^{\circ}.
\end{align*}

Combine Proposition \ref{prp_b_weil} and Proposition \ref{prp_b_tranbasic}, we have
\begin{thm}\label{thm_b}
	Let $\BL_0\in\CS(\SO_{2n}\backslash\SO_{2n+1}/\SO_{2n})$ be the push-forward of $1_{\RX(\Fo_{\RF})}\ud x \otimes 1_{\RX(\Fo_{\RF})}\ud x$, then we have
	\[
	\CT^{-1}\left(\BL_{\RX}^{\circ}\right)=\frac{q^{\dim\RX}}{\#\overline{\RX}(\kappa_{\RF})\zeta_{\RF}(2)\zeta_{\RF}(s_{\RX})}f_{L_{\RX}}=\frac{\zeta_{\RF}(2n)}{\zeta_{\RF}(2)\zeta_{\RF}(n)^2}f_{L\left(\std,n-\frac{1}{2}\right)L\left(\std,\frac{1}{2}\right)}.
	\]
\end{thm}

\subsection{Comparison between $\RA\backslash\PGL_2/(\RN,\psi)$ and $(\RN,\psi)\backslash\PGL_2/(\RN,\psi)$}

Let $\RA$ be the maximal torus of $\PGL_2$ consisting of diagonal matrices. For any $s\in\BC$, write $\CL_s$ for the line bundle on $\RA\backslash\PGL_2$ induced by the unramified character
\[
\RA_1\rightarrow\BC^{\times}:\begin{pmatrix}
	a & \\ & 1
\end{pmatrix}\mapsto |a|^s.
\]
Write $\CC_c^{\infty}(\RA\backslash\PGL_2,\CL_s)$ be the Schwartz sections of $\CL_s$, which can be trivialized as functions $\Psi$ on $\PGL_2$ satisfying
\[
\Psi\left(   \begin{pmatrix}
	a & \\ & 1
\end{pmatrix} g\right)=|a|^s\Psi(g)
\]
with compact support module $\RA$. Take $\Psi_s^{\circ}$ to be the \emph{basic function} such that its support is $\RA_1\cdot\PGL_2(\Fo_{\RF})$ and satisfying
\[
	\Psi_s\left(  \begin{pmatrix}
	a & \\ & 1
	\end{pmatrix}k   \right)=|a|^s,\;\forall \begin{pmatrix}
		a & \\ & 1
	\end{pmatrix}\in\RA\;\mathrm{and}\;k\in\PGL_2(\Fo_{\RF}).
	\]

Consider the unfolding map
\begin{align*}
\RU_{\psi}:\CC_c^{\infty}(\RA_1\backslash\PGL_2,\CL_s)&\rightarrow\CC^{\infty}(\RN,\psi\backslash\PGL_2)\\
\Psi &\mapsto \left(  \PGL_2\ni g\mapsto  \int_{\RF} \Psi \left(  \begin{pmatrix}
		1 & n \\ & 1 
	\end{pmatrix} g \right)\psi^{-1}(n)\ud n    \right).
\end{align*}
Note that $\RA\backslash\PGL_2$ can be identified with the line bundle $\CO(2)$ over $\RB\backslash\PGL_2=\BP^1$, and $\RN\backslash\PGL_2$ is then identified with the open $\PGL_2$-orbit consisting of non-zero sections in the total space of the dual bundle $\CO(-2)$ in the usual way, or see \cite{Sak13}. For any $s\in\BC$, let $\BL_{s}$ be the following line bundle over $\CO(-2)$:
\[
\BL_s^{\psi}:=\FL_s\otimes\CL_{\psi}.
\]
Here the line bundle $\CL_{\psi}$ is the Whittaker bundle over $\RN\backslash\PGL_2$, and the sections of the other line bundle $\FL_s$, are the smooth functions over $ \RN\backslash\PGL_2$, and in any small neighborhood of the boundary $(p,0)\in\CO(-2)$, where $p\in\BP^1$ and $0$ means the zero cotangent vector at $p$, the sections are of the form $|a|^{s+1}\Phi(q,a)$ for some smooth function $\Phi$ near $(p,0)$. Then according to \cite[Section 9.5, (9.17)]{SV17}, the unfolding map $\RU$ gives an isomorphism between the Schwartz sections
\[
\CF(\RA_1\backslash\PGL_2,\CL_s)\cong\CF(\CO(-2),\BL^{\psi}_{s}).
\]

\begin{lem}
	Let $\Psi^{\circ}_s$ be the basic function of $\CC_c^{\infty}(\RA_1\backslash\PGL_2,\CL_s)$, then
	\[
	\RU_{\psi}(\Psi^{\circ}_s)=\Phi_{L\left(\std, s+\frac{1}{2} \right)}
	\]
	
	\end{lem}

\begin{proof}
	See \cite[Lemma 5.5]{Sak13}.
\end{proof}

Note that $\CF(\RA\backslash\PGL_2)=\CF(\RA\backslash\PGL_2,\CL_0)$, which consists of usual Schwartz functions on $\RA\backslash\PGL_2$. While analogous results hold for the more general $\CL_s$, we restrict our attention to the case $\CL_0$ for the purposes of the present work.

\begin{lem}
	Let $\Psi_1\in\CF(\RA\backslash\PGL_2)$ and $\Psi_2\in\CF(\RN\backslash\PGL_2,\BL_s^{\psi^{-1}})$, when $\xi\in\RF$, the orbital integral
	\[
\CO_{\xi}(\Psi_1\otimes\Psi_2):=\int_{\PGL_2}\Psi_1\left( \Bw  \begin{pmatrix}
	1 & \xi \\ & 1
\end{pmatrix}  g\right)\Psi_2(g)\ud g
\]
is absolutely convergent when $\Re(s)>-1$. Write $\CS^-_{L\left(\st,s+\frac{1}{2} \right)}(\RA\backslash\PGL_2/\RN,\psi)$ to be the space of measures on $\RF$ of the form $\CO_{\xi}(\Psi_1\otimes\Psi_2)\ud \xi$ for $\Psi_1\in\CC_c^{\infty}(\RA\backslash\PGL_2)$ and $\Psi_2\in\CC_c^{\infty}(\RN\backslash\PGL_2,\BL_s^{\psi^{-1}})$.
\end{lem}

\begin{rmk}
The subscript $L$-value signifies the use of non-standard test measures derived from the Whittaker space. In the subsequent unfolding of $\CC_c^{\infty}(\RA\backslash\PGL_2)$, a second $L\left(\st,\frac{1}{2}\right)$ is contributed. Specifically, the reader should view the space 	$\CS^-_{L\left(\st,s+\frac{1}{2} \right)}(\RA\backslash\PGL_2/\RN,\psi)$ in relation to $\CS^-_{ L\left(\st,\frac{1}{2} \right)L\left(\st,s+\frac{1}{2} \right)}(\RN,\psi\backslash\PGL_2/\RN,\psi)$, a comparison which constitues the primary objective of this section.
\end{rmk}

\begin{proof}
	Using the Iwasawa decomposition, we have
	\begin{align*}
	&	\int_{\PGL_2}\Psi_2\left( \Bw  \begin{pmatrix}
	1 & \xi \\ & 1
\end{pmatrix}  g\right)\Psi_2(g)\ud g\\
&=\int_{\RF\times\RF^{\times}\times\PGL_2(\Fo_{\RF})}\Psi_1\left( \Bw \begin{pmatrix}
	1 & \xi+ n \\ & 1
\end{pmatrix}  \begin{pmatrix}
	a & \\ & 1
\end{pmatrix} k  \right)\psi^{-1}(n)\Psi_2\left(\begin{pmatrix}
	a & \\ & 1
\end{pmatrix} k  \right)|a|^{-1}  \ud n\frac{\ud a }{|a|}\ud k.
	\end{align*}
Note that
\[
\Psi_1\left( \Bw \begin{pmatrix}
	1 & \xi+ n \\ & 1
\end{pmatrix}  \begin{pmatrix}
	a & \\ & 1
\end{pmatrix} k  \right)=\Psi_1\left( \Bw \begin{pmatrix}
	1 & \frac{\xi+n}{a}\\ & 1
\end{pmatrix} k \right),
\]
and since $\Psi_1$ is a Schwartz function, for fixed $\xi$, there is some $M>0$ such that $\Psi_1\neq 0$ only when $|\xi+n|<|a|M$. As $\Psi_2$ is also a Schwartz section, we have
\begin{align*}
	\int_{\PGL_2}\Psi_2\left( \Bw  \begin{pmatrix}
	1 & \xi \\ & 1
\end{pmatrix}  g\right)\Psi_2(g)\ud g \ll \int_{|a|\leq 1}|a||a|^{s+1}|a|^{-1}\frac{\ud a}{|a|}<\infty
\end{align*}
when $\Re(s)>-1$.

\end{proof}

On the other hand, let $\Phi_1\in\CF(\CO(-2),\BL^{\psi}_{0})$     and $\Phi_2\in\CF(\CO(-2),\BL_{s}^{\psi^{-1}} )$, for $\xi\in\RF^{\times}$, 
\begin{lem}
	When $\Re(s)>0$, the orbital integral
	\[
\CO_{\xi}(\Phi_1\otimes\Phi_2):=\int_{\PGL_2}\Phi_1\left( \Bw \begin{pmatrix}
	\xi & \\ & 1
\end{pmatrix}  g \right)\Phi_2(g)\ud g,
\] 
is absolutely convergent. Moreover, the twisted-push forward measure
\[
\CO_{\xi}(\Phi_1\otimes\Phi_2)  \ud \xi 
\]
lies in $\CS^-_{L\left(\st,s+\frac{1}{2}  \right)L\left(\st, \frac{1}{2}\right)}(\RN,\psi\backslash \PGL_2/\RN,\psi)$.
\end{lem}

\begin{proof}
	According to the Iwasawa decomposition, we have
	\begin{align*}
		&\int_{\PGL_2}\Phi_1\left( \begin{pmatrix}
	 & -1 \\ 1 &
\end{pmatrix} \begin{pmatrix}
	\xi & \\ & 1
\end{pmatrix}  g \right)\Phi_2(g)\ud g\\
&=\int_{\RF\times\RF^{\times}\times\PGL_2(\Fo_{\RF})}\Phi_1\left( \begin{pmatrix}
	 & -1 \\ 1 &
\end{pmatrix} \begin{pmatrix}
	\xi & \\ & 1
\end{pmatrix}  \begin{pmatrix}
	1 & n \\ & 1
\end{pmatrix} \begin{pmatrix}
	a & \\ & 1
\end{pmatrix} k \right)\\
&\cdot \Phi_2\left(\begin{pmatrix}
	a & \\ & 1
\end{pmatrix}k \right)\psi^{-1}(n)|a|^{-1}\ud n\frac{\ud a}{|a|}\ud k.
	\end{align*}
	As $\Phi_2$ is a Schwartz section, there is some $M>0$ such that $|a|<M$ and near $0$ we have $\Phi_2\left( \begin{pmatrix}
		a & \\ & 1
	\end{pmatrix}  k\right)\ll |a|^{s+1}$. Note that
	\begin{align*}
		\begin{pmatrix}
	 & -1 \\ 1 &
\end{pmatrix} \begin{pmatrix}
	\xi & \\ & 1
\end{pmatrix}  \begin{pmatrix}
	1 & n \\ & 1
\end{pmatrix} \begin{pmatrix}
	a & \\ & 1
\end{pmatrix} =\begin{pmatrix}
	1 & \\ & a \xi
\end{pmatrix}\begin{pmatrix}
	1 & \\ -a^{-1}n &1
\end{pmatrix}\begin{pmatrix}
	& -1 \\ 1 & 
\end{pmatrix},
	\end{align*}
	when $|n|\leq |a|$, we have 
	\[
	\Phi_1\left( \begin{pmatrix}
	 & -1 \\ 1 &
\end{pmatrix} \begin{pmatrix}
	\xi & \\ & 1
\end{pmatrix}  \begin{pmatrix}
	1 & n \\ & 1
\end{pmatrix} \begin{pmatrix}
	a & \\ & 1
\end{pmatrix} k \right) =\Phi_1\left( \begin{pmatrix}
	a^{-1}\xi^{-1} & \\ & 1
\end{pmatrix} k^{\prime} \right)
	\]
	for some $k^{\prime}\in\PGL_2(\Fo_{\RF})$, then since $\Phi_1$ is a Schwartz section, there is some $N>0$ such that $|a^{-1}\xi^{-1}|<N$ when $\displaystyle{\Phi_1\left( \begin{pmatrix}
	a^{-1}\xi^{-1} & \\ & 1
\end{pmatrix} k^{\prime} \right)\neq 0} $, hence the integral is absolutely convergent in this part. Moreover, it is clear that when $\displaystyle{\xi\leq \frac{1}{MN}}$, 
	\begin{align*}
		&\int_{\RF^{\times}\times\RF\times\PGL_2(\Fo_{\RF}),|n|\leq |a|}\Phi_1\left( \begin{pmatrix}
	 & -1 \\ 1 &
\end{pmatrix} \begin{pmatrix}
	\xi & \\ & 1
\end{pmatrix}  \begin{pmatrix}
	1 & n \\ & 1
\end{pmatrix} \begin{pmatrix}
	a & \\ & 1
\end{pmatrix} k \right)\\
&\cdot \Phi_2\left(\begin{pmatrix}
	a & \\ & 1
\end{pmatrix}k \right)\psi^{-1}(n)|a|^{-1}\ud n\frac{\ud a}{|a|}\ud k=0.
	\end{align*}
When $\xi\rightarrow \infty$, we may first assume $\Phi_1$ and $\Phi_2$ are $\PGL_2(\Fo_{\RF})$-invariant without loss of generality. Over the domain $\{(n,a)\in\RF\times\RF^{\times}\mid |n|\leq |a| \leq 1\}$, write 
\[
f_{\xi}(n,a)=\Phi_1\left(  \begin{pmatrix}
	a^{-1}\xi^{-1} & \\ & 1
\end{pmatrix}   \right)\Phi_2\left(  \begin{pmatrix}
	a & \\ & 1
\end{pmatrix} \right)|a|^{-1}
\]
and
\[
g_{\xi}(n,a)=|a|^{-1}\Phi_2\left(   \begin{pmatrix}
	a & \\ & 1
\end{pmatrix} \right)\left\{ \begin{array}{rcl}
 \frac{C}{|a\xi |}   & \mbox{for}
& |a|\geq |\xi|^{-1}; \\ 0 & \mbox{for} & |a|<|\xi|^{-1}, \\
\end{array}\right.
\]
where \[C=\lim_{a\rightarrow 0}  |a|^{-1} \Phi_1\left( \begin{pmatrix}
	a & \\ & 1
\end{pmatrix}\right).\]
Then we have the pointwise convergence
\[
\lim_{|\xi|\rightarrow\infty}\left( f_{\xi}(n,a)-g_{\xi}(n,a)\right)=0
\]
and $|f_{\xi}(n,a)|\ll \max\Phi_1   \cdot    \Phi_2\left(  \begin{pmatrix}
	a & \\ & 1
\end{pmatrix} \right)|a|^{-1}$, $|g_{\xi}(n,a)|\ll C \cdot \Phi_2\left(  \begin{pmatrix}
	a & \\ & 1
\end{pmatrix} \right)|a|^{-1}$, both of which are integrable when $\Re(s)>0$, hence by dominant convergence theorem, 
\begin{align*}
	\lim_{\xi\rightarrow\infty} \left(\int_{\RF^{\times}\times\RF,|a|\leq |a|\leq 1}f_{\xi}(n,a)-g_{\xi}(n,a) \right)\ud n\frac{\ud a}{|a|} =0.\end{align*}
As it is clear that there are some constants $C_1$ and $C_2$ such that 
\[
\int_{\RF^{\times}\times\RF,|a|\leq |a|\leq 1}g_{\xi}(n,a)\ud n\frac{\ud a}{|a|}=C_1|\xi|^{-1}+C_2|\xi|^{-s-1},
\]
and the integral of $f_{\xi}(n,a)$ is a linear combination of complex powers of $|\xi|$, this yields the required asymptotic expansion near $\infty$.

	When $|n|>|a|$, note that
	\begin{align*}
\begin{pmatrix}
	a^{-1} \xi^{-1} & \\ & 1
\end{pmatrix}		 \begin{pmatrix}
	1 & \\ -a^{-1}n &1
\end{pmatrix}=\begin{pmatrix}
	1 & -\xi^{-1}n^{-1} \\ & 1
\end{pmatrix}\begin{pmatrix}
	\xi^{-1}n^{-1} & \\ & -a^{-1}n
\end{pmatrix}\begin{pmatrix}
	& 1 \\ 1 & -an^{-1}
\end{pmatrix},
	\end{align*}
	then
	\begin{align*}
		&\Phi_1\left( \begin{pmatrix}
	 & -1 \\ 1 &
\end{pmatrix} \begin{pmatrix}
	\xi & \\ & 1
\end{pmatrix}  \begin{pmatrix}
	1 & n \\ & 1
\end{pmatrix} \begin{pmatrix}
	a & \\ & 1
\end{pmatrix} k \right)\\
&=\psi\left(- \frac{1}{n\xi}  \right)\Phi_1\left(   \begin{pmatrix}
	-\frac{a}{\xi n^2 }& \\ & 1
\end{pmatrix}\begin{pmatrix}
	& 1 \\ 1 & -an^{-1}
\end{pmatrix}   \begin{pmatrix}
	& -1 \\ 1 & 
\end{pmatrix}  k \right).
	\end{align*}
Therefore, 
\begin{align*}
	&	\int_{\PGL_2}\Phi_1\left( \begin{pmatrix}
	 & -1 \\ 1 &
\end{pmatrix} \begin{pmatrix}
	\xi & \\ & 1
\end{pmatrix}  g \right)\Phi_2(g)\ud g \\
&= \int_{ \RF\times\RF^{\times}\times\PGL_2(\Fo_{\RF})   |n|>|a|}\psi^{-1}\left(  n+\frac{1}{n\xi}  \right)\Phi_1 \left( \begin{pmatrix} -\frac{a}{\xi n^2} & \\ & 1 \end{pmatrix}\begin{pmatrix}
	& 1 \\ 1 & -an^{-1}
\end{pmatrix}   \begin{pmatrix}
	& -1 \\ 1 & 
\end{pmatrix}  k   \right)\\
&\cdot \Phi_2\left(  \begin{pmatrix}
	a & \\ & 1
\end{pmatrix} k \right) \ud n\frac{\ud a}{|a|^2}\ud k.
\end{align*}
For fixed $\xi$, when $\Re(s)>0$,
\begin{align*}
	&\int_{ \RF\times\RF^{\times}\times\PGL_2(\Fo_{\RF})   |n|>|a|} \left|\Phi_1 \left( \begin{pmatrix} -\frac{a}{\xi n^2} & \\ & 1 \end{pmatrix}\begin{pmatrix}
	& 1 \\ 1 & -an^{-1}
\end{pmatrix}   \begin{pmatrix}
	& -1 \\ 1 & 
\end{pmatrix}  k   \right) \Phi_2\left(  \begin{pmatrix}
	a & \\ & 1
\end{pmatrix} k \right)\right| \ud n\frac{\ud a}{|a|^2}\ud k\\
&\ll \int_{|a|\leq 1  }  \left\{  \int_{a<|n|<|a|^{\frac{1}{2}}} 1\ud n +\int_{|n|\geq |a|^{\frac{1}{2}}}  |n|^{-2} \ud n \right\}  |a|^{s+2}    |a|^{-1} \frac{\ud a}{|a|}\\
&\ll \int_{|a|\leq 1}  (1+|a|^{-1})    |a|^{s+1} \frac{\ud a}{|a|} <\infty.
\end{align*}
when $\Re(s)>0$.

When $|\xi|\rightarrow 0$, make a change of variable that $n=\frac{1}{u}$, then the integral over $u$ is non-zero if and only if $|u|^2=|\xi|$, in which case the integral, up to a constant, is
\begin{align*}
	\int_{|u|^2=|\xi|}\psi^{-1}\left( \frac{u}{\xi }+u^{-1}  \right)\ud u\cdot |\xi|^{-1}.
\end{align*}

When $\xi\rightarrow\infty$, the expected asymptotic behavior follows from a similar argument using the dominant convergence theorem as above.

\end{proof}

Before proceeding to the analysis of the transfer operator between $\CS^-_{L\left(\st,s\frac{1}{2}\right)}(\RA\backslash\PGL_2/\RN,\psi)$ and $\CS^-_{L\left(\st,s+\frac{1}{2}  \right)L\left(\st, \frac{1}{2}\right)}(\RN,\psi\backslash \PGL_2/\RN,\psi)$, we have the following lemma, which is analogous to \cite[Proposition 2.10]{Sak13}.

\begin{lem}
	Let $\Psi_1,\Psi_2\in\CC_c^{\infty}(\RF)$. Write 
	\[
	\CO_{\xi}^1(\Psi_1\otimes\Psi_2):=\int_{\RF^{\times}}|a\xi|^{-1}\Psi_1(a^{-1}\xi^{-1})\Psi_2(a)\frac{\ud a}{|a|},
	\]
	and
	\[
	\CO_{\xi}^2(\Psi_1,\Psi_2):=\int_{\RF^{\times}}\Psi_1(\xi a)\Psi_2(a)\frac{\ud a}{|a|}.
	\]
	Write $\CS(\BA^2/(\BG_m^{\nabla},|\cdot|))$ for the space of measures on $\RF$ of the form $\CO_{\xi}^1(\Psi_1\otimes\Psi_2)\ud \xi$ and $\CS(\BA^2/\BG_m^{\Delta})$ for the space of measures on $\RF$ of the form $\CO_{\xi}^2(\Psi_1\otimes\Psi_2)\ud\xi$. Then we have the following commutative diagram
	\[
	\begin{tikzcd}
\CC_c^{\infty}(\RF)\otimes\CC_c^{\infty}(\RF) \arrow[r, "\BF_{\psi}\otimes\Id "] \arrow[d, "\CO^2_{\cdot}(\cdot)"]
&   \CC_c^{\infty}(\RF)\otimes\CC_c^{\infty}(\RF)    \arrow[d, "\CO_{\cdot}^1(\cdot)"] \\
   \CS(\BA^2/\BG_m^{\Delta})   \arrow[r, "\CG"]
& \CS(\BA^2/(\BG_m^{\nabla},|\cdot|))
\end{tikzcd},
	\]
	where \[
	\CG=|\cdot|^{-1}\circ\iota\circ\BF_{\psi}
	\]
\end{lem}

\begin{proof}
	Let $h(\xi)\in\CC_c(\RF)$ be a test function on $\RF$, then we have
	\begin{align*}
		&\int_{\RF}\CO_{\xi}^1(\BF_{\psi}(\Psi_1)\otimes\Psi_2)\overline{h(\xi)}\ud\xi \\
		&=\int_{\RF} \int_{\RF^{\times}}|a\xi|^{-1}\BF_{\psi}(\Psi_1)(a^{-1}\xi^{-1})\Psi_2(a) \overline{h(\xi)} \frac{\ud a}{|a|}\ud \xi\\
		&=\int_{\RF^{\times}}\int_{\RF}|\xi|\BF_{\psi}(\Psi_1)(\xi)\Psi_2(a)\overline{h(a^{-1}\xi^{-1})}\frac{\ud a}{|a|} \frac{1}{|a\xi^2|}\ud \xi.
	\end{align*}
	According to the Plancherel formula for the Fourier transform,
	\begin{align*}
		\int_{\RF}\BF_{\psi}(\Psi_1)(\xi) \overline{h(a^{-1}\xi^{-1})}|\xi|^{-1}\ud\xi=\int_{\RF}\Psi_1(\xi)\overline{  \BF_{\psi^{-1}}(h(a^{-1}(\cdot)^{-1})|\cdot|^{-1}) (\xi)  }\ud \xi.
	\end{align*}
	Note that
	\begin{align*}
		\BF_{\psi^{-1}}(h(a^{-1}(\cdot)^{-1})|\cdot|^{-1}) (\xi)&=\int_{\RF}h(a^{-1}x^{-1})|x|^{-1}\psi(x\xi )\ud x\\
		&=\int_{\RF}h(x^{-1})|x|^{-1}\psi\left(x\xi a^{-1}  \right)\ud x\\
		&=\BF_{\psi^{-1}}\circ|\cdot|^{-1}\circ\iota (h)(\xi a^{-1}),
	\end{align*}
	hence
	\begin{align*}
			&\int_{\RF}\CO_{\xi}^1(\BF_{\psi}(\Psi_1)\otimes\Psi_2)\overline{h(\xi)}\ud\xi\\
			&=\int_{\RF^{\times}}\int_{\RF}\Psi_1(\xi)\Psi_2(a) \overline{ \BF_{\psi^{-1}}\circ|\cdot|^{-1}\circ\iota (h)(\xi a^{-1})   } \frac{\ud a}{|a|^2}\ud \xi \\
			&=\int_{\RF^{\times}}\Psi_1(\xi a)\Psi_2(a)\overline{\BF_{\psi^{-1}}\circ|\cdot|^{-1}\circ\iota (h)(\xi)}\frac{\ud a}{|a|}\ud \xi \\
			&=\int_{\RF}\CO_{\xi}^2(\Psi_1\otimes\Psi_2)\overline{\BF_{\psi^{-1}}\circ|\cdot|^{-1}\circ\iota (h)(\xi)}\ud \xi\\
			&=\int_{\RF}  |\cdot|^{-1}\circ\iota\circ \BF_{\psi}(\CO_{\cdot}^2(\Psi_1\otimes\Psi_2))(\xi)\overline{h(\xi)}\ud\xi.
	\end{align*}
\end{proof}

\begin{prp}
	Let $s\in\BC$ such that $\Re(s)>0$, then we have a commutative diagram
	\[
	\begin{tikzcd}
\CC_c^{\infty}(\RA \backslash\PGL_2)\otimes\CC_c^{\infty}(\CO(-2),\BL_{s}^{\psi^{-1}}) \arrow[r, "\RU_{\psi}\otimes\Id "] \arrow[d, ""]
&     \CC_c^{\infty}(\CO(-2),  \BL^{\psi}_{0} )\otimes\CC_c^{\infty}(\CO(-2),\BL_{s }^{\psi^{-1}}  )     \arrow[d, ""] \\
\CS_{L\left(\st,s+\frac{1}{2}\right)}( \RA_1\backslash\PGL_2/(\RN,\psi))\arrow[r, "\overline{\RU}"]
& \CS^-_{L\left(\st,\frac{1}{2}  \right)L\left(\st, s+\frac{1}{2}\right)}(\RN,\psi\backslash \PGL_2/\RN,\psi)
\end{tikzcd},
	\]
with $\overline{\RU}$ given by the following formula on $\BA^1$:
	\begin{align*}
	|\cdot|^{-1}\circ\iota \circ \BF_{\psi}.
	\end{align*}
\end{prp}

\begin{proof}
	Let $\Psi_1\in\CC_c^{\infty}(\RA\backslash\PGL_2)$ and $\Psi_2\in\CC_c^{\infty}(\CO(-2),\BL_s^{\psi^{-1}})$, for $a\in\RF^{\times}$ and $g\in\PGL_2$, write
	\[
	\widehat{\Psi_1}(a,g):=\int_{\RF}\Psi_1\left(  \begin{pmatrix}
		1 & n \\ & 1
	\end{pmatrix} g  \right)\psi^{-1}(an)\ud n
	\]
	and $\Psi_2(a,g):=\Psi_2\left(\begin{pmatrix}
		a & \\ & 1
	\end{pmatrix} g \right)$. Then for $\xi\in\RF^{\times}$, we have
	\begin{align*}
		\CO_{\xi}(\RU_{\psi}(\Psi_1)\otimes\Psi_2)&=\int_{\PGL_2}\RU_{\psi}(\Psi_1)\left(\Bw \begin{pmatrix}
			\xi & \\ & 1
		\end{pmatrix}  g\right)\Psi_2(g)\ud g \\
		&=\int_{\RA\backslash\PGL_2}\int_{\RA}\RU_{\psi}(\Psi_1)\left(\Bw \begin{pmatrix}
			a\xi & \\ & 1
		\end{pmatrix} g\right)\Psi_2\left( \begin{pmatrix}
			a & \\ & 1
		\end{pmatrix} g \right)\ud g.
	\end{align*}
	Note that 
	\begin{align*}
		\RU_{\psi}(\Psi_1)\left(\Bw \begin{pmatrix}
			a\xi & \\ & 1
		\end{pmatrix} g\right)=|a\xi|^{-1}\widehat{\Psi_1}(a^{-1}\xi^{-1},g),
	\end{align*}
	hence
	\begin{align*}
		\CO_{\xi}(\RU_{\psi}(\Psi_1),\Psi_2)=\int_{\RA\backslash\PGL_2}\int_{\RA}|a\xi|^{-1}\widehat{\Psi_1}(a^{-1}\xi^{-1},g)\Psi_2(a,g)\ud g.
	\end{align*}
	According to the above lemma,
	\begin{align*}
		\int_{\RA}|a\xi|^{-1}\widehat{\Psi_1}(a^{-1}\xi^{-1},g)\Psi_2(g)=|\cdot|^{-1}\circ\iota \circ \BF_{\psi} (f)(\xi,g)
	\end{align*}
	where 
	\begin{align*}
		f(\xi,g)=\int_{\RA}\Psi_1\left(   \begin{pmatrix}
			1 & \xi \\ & 1
		\end{pmatrix} \Bw \begin{pmatrix}
			a & \\ & 1
		\end{pmatrix} g\right)\Psi_2(g)\ud a,
	\end{align*}
	and hence
	\[
	\CO_{\xi}(\RU_{\psi}(\Psi_1),\Psi_2)=|\cdot|^{-1}\circ\iota \circ\BF_{\psi}(\CO_{\cdot}(\Psi_1\otimes\Psi_2)).
	\]
\end{proof}

\subsection{ Comparison between $\RA\backslash\PGL_2/(\RN,\psi)$ and $(\RN,\psi)\backslash\Mp_2/(\RN,\psi)$ }\label{ssec_jac}

The results of this section are due to \cite{Jac87}. We present them here for the reader's convenience and to clarify the notation used throughout the paper. 

Let us first recall the isomorphism \[\widetilde{\lambda}:\CH(\PGL_2,\PGL_2(\Fo_{\RF}))\cong\CH(\Mp_2,\RK^{\prime})\]
as fixed in \cite[Section 2]{Jac87}. Let $s\in\BC$, and consider the genuine character of the abelian subgroup of $\Mp_2$
\begin{align*}
	 \left( \begin{pmatrix}
		a & \\ & a^{-1}
	\end{pmatrix}  ,\epsilon \right) \mapsto\gamma(a,\psi)^{-1}|a|^s \epsilon ,
\end{align*}
and consider the induced representation $\widetilde{\RI}(\chi(s))$ of $\Mp_2$ by left shift on the space of functions $\Phi$ such that 
\[
\Phi\left( g   \left( \begin{pmatrix}
		a & \\ & a^{-1}
	\end{pmatrix}  ,\epsilon \right)  \right)=|a|^{-1}\gamma(a,\psi)^{-1}|a|^s\epsilon\Phi(g).
\]
Then the isomorphism $\widetilde{\lambda}$ is such that for any $h\in\CH(\PGL_2,\PGL_2(\Fo_{\RF}))$, the action of $h$ on the unramified vector in $\RI(\chi(s))$ and $\widetilde{\lambda}(h)$ on the unramified vector in $\widetilde{\RI}(\chi(s))$ are the same. Moreover, for any $m\geq 0$, let $h_m$ be the characteristic function on $\displaystyle{\PGL_2(\Fo_{\RF})\begin{pmatrix}
	\varpi^m & \\ & 1
\end{pmatrix} \PGL_2(\Fo_{\RF}) }$, then $\widetilde{\lambda}(h_m)(g)=|a|^{\frac{1}{2}}\gamma(a,\psi)\epsilon$ if 
\[
g=\left(  k_1 \begin{pmatrix}
	a & \\ & a^{-1}
\end{pmatrix} k_2 ,\epsilon \right)
\]
with the valuation of $a$ is $m$; $\widetilde{\lambda}(h_m)(g)=0$ is $g$ is not of this form.

Let $\CF(\RN,\psi\backslash\Mp_2/\RN,\psi)$ be the space of functions on $\RF^{\times}$ of the form
\[
a\mapsto \int_{\RF\times\RF}\Phi\left( \begin{pmatrix}
	1 & x  \\ & 1
\end{pmatrix} \Bw \begin{pmatrix}
	a & \\ & a^{-1}
\end{pmatrix} \begin{pmatrix}
	1 & y \\ &  1
\end{pmatrix}\right)\psi^{-1}(x+y)\ud x\ud y=\CO_{\xi}(\Phi)
\]
where $\Phi$ is a smooth function of compact support on $\Mp_2$. Similarly let $\CF(\RA\backslash\PGL_2/\RN,\psi)$ be the space of orbital integrals
\[
\RF\ni y \mapsto \int_{\RF^{\times}\times\RF}\Phi\left(  \begin{pmatrix}
	a & \\ & 1
\end{pmatrix} \begin{pmatrix}
	1 & y \\ & 1
\end{pmatrix} \Bw  \begin{pmatrix}
	1 & x \\ & 1
\end{pmatrix}  \right)\psi^{-1}(x)\frac{\ud a}{|a|} \ud x.
\]

\begin{prp}\label{prp_mpvspgl}
	For any $h\in\CH(\PGL_2,\PGL_2(\Fo_{\RF}))$, we have
	\[
	\CO_{a }(\widetilde{\lambda}(h) ) =    |a |^{-1/2} \gamma(a,\psi)\psi\left(\frac{2}{a}\right) \CO_{ \frac{a }{4}    }(h)=:\BJ(\CO_{\cdot}(h))(a)
	\]
\end{prp}

\begin{proof}
	See \cite[Section 8]{Jac87}.
\end{proof}

\begin{rmk}
	In terms of measures, we should multiply by $|a|^{\frac{1}{2}}$ instead of $|a|^{-\frac{1}{2}}$. And as a result, the diagram \ref{diag_b} is a commutative diagram.
\end{rmk}

\subsection{Comparison between $\SO_{2n}\backslash\SO_{2n+1}/SO_{2n}$ and $(\RN,\psi)\backslash\Mp_2/(\RN,\psi)$ }

\begin{lem}\label{lem_opBX_b}
Let \[
	\RX^{\circ}=\{v\in\RX\mid \BB_{\BV}(v,e_1)\neq 0\}.
	\]
Then $v_0\in\RX^{\circ}$ and $\RX^{\circ}$ is an open $\RB$-orbit in $\RX$.	
\end{lem}

\begin{proof}
	Same as Lemma \ref{lem_open_b}.
\end{proof}

\begin{cor}
	$\RP(\RX)$ is the maximal proper parabolic subgroup fixing the isotropic flag
	\[
	\langle e_1\rangle .
	\]

Then under the above basis $\{ e_1,\cdots, e_n,f_n,\cdots, f_1\}$, the Levi $\RL(\RX)$ of $\RP(\RX)$ is 
\[
\RL(\RX)=\left\{  \begin{pmatrix}
	t & & \\ & g & \\ & & t^{-1}
\end{pmatrix}  \mid t\in\BG_m,g\in\SO(v_0^{\perp})    \right\}\cong\BG_m\times\SO(v_0^{\perp}),
\]
and hence
\[
\delta_{(\RX)}^{\frac{1}{2}} \left(  \begin{pmatrix}
	t_1 & & & & & \\ & \rotatebox{45}{$\vdots$} & & & & \\ & & t_n & & & \\ & & & t_n^{-1} & & \\ & & & & \rotatebox{45}{$\vdots$} & \\ & & & & & t_1^{-1} 
\end{pmatrix}  \right)=|t_2|^{n-\frac{3}{2} }|t_3|^{n-\frac{5}{2}}\cdots |t_n|^{\frac{1}{2}}.
\]
Moreover, the quotient map $\RA\rightarrow\RA_{\RX}\cong\BG_m$ is given by 
\[
\begin{pmatrix}
	t_1 & & & & & \\ & \rotatebox{45}{$\vdots$} & & & & \\ & & t_n & & & \\ & & & t_n^{-1} & & \\ & & & & \rotatebox{45}{$\vdots$} & \\ & & & & & t_1^{-1} 
\end{pmatrix} \mapsto t_1.
\]
Then after the Satake transform, $\BC[\RA^{\vee}]^{\BW_{\RG}}\rightarrow\BC[\RA_{\RX}^{\vee}]^{\BW_{\RX}}$ is given by
\begin{align}\label{hkdualb}
	\BC[x_1^{\pm},\cdots,x_n^{\pm}]^{\BW_{\SO_{2n}}}\rightarrow\BC[x^{\pm}]^{\BZ/2\BZ}:x_1\mapsto x ,\;x_j\mapsto q^{j-n-\frac{1}{2}}\;\mathrm{for}\;2\leq j\leq n.
\end{align}
\end{cor}

Now consider the nilpotent cone $\Sigma:=\{ v\in\BV\mid \BB_{\BV}(v,v)=0  \}$.
\begin{lem}\label{lem_nil_b}
	\[
	\Sigma^{\circ}:=\left\{ g\begin{pmatrix}
		1 \\ 0 \\ \vdots \\ 0
	\end{pmatrix}\mid g\in\SO(\BV) \right\}
	\]
	is an open dense $\SO(\BV)$-orbit of $\Sigma$.
\end{lem}

\begin{proof}
	Same as Lemma \ref{lem_nil_d}.
\end{proof}

\begin{lem}\label{lem_nilipvan_b}
	Let $\Phi\in\CF(\BY)^{\RK^{\prime}}$, and $\omega_{\psi}(h)\Phi|_{\Sigma}=0$ for all $h\in\Mp_2$, then $\Phi=0$.
\end{lem}

\begin{proof}
	The same argument as in Lemma \ref{lem_nilipvan_d}, noting that cuspidal genuine representations of $\Mp_2$ are not spherical according to \cite[Proposition 17]{FK86}.
\end{proof}

Now let $\sigma\in\BC$ be such that $\Re(\sigma)>1$. Consider the following intertwining operators:
\[
\RZ_{\sigma}:\CF(\BY)\rightarrow \CC^{\infty}(\SO_{2n}\times\Mp_2) :\Phi \mapsto (g,h)\mapsto \int_{\RF} \omega_{\psi}(g,h)^{-1}\Phi\left( \begin{pmatrix}
	a \\ 0 \\ \vdots \\ 0
\end{pmatrix} \right)|a|^{\sigma} \frac{\ud a}{|a|}.
\]
which is absolutely convergent.

\begin{prp}\label{prp_cpthk_b}
	When $\Re(\sigma)>1$,
	\begin{itemize}
		\item [(1)] Let $\Fs=(s_1,s_2,\cdots,s_n)\in X^*(\RA)\otimes_{\BZ}\BC=\BC^n$, write 
\[
\chi(\Fs): \RT\rightarrow\BC:  \begin{pmatrix}
	t_1 & & & & & \\ & \rotatebox{45}{$\vdots$} & & & & \\ & & t_n & & & \\ & & & t_n^{-1} & & \\ & & & & \rotatebox{45}{$\vdots$} & \\ & & & & & t_1^{-1} 
\end{pmatrix} \mapsto \prod_{i=1}^n|t_i|^{s_i}
\]
and $\RI(\chi(\Fs))$ the correponding normalized induced representation of $\SO_{2n}$. Then $\RZ_{\sigma}$ factors through $\RI(\chi(n-\frac{1}{2}-\sigma,n-\frac{3}{2},n-\frac{5}{2},\cdots,\frac{1}{2}) )\otimes\widetilde{\RI}(\chi(\sigma-n+\frac{1}{2}))$ as representations of $\SO_{2n}\times\Mp_2$.

\item [(2)] Write $\RK$ for the maximal compact subgroup of $\SO(\BV)$ fixing the standard lattice $\Fo_{\RF}^{2n}$ under the ordered basis $\{e_1,\cdots,e_n,w,f_n,\cdots,f_1\}$. Let \[\lambda_{\RX}:\CH(\SO(\BV),\RK)\rightarrow\CH(\Mp_2,\RK^{\prime})\] be the morphism of the Hecke algebra corresponding to the above. Let $\Phi_0$ be the characteristic function of $\BV(\Fo_{\RF})$, then we have 
	\[
\omega_{\psi}(f^{\vee})\Phi_0=\omega_{\psi}(\lambda_{\RX}(f))\Phi_0.
\]
for all $f\in\CH(\SO(\BV),\RK)$. 
	\end{itemize}
\end{prp}

\begin{proof}
	$(1)$ is simply a matter of computing the effect of $\RA\times\RT$ on $\RZ_{\sigma}$, noting that
	\[
	\gamma(t,\psi_{\frac{1}{2}})^{-(2n+1)}(t,(-1)^n 2)=\gamma(t,\psi)^{-1}.
	\]
	
	$(2)$ is totally the same as in the previous sections.
 \end{proof}

\subsection{Comparison between $(\RN,\psi)\backslash\PGL_2/(\RN,\psi)$ and $\SO_{2n}\backslash\SO_{2n+1}/\SO_{2n}$}

\begin{lem}\label{lem_upb}

	Let $h_{\Sym^m}$ be the Satake inverse of the symmetric $m$-th power on $\CH(\PGL_2,\PGL_2(\Fo_{\RF}))$, then 
	\begin{align*}
			\Omega(\Phi_0)(g)= \frac{1}{\zeta_{\RF}(n)}  \sum_{m=0}^{\infty}  q^{- m(n-\frac{1}{2})} \int_{\RN}\widetilde{\lambda}(h_{\Sym^m})(ng)\psi^{-1}(n)\ud n.
	\end{align*}
\end{lem}

\begin{proof}
	A first observation is that 
	\begin{align*}
		\sum_{m=0}^{\infty}q^{-m(n-\frac{1}{2})}h_{\Sym^m}&=\sum_{m=0}^{\infty}q^{-mn}(h_m+h_{m-2}+\cdots)\\
		&=\sum_{m=0}^{\infty} \frac{q^{-mn}}{1-q^{-2n}}h_m.
	\end{align*}
According to \cite[Lemma on Page 143]{Jac87}, we have
\begin{align*}
	\int_{\RN}\widetilde{\lambda}(h_m)\left(n  \begin{pmatrix}
		a & \\ & a^{-1}
	\end{pmatrix}  \right)\psi^{-1}(n)\ud n=|a|^{\frac{1}{2}}\gamma(a,\psi)1_{\varpi^m\Fo_{\RF}^{\times}\cup \varpi^{m-1}\Fo_{\RF}^{\times}}(a),
\end{align*}
	we have when $|a|\leq 1$,
	\begin{align*}
	&	\sum_{m=0}^{\infty}\frac{q^{-mn}}{1-q^{-2n}}\int_{\RN}\widetilde{\lambda}(h_m)\left(n  \begin{pmatrix}
		a & \\ & a^{-1}
	\end{pmatrix}  \right)\psi^{-1}(n)\ud n\\
	&=\sum_{m=0}^{\infty}\frac{|a|^{\frac{2n+1}{2}} }{1-q^{-2n}}(1+q^{-n})\gamma(a,\psi) ,
	\end{align*}
	and $0$ otherwise.

Also note that
\begin{align*}
		\omega_{\psi}\left(  \begin{pmatrix}
			a & \\ & a^{-1}
		\end{pmatrix}  \right)\Phi_0(v_0) =\left\{ \begin{array}{rcl}
\gamma(a,\psi)|a|^{\frac{2n+1}{2}}& \mathrm{if}\; |a|\leq 1;
\\ 0 & \mathrm{otherwise}.\\
\end{array}\right.
	\end{align*}
Then the Lemma follows.
\end{proof}

\begin{rmk}
	Formally, 
	\[
	\Omega(\Phi_0)= \frac{1}{\zeta_{\RF}(n)} \int_{\RN}^* \widetilde{\lambda}\left(\Phi_{L\left(\st,n-\frac{1}{2}\right)}\right)(ng)\psi^{-1}(n)\ud n.
	\]
\end{rmk}

\begin{thm}
	Conjecture \ref{cnj_flhk} is true for $\RX$ of type $\RB_n$.
\end{thm}

\begin{proof}
	Let us start with the basic function $1_{(\RN,\psi)\backslash\PGL_2(\Fo_{\RF})}\otimes\Phi_{L\left(\st,n-\frac{1}{2}\right)L\left(\st,\frac{1}{2}\right)}$, and for any $h\in\CH(\SO(\BV),\RK)$, we have
	\[
	\CO_{\xi}\left(1_{(\RN,\psi)\backslash\PGL_2(\Fo_{\RF})} \star \lambda_{\RX}(h), \Phi_{L\left(\st,n-\frac{1}{2}\right)L\left(\st,\frac{1}{2}\right)}\right)=\CO_{\xi}\left(\Phi_{L\left(\st,\frac{1}{2}\right)}\star \lambda_{\RX} (h),\Phi_{L\left(\st,n-\frac{1}{2}\right)}\right )
	\]
	according to the commutativity of the Hecke algebra. Then, according to the unfolding,
	\[
	\CO_{\xi}\left(\Phi_{L\left(\st,\frac{1}{2}\right)}\star \lambda_{\RX}(h),\Phi_{L\left(\st,n-\frac{1}{2}\right)}\right )=\overline{\RU}^{-1}\left(     \CO_{\xi}\left( 1_{\RA\backslash\PGL_2(\Fo_{\RF})}\star \lambda_{\RX}(h),  \Phi_{L\left(\st,n-\frac{1}{2}\right)}\right )\right) .
	\]
	According to the above lemma,
	\begin{align*}
		& \CO_{\xi}\left( 1_{\RA\backslash\PGL_2(\Fo_{\RF})}\star \lambda_{\RX}(h),  \Phi_{L\left(\st,n-\frac{1}{2}\right)}\right )\\
		 &=\sum_{m=0}^{\infty}q ^{-mn} \CO_{\xi}\left(  1_{\RA\backslash\PGL_2(\Fo_{\RF})}\star \Sat^{-1}(\Sym^m) \star \lambda_{\RX}(h),1_{(\RN,\psi)\backslash\PGL_2(\Fo_{\RF})}   \right).
	\end{align*}
	Then by Proposition \ref{prp_mpvspgl}, 
	\[
	\CO_{\xi}\left(  1_{\RA\backslash\PGL_2(\Fo_{\RF})}\star \Sat^{-1}(\Sym^m) \star \lambda_{\RX} (h),1_{(\RN,\psi)\backslash\PGL_2(\Fo_{\RF})}   \right)=\BJ^{-1}\left(\CO_{\xi}\left(  \widetilde{\lambda}(  \Sat^{-1}(\Sym^m) \star \lambda_{\RX}( h) )  \right)\right).
	\]
	Since the orbital integral is absolutely convergent, combined with Lemma \ref{lem_upb}, 
	\begin{align*}
		&\sum_{m=0}^{\infty} q^{-mn} \BJ\left(\CO_{\xi}\left(  1_{\RA\backslash\PGL_2(\Fo_{\RF})}\star \Sat^{-1}(\Sym^m) \star  \lambda_{\RX}(h),1_{(\RN,\psi)\backslash\PGL_2(\Fo_{\RF})}   \right) \right)\\
		&=\zeta_{\RF}(2)\zeta_{\RF}(n)\int_{\RN}^* \omega_{\psi}\left(\Bw \begin{pmatrix}
			\zeta & \\ & 1
		\end{pmatrix} n  \right) \omega_{\psi}( \widetilde{\lambda}( \lambda_{\RX}(h))) \Phi_0(v_0) \psi^{-1}(n)\ud n,
	\end{align*}
		which is
	\[
	\frac{\zeta_{\RF}(2)\zeta_{\RF}(n)^2}{\zeta_{\RF}(2n)}\CT^{-1}( \BL_{\RX}^{\circ}\star h )
	\]
due to Proposition \ref{prp_cpthk_b}.

\end{proof}

\bibliographystyle{alpha}
	\bibliography{references}
\end{document}